\documentclass[11pt]{amsart}

\usepackage[colorlinks=true,citecolor=black!60!green,linkcolor=black!60!red,filecolor=black!60!cyan,urlcolor=black!60!magenta]{hyperref}
\usepackage[text={6.1in,8.5in},centering]{geometry}
\usepackage{amssymb,amsmath,amsthm,combelow,mathtools,extpfeil,wrapfig,graphicx}
\usepackage{caption}
\usepackage{subcaption}
\input{insbox}
\usepackage[all,cmtip]{xy}
\usepackage{tikz}
\usepackage[shortlabels]{enumitem}

\newtheorem{thm}{Theorem}
\newtheorem*{thm*}{Theorem}
\newtheorem*{lem*}{Lemma}
\newtheorem{thmA}{Theorem}

\newtheorem{thmCprime}{Theorem}

\newtheorem{lem}[thm]{Lemma}
\newtheorem{prop}[thm]{Proposition}
\newtheorem*{prop*}{Proposition}
\newtheorem{cor}[thm]{Corollary}

\theoremstyle{definition}
\newtheorem*{defn}{Definition}
\newtheorem{rmk}[thm]{Remark}
\newtheorem{question}[thm]{Question}
\newtheorem{ex}[thm]{Example}

\numberwithin{thm}{section}

\DeclareMathOperator{\id}{id}

\DeclareMathOperator{\vol}{vol}

\DeclareMathOperator{\Dil}{Dil}

\DeclareMathOperator{\Lip}{Lip}
\DeclareMathOperator{\Hom}{Hom}
\DeclareMathOperator{\Aut}{Aut}
\DeclareMathOperator{\supp}{supp}
\newcommand{\ph}{\varphi}
\newcommand{\epsi}{\varepsilon}
\newcommand{\arr}{r}

\newcommand{\FF}{\mathbb{F}}
\newcommand{\CC}{\mathbb{C}}
\newcommand{\eps}{\epsilon}

\newcommand{\QQ}{\mathbb{Q}}
\newcommand{\ZZ}{\mathbb{Z}}
\newcommand{\RR}{\mathbb{R}}
\DeclareMathOperator{\Ann}{Ann}
\DeclareMathOperator{\topp}{top}
\DeclareMathOperator{\Prim}{Prim}

\begin{document}
\title{Degrees of maps and multiscale geometry}
\author[A.~Berdnikov]{Aleksandr Berdnikov}
\address[A.~Berdnikov]{Institute for Advanced Study, Princeton, NJ, United States}
\email{beerdoss@mail.ru}
\author[L.~Guth]{Larry Guth}
\address[L.~Guth]{Department of Mathematics, MIT, Cambridge, MA, United States}
\email{lguth@math.mit.edu}
\author[F.~Manin]{Fedor Manin}
\address[F.~Manin]{Department of Mathematics, UCSB, Santa Barbara, CA, United States}
\email{manin@math.ucsb.edu}

\begin{abstract}
We study the degree of an $L$-Lipschitz map between Riemannian manifolds, proving new upper bounds and constructing new examples.  For instance, if $X_k$ is the connected sum of $k$ copies of $\mathbb CP^2$ for $k \ge 4$, then we prove that the maximum degree of an $L$-Lipschitz self-map of $X_k$ is between $C_1 L^4 (\log L)^{-4}$ and $C_2 L^4 (\log L)^{-1/2}$.  More generally, we divide simply connected manifolds into three topological types with three different behaviors.  Each type is defined by purely topological criteria.  For scalable simply connected $n$-manifolds, the maximal degree is $\sim L^n$.  For formal but non-scalable simply connected $n$-manifolds, the maximal degree grows roughly like $L^n (\log L)^{-\theta(1)}$.  And for non-formal simply connected $n$-manifolds, the maximal degree is bounded by $L^\alpha$ for some $\alpha < n$.
\end{abstract}
\maketitle

\section{Introduction}

\subsection{Background}
Given an oriented Riemannian manifold $M$, how does the Lipschitz constant of a map $M \to M$ control its degree?  In all cases, if $M$ is an $n$-manifold, an $L$-Lipschitz map $M \to M$ multiplies $n$-dimensional volumes by at most $L^n$, and so its degree is at most $L^n$.  In \cite[Ch.~2]{GrMS}, Gromov studied the extent to which this estimate is sharp.  For example, he showed that if $M$ admits a sequence of self-maps $f_k$ with \[\deg(f_k) \ge (1 - o(1)) \Lip(f_k)^n,\] then $M$ must be flat \cite[2.32]{GrMS}.  He also asked: for what $M$ are there $f_k$ with unbounded degree such that the ratio $\Lip(f_k)^n/\deg(f_k)$ is bounded \cite[2.40(c)]{GrMS}?  The answer to this modified question only depends on the topology of $M$.  Gromov constructed such maps when $M$ is a sphere or a product of spheres.  He singled out $(S^2 \times S^2) \# (S^2 \times S^2)$ as a case in which he did not know whether such maps exist.

We now know that the answer for connected sums of copies of $S^2 \times S^2$ or of $\CC P^2$ is rather subtle.  (The behavior is similar for both families.)  Consider the family of manifolds $X_k = \#_k \mathbb CP^2$.  Volume considerations show that an $L$-Lipschitz self-map of any $4$-manifold has degree at most $L^4$.  It's not difficult to construct an $L$-Lipschitz self-map of $\mathbb CP^2$ with degree $\sim L^4$.  When $k=2$ or $3$, then \cite{scal} shows that there are also $L$-Lipschitz self-maps of $X_k$ with degree $\sim L^4$.  But when $k \ge 4$, \cite{scal} shows that an $L$-Lipschitz self-map of $X_k$ has degree $o(L^4)$.  Before this paper, the most efficient known maps had degree $\sim L^3$.  

One of our goals in this paper is to give sharper quantitative estimates for the case $k \ge 4$.  We will show that the maximal degree $p$ lies in the range
\[L^4(\log L)^{-4} \lesssim p \lesssim L^4(\log L)^{-1/2}.\]

This phase transition between $k=3$ and $k=4$ is an example of a broader phenomenon.  Our second goal in the paper is to develop the general theory of this phenomenon.  


For a given $M$, the maximally efficient relationship $\Lip f \sim (\deg f)^{1/n}$ may not be achievable for several reasons.  For example, $M$ may be \emph{inflexible}, meaning that it does not have self-maps of degree $>1$.  (Examples of inflexible simply connected manifolds are given in \cite{ArLu,CLoh,CV,Am}.)  Or it may be the case that any self-map of $M$ of degree $D$ multiplies some $k$-dimensional homology class by a factor greater than $D^{k/n}$, giving a stronger bound on the Lipschitz constant.

A compact manifold $M$ is \emph{formal} if it has a self-map $M \to M$ which, for some $p$, induces multiplication by $p^k$ on $H_k(M;\mathbb R)$, for every $k \geq 1$.  This notion, first defined by Sullivan and coauthors in terms of rational homotopy theory, has played a role in many other geometric applications, starting with \cite{DGMS}.  If $M$ is a formal $n$-manifold, then obstructions to obtaining an $L$-Lipschitz map $M \to M$ of degree $L^n$ cannot come from measuring volumes of cycles.  However, in \cite{scal} it was shown that more subtle obstructions may exist.  This motivates the definition of a \emph{scalable} manifold to be one which has $O(L)$-Lipschitz self-maps of degree $L^n$.  The paper \cite{scal} shows that scalability is equivalent to several other conditions; most importantly, a manifold $M$ (perhaps with boundary) is scalable if and only if there is a ring homomorphism $H^*(M;\RR) \to \Omega^*(M)$ which realizes cohomology classes as differential forms representing them.

\subsection{Main results}
For non-scalable formal spaces, \cite{scal} proves that any $L$-Lipschitz self-map has degree $o(L^n)$.  Before this paper, the examples that had been constructed had degree $O(L^{n-1})$.   In this paper, we gain a sharper quantitative understanding:

\begin{thmA} \label{main}
  Let $M$ be a formal, simply connected closed $n$-manifold which is not scalable.  Then the maximal degree $p$ of an $L$-Lipschitz map $M \to M$ satisfies
  \[L^n(\log L)^{-\beta(M)} \lesssim p \lesssim L^n(\log L)^{-\alpha(M)},\]
  where $\beta(M) \geq \alpha(M)>0$ are constants depending only on the real cohomology ring of $M$.
\end{thmA}

For example, in the case of $M=\#_k \CC P^2$, $\beta(M)=4$ and $\alpha(M)=1/2$.

The lower bound of Theorem \ref{main} generalizes to compact manifolds with boundary with a slightly more complicated statement (see Theorem \ref{self-maps}).

We obtain a similar result for sizes of nullhomotopies of $L$-Lipschitz maps to a non-scalable formal space:
\begin{thmA} \label{main:homotopies}
  Let $Y$ be a formal, simply connected compact Riemannian $n$-manifold (perhaps with boundary).  Then for any finite simplicial complex $X$, any nullhomotopic $L$-Lipschitz map $f:X \to Y$ is $O(L(\log L)^{n-2})$-Lipschitz nullhomotopic.
\end{thmA}
For scalable spaces, a linear bound was proved in \cite{scal}; thus this result is interesting mainly for non-scalable formal spaces.  In contrast, in non-formal spaces it is often impossible to do better than a bound of the form $L^\alpha$ for some $\alpha>1$.

One of the main theorems of \cite{scal} says that a manifold $Y$ is scalable if and only if there is a ring homomorphism from $H^*(Y; \RR)$ to $\Omega^*(Y)$ which takes each cohomology class to a differential form in that class.  Because $\Omega^*(Y)$ is infinite-dimensional, this condition is not so easy to check.  We verify the conjecture given in \cite{scal} which states that scalability is equivalent to a simple homological criterion:
\begin{thmA} \label{optconj}
  Let $Y$ be a formal, simply connected compact Riemannian $n$-manifold (perhaps with boundary).  Then $Y$ is scalable if and only if there is an injective ring homomorphism
  \[h:H^*(Y; \RR) \to \bigoplus_{i=1}^N \Lambda^*\RR^{n_i}\]
  for some integers $n_1,\ldots,n_N$.  In particular, if $Y$ is a closed manifold, then it is scalable if and only if there is an injective ring homomorphism $H^*(Y;\mathbb R) \to \Lambda^*\mathbb R^n$.
\end{thmA}
In particular, scalability is an invariant not only of rational, but of real homotopy type.
\begin{ex}
  If $M$ is an $(n-1)$-connected $2n$-manifold, then its real cohomology ring is completely described by the signature $(k,\ell)$ of the bilinear form
  \[\smile:H^n(M;\mathbb R) \times H^n(M;\mathbb R) \to H^{2n}(M;\mathbb R).\]
  Then $M$ is scalable if and only if $k$ and $\ell$ are both at most ${2n \choose n}/2$.
\end{ex}

Theorem \ref{optconj} is closely related to another idea studied by Gromov in \cite[2.41]{GrMS}.  For a closed $n$-manifold $M$, say a map $f:\mathbb R^n \to M$ has \emph{positive asymptotic degree} if
\[\limsup_{R \to \infty} \frac{\int_{B_R(0)} f^*d\vol_M}{R^n}=\delta>0.\]
Given an efficient self-map $M \to M$ of high degree, you can zoom in and find a map of positive asymptotic degree on a large ball.  If $M$ is formal, then the converse also holds:
\begin{thmCprime} \label{elliptic}
  Let $M$ be a formal, simply connected closed $n$-manifold.  Then there is a $1$-Lipschitz map $f:\mathbb R^n \to M$ of positive asymptotic degree if and only if $M$ is scalable.
\end{thmCprime}
\begin{rmk}
  Gromov refers to manifolds with this property as \emph{elliptic}, suggesting a connection with the notion of elliptic spaces from rational homotopy theory.  However, this notion is not closely connected to scalability.
\end{rmk}
\begin{question}
  Can a non-formal simply connected manifold be Gromov-elliptic?
\end{question}

Finally, we explore the behavior of non-formal manifolds:
\begin{thmA} \label{main:NF}
  Let $M$ be a closed simply connected $n$-manifold which is not formal.  Then either $M$ is inflexible (has no self-maps of degree $>1$) or the maximal degree of an $L$-Lipschitz map $M \to M$ is bounded by $L^\alpha$ for some real number $\alpha<n$.
\end{thmA}
To see how the latter situation arises, consider the simplest example of a non-formal simply connected manifold, given in \cite[p.~94]{FOT}.  This is the total space $M$ of a fiber bundle $S^3 \to M \to S^2 \times S^2$ obtained by pulling back the Hopf fibration $S^3 \to S^7 \to S^4$ along the degree $1$ map $S^2 \times S^2 \to S^4$.

A self-map of $M$ is determined by its action on $H^2(M) \cong \mathbb Z^2$.  This is because the generators of $H^5(M)$ can be obtained from the generators of $H^2(M)$ by taking Massey products (a higher cohomology operation) of order~3.  An $L$-Lipschitz self-map takes the generators of $H^5(M)$ to vectors of length $O(L^5)$, and therefore it takes the generators of $H^2(M)$ to vectors of length $O(L^{5/3})$.  This means the degree of such a map is $O(L^{20/3}) \prec L^7$.

Something similar happens for any non-formal space: an alternate definition of formality is that a formal space has no nontrivial higher-order rational cohomology operations.

\subsection{Proof ideas}
\label{discus}
The key idea behind Theorem \ref{main} is that efficient self-maps of a formal but non-scalable space must behave nontrivially on many scales.  We explain the intuition here.

In \cite{scal}, the $o(L^n)$ upper bound for the degree of an $L$-Lipschitz map $M \to M$ is obtained by looking at the induced pullbacks of differential forms representing cohomology classes of $M$ and taking flat limits.  To get the sharper upper bound of Theorem \ref{main}, we analyze the same pullback forms using Fourier analysis, namely Littlewood--Paley theory.  
These pullback forms can be decomposed into summands concentrated in different frequency ranges.

To start to get an idea how the proof works, first imagine that all the pullback forms are concentrated in a single frequency range.  If the frequency range is high, then we got a lot of cancellation when we integrate the forms, leading to a non-trivial bound for the degree.  If the frequency range is low, then we use the fact that $M$ is not scalable to get a non-trivial bound for the degree---roughly speaking, if all the relevant forms were large and low frequency, we could use them to build a ring homomorphism from $H^*(M; \RR)$ to $\Omega^*(M)$.  

In general, the pullback forms have contributions from many frequency ranges.  We carefully break up the integral for the degree into pieces involving different frequency ranges, and we use the two ideas above to bound the pieces.  It turns out that the interaction of different frequency ranges is important in this estimate.  In the worst case, the forms have roughly equal contributions in every frequency range.  Indeed, a self-map of $M$ which comes close to the upper bound must have pieces in a wide range of frequencies (see Proposition \ref{allfreq} for a precise statement).

Let's see what such a self-map might look like in the case of $M=\#_k \CC P^2$.  We think of $M$ as a CW complex with one $0$-cell, $k$ $2$-cells, and one $4$-cell.  We construct self-maps $\arr_\ell:M \to M$ which have degree $2^{4\ell}$ on the top cell. We would like to arrange that $\arr_\ell$ has Lipschitz constant at most $C \ell \cdot 2^{\ell}$.  A naive way to build a map $\arr_\ell$ of the right degree is to start with some $\arr_1$ and iterate it $\ell$ times to get $\arr_\ell$.  In this case, $\Lip(\arr_\ell) \le \Lip(\arr_1)^\ell$.  However, $\Lip(\arr_1)$ is strictly bigger than 2 (by \cite[2.32]{GrMS}, the Lipschitz constant could only be 2 if $M = \#_k \CC P^2$ had a flat metric).  Therefore, the bound $\Lip(\arr_1)^\ell$ is too big.  By performing some optimization each time we iterate, we can bring $\Lip(\arr_\ell)$ down to the target value.

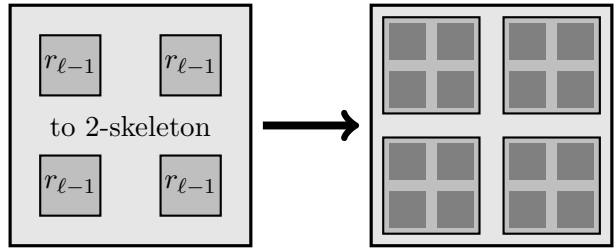
\begin{figure}
  \centering
  \begin{tikzpicture}[scale=0.8]
        \filldraw[very thick,fill=gray!20] (-2,-2) rectangle (2,2);
        \foreach \x/\y in {-1.5/-1.5, -1.5/0.5, 0.5/-1.5, 0.5/0.5} {
          \filldraw[thick,fill=gray!50] (\x,\y) rectangle (\x+1,\y+1);
          \draw (\x+0.5,\y+0.5) node {$\arr_{\ell-1}$};
        }
        \draw (0,0) node {to $2$-skeleton};
        \draw[line width=3pt,->] (2.2,0) -- (3.8,0);
        \filldraw[very thick,fill=gray!20] (4,-2) rectangle (8,2);
        \foreach \x/\y in {4.2/-1.8, 4.2/0.2, 6.2/-1.8, 6.2/0.2} {
          \filldraw[thick,fill=gray!50] (\x,\y) rectangle (\x+1.6,\y+1.6);
        }
        \foreach \x in {4.3,5.1,6.3,7.1} {
          \foreach \y in {-1.7,-0.9,0.3,1.1} {
            \fill[gray] (\x,\y) rectangle (\x+0.6,\y+0.6);
          }
        }
      \end{tikzpicture}
      \caption{Rescaling the ``layers'' of the iterated map.}
      \label{blocks}
  \end{figure}

We may build $\arr_1$, which has degree $16$, as follows: the top cell $e_4$ contains 16 cubical regions that each map homeomorphically, even homothetically, to the whole cell, whereas the area outside those cubical regions maps to the $2$-skeleton.  To try to make this map efficient, we can arrange the cubical regions in a $2 \times 2 \times 2 \times 2$ grid.  But when we iterate this map many times, the regions that map homothetically to the $4$-cell become tiny, and most of the $4$-cell maps to the $2$-skeleton.

The main idea of the construction is that we can actually expand the homothetic regions so that they take up a much larger part of the cell, while compressing the parts that map to the $2$-skeleton to a thin layer.  This has to do with the fact that self-maps of $S^2$ of high degree are easy to produce and modify.  In the end, each of the $\ell$ iterations contributes a layer of roughly the same thickness, leading to an estimate of $O(\ell \cdot 2^\ell)$ for the Lipschitz constant, or $O(d^{1/4}\log d)$ in terms of the degree $d=2^{4\ell}$.  See Figure \ref{blocks} for a rough illustration.

The proof of the lower bound of Theorem \ref{main} is a straightforward generalization of this idea.

To end this introduction, we consider the Littlewood--Paley pieces of the differential forms from this map and from other maps we have discussed.  For simplicity, let us first discuss a self-map $S^2 \to S^2$ with degree $2^{2p}$ and Lipschitz constant $2^p$.  The pullback of the volume form is very repetitive, so that after averaging on scale $2^{-p}$ it becomes essentially constant.  Therefore, the Littlewood--Paley pieces of the pullback are large at the highest frequency scale $2^p$ and at frequency 1, but they can be very small at all the in-between frequencies.

The maps between scalable spaces constructed in \cite{scal} have a similar Littlewood--Paley profile.  These maps are highly regular ``rescalings''.  In fact, we prove Theorem \ref{optconj} by building maps which are modeled on constant forms---the lowest possible frequency.  Such maps are built on each cell and patched together using previous results from quantitative homotopy theory.  The patching introduces high frequency pieces, but there don't need to be any contributions from the intermediate frequencies.

The Littlewood--Paley decomposition for the self-map of $\#_k \CC P^2$ sketched above is very different.  The outermost layer is dominated by very low-frequency terms (at scale around the diameter of the space) and very high-frequency terms (at scale $\sim 2^{-\ell}$).  Similarly, the $k$th layer, which looks like the outermost layer but on a different scale, is dominated by terms at scale $2^{-k}$ and $2^{-\ell}$.  Overall, the map has pieces at every frequency range, as suggested by its fractal-like self-similarity.

\subsection{Structure of the paper}
Section \ref{S:Fourier} contains the Fourier-analytic proof of the upper bound of Theorem \ref{main}; it is independent of the remainder of the paper.  Section \ref{S:lower} discusses the corresponding lower bound, and is likewise largely self-contained.  Section \ref{S:RHT} introduces some necessary results from rational and quantitative homotopy theory.  In Section \ref{S:optconj}, we use this machinery to prove Theorems \ref{optconj} and \ref{elliptic}, and in Section \ref{S:homotopies}, we use it to prove Theorem \ref{main:homotopies}.  Finally, in Section \ref{S:NF}, we discuss what our techniques can say about non-formal spaces, proving Theorem \ref{main:NF} as well as some complementary bounds.

\subsection{Acknowledgements} The second author is supported by a Simons Investigator Award.  The third author was partially supported by NSF individual grant DMS-2001042.  We are grateful to an anonymous referee for a number of comments, including one catching a significant error.

\section{Upper bounds on degree using Fourier analysis} \label{S:Fourier}



In this section, we show the upper bound of Theorem \ref{main}.  To introduce the method, we first handle the case of a connected sum of $\CC P^2$s:

\begin{thm} \label{cp2k}
  Let $X_k = \#_k \CC P^2$.  Fix a metric $g$ on $X_k$.  Suppose that $f: X_k \rightarrow X_k$ is $L$-Lipschitz.  If $k \ge 4$, then
  \[\deg (f) \le C(k, g) L^4 (\log L)^{-1/2}.\]
\end{thm}

We then use the same method to prove the general result:

\begin{thm} \label{gendegbound}
  Suppose that $M$ is a closed connected oriented $n$-manifold such that $H^*(M; \RR)$ does not embed into $\Lambda^* \RR^n$, and $N$ is any closed oriented $n$-manifold.  Then there exists $\alpha(M) > 0$ so that for any metric $g$ on $M$ and $g'$ on $N$ and any $L$-Lipschitz map $f: N \rightarrow M$,
  \[\deg(f) \le C(M,g,N,g') L^n (\log L)^{- \alpha(M)}.\]
\end{thm}

Note that by Theorem \ref{optconj}, proved later in the paper, if $M$ is simply connected and formal, then this condition holds if and only if $M$ is not scalable.  However, the theorem also holds for non-formal manifolds as well as those with nontrivial fundamental group.

A similar result also holds for many non-closed domain manifolds.  We give the proof for a unit ball, although it extends easily to any compact manifold with boundary:
\begin{thm} \label{balldegbound}
  Suppose that $M$ is a closed connected oriented $n$-manifold such that $H^*(M; \RR)$ does not embed into $\Lambda^* \RR^n$, and let $\alpha(M)>0$ be as in the statement of Theorem \ref{gendegbound}.  Let $B^n \subseteq \RR^n$ be the unit ball.  Then for any metric $g$ on $M$ and any $L$-Lipschitz map $f:B^n \to M$,
  \[\int_{B^n} f^*d\vol_M \leq C(M,g)L^n(\log L)^{-\alpha(M)}.\]
\end{thm}


As discussed in the introduction, we prove these results by using Littlewood--Paley theory to divide the forms into pieces at different frequency ranges.
In the first subsection, we review the tools from Littlewood--Paley theory that we need.  In the second part, we prove Theorem \ref{cp2k}.  In the third part, we introduce the modifications needed to prove the more general estimate in Theorem \ref{gendegbound}.

\subsection{Littlewood--Paley theory}

If $a$ denotes a differential form on $\RR^d$, then we can define its Fourier transform term by term.  In other words, if $I$ is a multi-index and  $ a = \sum_I a_I(x) dx^I$, then
\[\hat a := \sum_I \hat a_I dx^I.\]

To set up Littlewood--Paley theory, pick a partition of unity on Fourier space:
\[\sum_{k \in \ZZ} \eta_k (\xi) := 1,\]
where $\eta_k$ is supported in the annulus $\Ann_k := \{ \xi: 2^{k-1} \le |\xi| \le 2^{k+1} \}$.  We can also arrange that $0 \le \eta_k \le 1$, and that $\eta_k$ are smooth with appropriate bounds on their derivatives.

Then define
\[P_k a := ( \eta_k  \hat a)^{\vee},\]
where $\vee$ denotes the inverse Fourier transform.  We have $a = \sum_{k \in \ZZ} P_k a$, and we know that $\widehat{P_k a} = \eta_k \hat a$ is supported in $\Ann_k$.

We also write $P_{\le k} a = \sum_{k' \le k} P_{k'} a$, and $\eta_{\le k} = \sum_{k' \le k} \eta_k$, so $P_{\le k} a = (\eta_{\le k} \hat a)^\vee$.  

We say that a form $a = \sum_I a_I(x) dx^I$ is Schwartz if each function $a_I(x)$ is Schwartz.  A form $a$ is Schwartz if and only if $\hat a$ is Schwartz.  Therefore, if $a$ is Schwartz, then $P_k a$ and $P_{\le k} a$ are also Schwartz.

In this section, we review some estimates related to the $P_k a$.  These results are proven using some inequalities about the inverse Fourier transform of smooth bump functions.

\begin{lem} \label{etaveebound1} Suppose that $\eta(\omega)$ is a smooth function supported on a ball $B\subset \RR^d$ of radius~1 such that
\begin{itemize}
\item $| \eta(\omega) | \le A$ for all $\omega$.
\item $| \partial_J \eta(\omega) | \le A_N$ for all multi-indices $J$ with $|J| \le N$.
\end{itemize}
Then
\begin{align*}
  \lvert\eta^\vee(x)\rvert &\lesssim_d A \qquad\text{for every } x \in \RR^d. \\
  \lvert \eta^\vee(x) \rvert &\lesssim_d A_N \lvert x \rvert^{-N}  \qquad \text{for every } x \in \RR^d.
\end{align*}
Therefore, if $N > d$,
\[ \lVert \eta^\vee \rVert_{L^1} \lesssim_d A + A_N.\]

\end{lem}

\begin{proof}
For the first bound, we write
\[|\eta^\vee(x)| = | {\textstyle\int \eta(\omega) e^{2 \pi i \omega x} d \omega} | \le {\textstyle\int} | \eta| \le |B| A.\]

For the second bound, we integrate by parts $N$ times.  For a given $x \in \RR^d$, we choose a multi-index $J$ with $|J| =N$ and $|x|^N \sim x^J$.  Then
\[|\eta^\vee(x)| = \bigl\lvert {\textstyle\int \eta(\omega) e^{2 \pi i \omega x} d \omega} \bigr\rvert = \bigl\lvert{\textstyle\int \partial_J \eta (2 \pi i)^{-N} x^{-J} e^{2 \pi i \omega x} d \omega} \bigr\rvert \lesssim |x|^{-N} {\textstyle\int |\partial_J \eta|} \le \lvert x \rvert^{-N} \lvert B \rvert A_N.\]
To bound $\int |\eta^\vee(x)| dx$, we use the first bound when $|x| \le 1$ and the second bound when $|x| \ge 1$.
\end{proof}

\begin{lem} \label{etaveeboundR} Suppose that $\eta(\omega)$ is a smooth function supported on a ball $B\subset \RR^d$ of radius R such that
\begin{itemize}
\item $| \eta(\omega) | \le A$ for all $\omega$.
\item $| \partial_J \eta(\omega) | \le A_N R^{-|J|}$ for all multi-indices $J$ with $|J| \le N$.
\end{itemize}
Then
\begin{align*}
  | \eta^\vee(x) | &\lesssim_d A R^d \qquad\textrm{ for every } x \in \RR^d. \\
  | \eta^\vee(x) | &\lesssim_d A_N R^d \lvert Rx \rvert^{-N} \qquad\textrm{ for every } x \in \RR^d.
\end{align*}
Therefore, if $N > d$,
\[\|  \eta^\vee \|_{L^1} \lesssim_d A + A_N.\]

\end{lem}

\begin{proof}
The first two bounds follow from Lemma \ref{etaveebound1} by a change of variables.  Alternatively, one can use the same method as in Lemma \ref{etaveebound1}.

To bound $\int |\eta^\vee(x)| dx$, we use the first bound when $|x| \le 1/R$ and the second bound when $|x| \ge 1/R$.
\end{proof}

\begin{lem} \label{multiplierbound}
  Suppose that $\eta(\omega)$ is a smooth function supported on a ball $B\subset \RR^d$ of radius R such that
\begin{itemize}
\item $| \eta(\omega) | \le A$ for all $\omega$.
\item $| \partial_J \eta(\omega) | \le A_N R^{-|J|}$ for all multi-indices $J$ with $|J| \le N$.
\end{itemize}
Write $M f = \bigl( \eta \hat f \bigr)^\vee$.  Then if $N>d$,
\[\| M f \|_{L^p} \lesssim_d (A + A_N) \| f \|_{L^p}\text{ for every }1 \le p \le \infty.\]
\end{lem}

\begin{proof} We have $Mf = f * \eta^\vee$.  So $\| M f \|_{L^p} \le \| f \|_{L^p} \| \eta^\vee \|_{L^1}$.  Now apply the bound for $\| \eta^\vee \|_{L^1}$ from Lemma \ref{etaveeboundR}.
\end{proof}

We apply these bounds to study the Littlewood--Paley projections $P_k$.  

\begin{lem} \label{etaveebound}
  $\| \eta_k^\vee \|_{L^1} \lesssim 1$ uniformly in $k$.   $\| d \eta_k^\vee \|_{L^1} \lesssim 2^k$ uniformly in $k$.
\end{lem}

\begin{proof}
  We can first arrange that $\eta_k(\omega) = \eta_0(2^{-k} \omega)$.  Then the function $\eta_k$ obeys the hypotheses of Lemma \ref{etaveeboundR} with $R = 2^k$, with bounds that are uniform in $k$.  Then Lemma \ref{etaveeboundR} gives the estimate $\| \eta_k^\vee \|_{L^1} \lesssim_d 1$.  

Next, we will show that $\| \partial_j \eta_k^\vee \|_{L^1} \lesssim_d 2^k$.  This will imply $\| d \eta_k^\vee \|_{L^1} \lesssim_d 2^k$ as desired.

The Fourier transform of $\partial_j \eta_k^\vee$ is $2 \pi i \omega_j \eta_k(\omega)$.  Notice that $|\omega_j| \lesssim 2^k$ on $\Ann_k$.  We write
\[2 \pi i \omega_j \eta_k = 2^k \cdot \underbrace{ 2 \pi i \frac{\omega_j}{2^k} \eta_k }_{\psi}.\]

The function $\psi$ obeys the hypotheses of Lemma \ref{etaveeboundR}.  Therefore, $\| \psi^\vee \|_{L^1} \lesssim_d 1.$  And so
\[\| \partial_j \eta_k^\vee \|_{L^1} = 2^k \| \psi^\vee \|_{L^1} \lesssim 2^k. \qedhere\]
\end{proof}

\begin{lem} \label{Pkbound}
  $\| P_k a \|_{L^p} \le C \| a \|_{L^p}$, for all $k$ and all $1 \le p \le \infty$ with a uniform constant $C$.
\end{lem}

\begin{proof}
  $\|P_k a \|_{L^p} = \| \eta_k^\vee * a \|_{L^p} \le \| \eta_k^\vee \|_{L^1} \| a \|_{L^p}$.  Now $\| \eta_k^\vee \|_{L^1}$ is bounded uniformly in $k$ by Lemma \ref{etaveebound}.
\end{proof}

\begin{lem} \label{projcomm}
  The projection operator $P_k$ commutes with the exterior derivative $d$:
  \[d (P_k a) = P_k (da).\]
\end{lem}

\begin{proof}
  We can see this by taking the Fourier transform on both sides.  The exterior derivative $d$ becomes pointwise multiplication by a matrix on the Fourier side.  The projection operator $P_k$ becomes pointwise multiplication by the scalar $\eta_k$.  These commute.
\end{proof}

\begin{lem} \label{prim}
  Suppose that $a$ is a Schwartz form on $\RR^d$ with $da = 0$ and with $\hat a$ is supported in $\Ann_k:= \{ \xi: 2^{k-1} \le |\xi| \le 2^{k+1} \}$.  Then $a$ has a primitive, which we denote $\Prim(a)$, so that
\begin{itemize}
\item $ d \Prim(a) = a$.  (This is what the word `primitive' means.)
\item $\Prim(a)$ is a Schwartz form.
\item $ \lVert\Prim(a)\rVert_{L^p} \le C 2^{-k} \lVert a \rVert_{L^p}$ for all $1 \le p \le \infty$, with a uniform constant $C$.
\end{itemize}
\end{lem}

This is really the key property of frequency localized forms.  The intuition is that $\Prim(a)$ is defined by integrating $a$, and the integral cancels at length scales larger than $2^{-k}$. 

Before starting the proof, we make a quick remark about top-dimensional forms.  If $a$ is a $d$-form on $\RR^d$, then the condition $da=0$ is automatic.  In order for $a$ to have a Schwartz primitive, we need to know that $\int_{\RR^d} a = 0$.  This fact is implied by our assumption that $\hat a$ is supported in $\Ann_k$, because $\int_{\RR^d} a = \hat a(0) = 0$.

\begin{proof}
  First cover $\Ann_k$ with $\sim 1$ balls $B$ so that the radius of each ball is $\sim 2^k$ and the distance from each ball to the origin is also $\sim 2^k$.  Let $\psi_B$ be a partition of unity: $\sum_B \psi_B = 1$ on $\Ann_k$ and $\psi_B$ is supported in $B$.  Decompose $a = \sum_B a_B$ where
\[\hat a_B = \psi_B \hat a.\]
The form $\hat a_B$ is smooth and supported in $\Ann_k \cup \Ann_{k-1} \cup \Ann_{k+1}$.  Just as in the proof of Lemma \ref{projcomm}, it follows that $d a_B = 0$.  Using Lemma \ref{multiplierbound}, $ \| a_B \|_{L^p} \le C \| a \|_{L^p}$ for all $1 \le p \le \infty$.

We will construct a primitive $\Prim(a_B)$ for each form $a_B$ such that
\begin{itemize}
\item $ d \Prim(a_B) = a_B$. 
\item $\Prim(a_B)$ is a Schwartz form.
\item $ \lVert \Prim(a_B) \rVert_{L^p} \le C 2^{-k} \lVert a_B \rVert_{L^p}$ for all $1 \le p \le \infty$, with a uniform constant $C$.
\end{itemize}
Finally, we define $\Prim(a) = \sum_B \Prim(a_B)$.  Since $\Prim(a_B)$ has the desired properties, it follows that $\Prim(a)$ does also.

Now we have to construct $\Prim(a_B)$.  For ease of notation, we will abbreviate $a_B$ by $a$.  We know that $\hat a$ is supported on $B$.  We can choose coordinates so that $\omega_1 \sim 2^k$ on $B$.

We write the form $a$ as
\[\sum_I a_I(x) dx_I = \sum_{I = 1 \cup J} a_I(x) dx_1 \wedge dx_J + \sum_{1 \notin I} a_I dx_I.\]
We define the antiderivative $\int a_I dx_1$ via the Fourier transform by the formula:
\begin{equation} \label{defantider}
  \widehat {\textstyle\int a_I dx_1} (\omega) = \frac{1}{2 \pi i \omega_1} \hat a_I(\omega).
\end{equation}

Since $\omega_1 > 0$ on $B$, and $\hat a_I(\omega)$ is supported in $B$, the right-hand side is a smooth compactly supported function on Fourier space.  Therefore, $\int a_I dx_1$ is a Schwartz function on $\RR^d$.  From \eqref{defantider}, we can also check that 
\[\frac{\partial}{\partial x_1} \left( {\textstyle\int a_I dx_1} \right) = a_I.\]
(We can also define $\int a_I dx_1$ using definite integrals:
\[\int a_I dx_1 (x_1, x_2, ..., x_d) = \int_{- \infty}^{x_1} a_I(\tilde x_1, x_2, ..., x_d) d \tilde x_1.\]
This definite integral formula is equivalent to \eqref{defantider}.  From the definite integral formula, it takes a little work to check that $\int a_I dx_1$ is in fact a Schwartz function on $\RR^d$, although it's not that difficult.  In our proof we will only need \eqref{defantider}.)

We now define
\[\Prim(a) = \sum_{I = 1 \cup J} ({\textstyle\int} a_I dx_1) dx_J.\]
This is a standard construction for primitives of forms which appears in the proof of the Poincar\'e lemma, cf. \cite[p.~38]{BottTu}.   We will check that $d \Prim(a) = a$, following the same general method as in \cite{BottTu}.

We first compute $d( \int a_I dx_1)$:
\[d({\textstyle\int} a_I dx_1) = \partial_1 ({\textstyle\int} a_I dx_1) dx_1 + \sum_{j=2}^d \partial_j ({\textstyle\int} a_I dx_1) dx_j = a_I dx_1 + \sum_{j=2}^d {\textstyle\int} \partial_j a_I dx_1.\]
Now
\[d \Prim(a) =  \sum_{I = 1 \cup J} d ({\textstyle\int a_I dx_1}) dx_J = \sum_{I = 1 \cup J} a_I dx_1 \wedge dx_J + \sum_{I = 1 \cup J} \sum_{j=2}^d ({\textstyle\int \partial_j a_I dx_1}) dx_j \wedge dx_J.\]
The first term is $\sum_{I = 1 \cup J} a_I dx_I$.  So we have to check that the second term is the rest of $a$.  In other words, we want to show that
\begin{equation} \label{havetocheck}
  \sum_{I = 1 \cup J} \sum_{j=2}^d ({\textstyle\int} \partial_j a_I dx_1) dx_j \wedge dx_J = \sum_{1 \notin I'} a_{I'} dx_{I'}.
\end{equation}
Since both forms are Schwartz, it suffices to check that $\partial_1$ of both sides are equal:
\begin{equation} 
  \sum_{I = 1 \cup J} \sum_{j=2}^d \partial_j a_I  dx_j \wedge dx_J = \sum_{1 \notin I'} \partial_1 a_{I'} dx_{I'}.
\end{equation}
Since there is no 1 in $J$ or $j$ or $I'$, it suffices to check that $dx_1$ wedged with both sides are equal:
\begin{equation} 
  \sum_{I = 1 \cup J} \sum_{j=2}^d \partial_j a_I  dx_1 \wedge dx_j \wedge dx_J = \sum_{1 \notin I'} \partial_1 a_{I'} dx_1 \wedge dx_{I'}.
\end{equation}

This in turn follows from $da = 0$.

To bound $\Prim(a)$, the main point is that $| \frac{1}{2 \pi i \omega_1}| \sim 2^{-k}$ on the ball $B$.  Define $\eta_B = 1 $ on $B$, and $0 \le \eta_B \le 1$ and with $\eta_B$ supported in a slightly larger ball $\tilde B = 1.01 B$.  We can assume that $\omega_1 \sim 2^k$ on $\tilde B$.  Then
\[\frac{1}{2 \pi i \omega_1} \hat a_I(\omega) = 2^{-k} \underbrace{\frac{1}{2 \pi i} \frac{2^k}{\omega_1} \eta_B}_{\tilde \eta_B} \hat a_I (\omega).\]
The function $\tilde \eta_B$ is supported on $\tilde B$, and it obeys the bounds from Lemma \ref{multiplierbound}.  The lemma tells us that 
\[\| {\textstyle\int} a_I dx_1 \|_{L^p} = 2^{-k} \lVert\left( \tilde \eta_B \hat a_I \right)^\vee \rVert_{L^p} \le C 2^{-k} \lVert a_I \rVert_{L^p}.\]

Therefore $\lVert\Prim(a)\rVert_{L^p} \le C 2^{-k} \lVert a \rVert_{L^p}$ as desired.
\end{proof}

\begin{lem} \label{orth} For any function $f$,
\[\sum_{k \in \ZZ} \| P_k f \|_{L^2}^2 \sim \| f \|_{L^2}^2.\]
Similarly, for any form $a$,
\[\sum_{k \in \ZZ} \| P_k a \|_{L^2}^2 \sim \| a \|_{L^2}^2.\]
\end{lem}

\begin{proof}
  By the Plancherel theorem,
  \[\sum_{k \in \ZZ} \| P_k f \|_{L^2}^2 = \sum_{k \in \ZZ} \int_{\RR^d} \bigl\lvert \widehat{P_k f} \bigl\rvert^2 = \sum_{k \in \ZZ} \int_{\RR^d}  \lvert\eta_k(\omega)\rvert^2 \lvert\hat f(\omega)\rvert^2 d \omega.\]
  Now for every $\omega$, $(1/10) \le \sum_{k \in \ZZ} \eta_k(\omega)^2 \le 1$.  This holds because $\sum_{k \in \ZZ} \eta_k(\omega) = 1$ and each $\eta_k (\omega) \ge 0$, and each $\omega$ lies in the support of $\eta_k$ for at most 5 values of $k$.  Therefore,
  \[\sum_{k \in \ZZ} \| P_k f \|_{L^2}^2 = \int_{\RR^d} \biggl( \sum_{k \in \ZZ}   \eta_k(\omega)^2 \biggr) |\hat f(\omega)|^2 d \omega \sim \int_{\RR^d} |\hat f (\omega)|^2 d \omega = \int_{\RR^d} |f(x)|^2 dx.\]
  For a form $a = \sum_{I} a_I(x) dx_I$, $P_k(a) = \sum_I P_k a_I(x) dx_I$ and
  $\| a \|_{L^2}^2 := \sum_I \int |a_I(x)|^2 dx$.  So the case of forms follows
  from the case of functions.
\end{proof}

\begin{lem} \label{foursupp}
  The Fourier support of $P_{\le k} a_1 \wedge P_{\le k} a_2$ is contained in the ball of radius $2^{k+2}$ around 0.  Therefore,
  \[P_{\le k+3} \left( P_{\le k} a_1 \wedge P_{\le k} a_2 \right) = P_{\le k} a_1 \wedge P_{\le k} a_2.\]
\end{lem} 

\begin{proof}
The Fourier support of $P_{\le k} a$ is contained in the ball $B(2^{k+1}, 0)$.  For any functions $f$ and $g$, the Fourier transform of $fg$ is given by 
\[\widehat{fg}(\omega)  = \hat f * \hat g(\omega) = \int \hat f(\tilde \omega) \hat g( \omega - \tilde \omega) d \tilde \omega.\]
If $\hat f$ and $\hat g$ are supported in $B(2^{k+1}, 0)$, then $\widehat{fg}$ is supported in $B( 2\cdot 2^{k+1}, 0)$.

This argument also applies to wedge products of forms instead of products of functions, just by writing out the components of the forms.  This shows that the Fourier transform of $P_{\le k} a_1 \wedge P_{\le k} a_2$ is supported in $B(2^{k+2}, 0)$.  Now $\eta_{\le k+3}(\omega)$ is identically 1 on this ball, and so
\[P_{\le k+3} \left( P_{\le k} a_1 \wedge P_{\le k} a_2 \right) = P_{\le k} a_1 \wedge P_{\le k} a_2. \qedhere\]
\end{proof}

\subsection{Bounds for connected sums of $\mathbb CP^2$s}

\subsubsection{Setup}
In this section, we will prove Theorem \ref{cp2k}.  We recall the statement.

\begin{thm*}
  Let $X_k = (\CC P^2)^{\# k}$.  Fix a metric $g$ on $X_k$.  Suppose that
  $f: X_k \rightarrow X_k$ is $L$-Lipschitz.  If $k \ge 4$, then
  \[\deg (f) \le C(k, g) L^4 (\log L)^{-1/2}.\]
\end{thm*}

\begin{proof}
Let $u_i \in H^2(X_k; \RR)$ be a cohomology class dual to the $i$th copy of $\CC P^1$ in $X_k$, for $i = 1, \ldots, k$.  Let $\alpha_i$ be a 2-form in the cohomology class $u_i$.   We can assume that the $\alpha_i$ have disjoint supports.  For any $i$, we can write
\begin{equation} \label{degform}
  \deg (f) = \int_{X_k} f^* \alpha_i \wedge f^* \alpha_i.
\end{equation}

We will use Littlewood--Paley theory to estimate the right-hand side.  Because Littlewood--Paley theory is by far nicest on $\RR^d$, we first switch to charts.  Fix an atlas of charts for $X_k$: suppose that $X_k = \cup U'$, and $\phi_U: U \rightarrow U'$ are parametrizations.  Suppose that $\sum_{U'} \psi_{U'} = 1$ is a partition of unity on $X_k$ subordinate to these charts.  Define $\psi_U:\RR^4 \to \RR$ by
\[\psi_U(x) = \begin{cases}
  \phi^{*} \psi_{U'}(x) & x \in U \\
  0 & x \notin U.
\end{cases}\]
Now we can extend $\phi_U|_{\supp(U)}$ to a smooth map $\tilde\phi_U:\mathbb R^4 \to X_k$, and we can do it so that $\tilde\phi_U$ sends the complement of a compact set to a single point.  Then define differential forms $a_i$ on $\RR^4$ by 
\begin{equation} \label{defai}
  a_i = \tilde\phi_U^* f^*\alpha_i.
\end{equation}
(The forms $a_i$ also implicitly depend on $U$.)  Plugging this definition into \eqref{degform}, we get
\begin{equation} \label{degform2}
  \deg (f) = \sum_U \int_{\RR^4} \psi_U a_i \wedge a_i.
\end{equation}
We will bound each of these integrals.

Before going on, we discuss properties of the $a_i$.  We made sure these forms are defined on all of $\RR^4$ so that we can apply Littlewood--Paley theory.  We have $\| a_i \|_{L^\infty} \lesssim L^2$.  We also know that $d a_i = 0$.  
The form $a_i$ is supported on a fixed ball, and so for every $1 \le p \le \infty$, we also have $\| a_i \|_{L^p} \lesssim \| a_i \|_{L^\infty} \lesssim L^2$.

\subsubsection{Using that $k$ is large} \label{seckge4}

In this section, we prove a lemma that takes advantage of the fact that $k \ge 4$.   This lemma is similar to a lemma in \cite{scal}.

\begin{lem} \label{kge4} Suppose that $k \ge 4$ and that $b_1, \ldots, b_k$ are 2-forms on $\RR^4$.  Then at each point $x$, we have
\[| b_1 \wedge b_1 (x) | \le C \sum_{i \not= j} | b_i \wedge b_i - b_j \wedge b_j| + | b_i \wedge b_j|.\]
\end{lem}

\begin{proof}
  Suppose not.  By scaling, we can assume that $b_1 \wedge b_1(x) = dx_1 \wedge \cdots \wedge dx_4$.  Then we must have $b_j \wedge b_j(x)$ is almost $dx_1 \wedge \cdots \wedge dx_4$ for every $j$ and $b_i \wedge b_j(x)$ is almost zero for every $i \neq j$.  Next we will get a contradiction by considering the wedge product.

  Let $W: \Lambda^2 \RR^4 \times \Lambda^2 \RR^4 \rightarrow \Lambda^4 \RR^4$ be the quadratic form given by the wedge product.  It has signature (3,3).  Now let $B \subset \Lambda^2 \RR^4$ be the subspace spanned by $b_1, \ldots, b_k$.  When we restrict $W$ to the subspace $B$, we will check that it has signature $(k,0)$.  Since $k \ge 4$, this gives the desired contradiction.

  It remains to compute the signature of the quadratic form $W$ restricted to $B$.  This is isomorphic to  the quadratic form $(c_1, \ldots, c_k) \mapsto (\sum c_i b_i(x)) \wedge (\sum c_i b_i(x))$.  Expanding out the right-hand side, we get
  \[\sum_{i,j} c_i c_j b_i \wedge b_j.\]
  Since $b_i \wedge b_j$ is almost 0 for every $i \neq j$ and $b_i \wedge b_i$ is almost $dx_1 \wedge \cdots \wedge dx_4$ for every $i$, we see that this form is almost 
  \[(c_1, \ldots, c_k) \mapsto (c_1^2 + \cdots + c_k^2) dx_1 \wedge \cdots \wedge dx_4.\]
  In particular, the form has signature $(k,0)$.
\end{proof}

\subsubsection{Relations in cohomology and low-frequency bounds}

Let $u_i \in H^2(X_k; \RR)$ be a cohomology class dual to the $i$th copy of $\CC P^1$ in $X_k$, for $i = 1, \ldots, k$.  Let $\alpha_i$ be a 2-form in the cohomology class $u_i$.  

We know that $u_i \smile u_i - u_j \smile u_j = 0$ in $H^4(X_k; \RR)$.  Therefore, the corresponding differential forms $\alpha_i \wedge \alpha_i - \alpha_j \wedge \alpha_j $ are exact.  Similarly, for $i \neq j$, $u_i \smile u_j = 0$, and so the forms $\alpha_i \wedge \alpha_j$ are exact.  Let $\gamma_r$ be primitives for these forms.  We have $2{k \choose 2}$ exact forms total, and so $r$ goes from $1$ to $2{k \choose 2}$.

Define $g_r = \phi^* f^* \gamma_r$.  Since $\gamma_r$ is a 3-form, 
\begin{equation} \label{grbound}
  \|g_r \|_{L^\infty} \lesssim L^3.
\end{equation}
Depending on $r$, we have $dg_r = a_i \wedge a_i - a_j \wedge a_j$ or $dg_r = a_i \wedge a_j$ with $i \neq j$.  

The bound $\| g_r \|_{L^\infty} \lesssim L^3$ gives extra information about $a_i \wedge a_j$.  In particular, we get bounds on the low frequency parts of $a_i \wedge a_j$.

\begin{lem} \label{lowfreqaiaj}
  If $i \neq j$, then
  \begin{align*}
    \lVert P_{ k} ( a_i \wedge a_j) \rVert_{L^\infty} &\lesssim 2^k L^3 \\
    \lVert P_{k} ( a_i \wedge a_i - a_j \wedge a_j) \rVert_{L^\infty} &\lesssim 2^k L^3.
  \end{align*}
  The same bounds hold with $P_{\le k}$ in place of $P_k$.
\end{lem}

Notice that $\| a_i \|_{L^\infty} \lesssim L^2$, and so we have $\| a_i \wedge a_j \|_{L^\infty} \lesssim L^4$.  But the low frequency part of $a_i \wedge a_j$ obeys a much stronger bound.

\begin{proof}
  We write
  \[\left\lvert P_{k} (a_i \wedge a_j) (x) \right\rvert = \left\lvert \int \eta_{k}^\vee (y) a_i \wedge a_j (x-y) dy \right\rvert.\]
  We now substitute in $a_i \wedge a_j = d g_r$ and then integrate by parts:
  \[\left\lvert  \int \eta_{k}^\vee (y) d g_r (x-y) dy \right\rvert = \left\lvert \int d \eta_{k}^\vee (y) g_r(x-y) dy \right\rvert.\]
  Since $\| g_r \|_{L^\infty} \lesssim L^3$, and $\int |d \eta_k^\vee| \lesssim 2^k$ by Lemma \ref{etaveebound}, our expression is bounded by
  \[\lesssim L^3 \int | d \eta_k^\vee| \lesssim 2^k L^3.\]

The same proof applies to $ \| P_{k} ( a_i \wedge a_i - a_j \wedge a_j) \|_{L^\infty}$ and with $P_{\le k}$ in place of $P_k$.
\end{proof}

\subsubsection{Toy case: all forms are low frequency}

To illustrate how the tools we have developed work together, we now do a toy case of our main theorem: the case where all forms have low frequency.

Suppose that the forms $a_i$ are all low-frequency: $P_{\le 1} a_i = a_i$ for every $i$.  It follows that the wedge products are also fairly low frequency: $P_{\le 2} (a_i \wedge a_j) = a_i \wedge a_j$ for every $i, j$.

We can now bound $\int \psi_U a_1 \wedge a_1$ using the tools we have developed.  First, Lemma \ref{kge4} tells us that
\[\int \psi_U a_1 \wedge a_1 \le \int \psi_U |a_1 \wedge a_1| \le \sum_{i \neq j} \int \psi_U | a_i \wedge a_j| + \int \psi_U |a_i \wedge a_i - a_j \wedge a_j|.\]

We are discussing the low frequency special case, where $|a_i \wedge a_j| = | P_{\le 2} (a_i \wedge a_j) |$.  By Lemma \ref{lowfreqaiaj}, we have
\[|a_i \wedge a_j| = | P_{\le 2} (a_i \wedge a_j) | \lesssim L^3.\]
Similarly,
\[|a_i \wedge a_i - a_j \wedge a_j| = | P_{\le 2} (a_i \wedge a_i - a_j \wedge a_j) | \lesssim L^3.\]
Therefore, $\int \psi_U a_1 \wedge a_1 \lesssim L^3$, and so finally we have $\deg f \lesssim L^3$.

If we have a weaker low frequency assumption that $P_{\le \bar \ell} a_i = a_i$ for every $i$, then the same argument shows that $\deg f \lesssim 2^{\bar \ell} L^3$.  As long as the frequency range $2^{\bar \ell}$ is significantly less than $L$, then we get a strong estimate.  For instance, if $2^{\bar \ell} = L^{.9}$, then $\deg f \le L^{3.9}$.

\subsubsection{Bounding high-frequency contributions}

We use the Littlewood--Paley decomposition to write
\[\int_{\RR^d} \psi_U a_i \wedge a_i = \int_{\RR^d} \psi_U \sum_{k \in \ZZ} P_k a_i \wedge \sum_{\ell \in \ZZ} P_\ell a_i.\]
We can bound each term on the right-hand side by using our primitive estimate, Lemma \ref{prim}, and integration by parts:
\begin{align*}
  \left\lvert\int_{\RR^d} \psi_U P_k a_i \wedge P_\ell a_i \right\rvert
  &= \left\lvert\int_{\RR^d} \psi_U P_k a_i \wedge d( \Prim(P_\ell a_i) )\right\rvert \\
  &= \left\lvert\int d \psi_U \wedge P_k a_i \wedge \Prim( P_\ell a_i)\right\rvert \\
  &\le \int \lvert d\psi_U \rvert \lvert P_k a_i \rvert \lvert\Prim (P_\ell a_i)\rvert.
\end{align*}
Now $d \psi_U$ is a fixed $C^\infty_{comp}$ form, and we have $|P_k a_i| \lesssim L^2$ and $| \Prim P_\ell(a_i)| \lesssim 2^{-\ell} L^2$.  All together, we get the bound
\begin{equation} \label{highfreqsmall} 
 \left\lvert\int_{\RR^d} \psi_U P_k a_i \wedge P_\ell a_i\right\rvert \lesssim 2^{-\ell} L^4.
\end{equation}

This shows that the high-frequency parts of $a_i$ contribute little to the integral for the degree.  By summing this geometric series of error terms, we see that

\begin{lem}  \label{highfreqsmall2}
  For any frequency cutoff $\bar \ell$,
  \[\left\lvert\int_{\RR^d} \psi_U a_i \wedge a_i\right\rvert \lesssim \left\lvert\int \psi_U P_{\le \bar \ell} a_i \wedge P_{\le \bar \ell} a_i \right\rvert + O(2^{-\bar \ell} L^4).\]
\end{lem}

In particular, Lemma \ref{highfreqsmall2} allows us to resolve another toy case of our problem.  If every form $a_i$ is purely high-frequency, in the sense that $P_{\le \bar \ell} a_i = 0$, then Lemma \ref{highfreqsmall2} gives the bound $\deg f \lesssim 2^{- \bar \ell} L^4$.  For instance, if $2^{\bar \ell}$ is at least $L^{1/10}$, then we get a strong estimate: $\deg f \lesssim L^{3.9}$.

We now have strong bounds in two toy cases: the pure low frequency case and the pure high frequency case.  We will prove bounds in the general case by combining these tools.  

However, combining the tools is not completely straightforward.  Based on the discussion above, it initially sounds like we might get a bound of the form $\deg f \lesssim L^{4 - \beta}$ for some $\beta > 0$.  But there are maps $f$ with Lipschitz constant $L$ and degree at least $L^4 (\log L)^{-C}$ for some constant $C$.  The forms coming from these maps crucially have signifinant contributions at all frequency levels.

\subsubsection{Bounds in the general case}

We begin by applying Lemma \ref{highfreqsmall2}.  For any frequency cutoff $\bar \ell$, the lemma tells us that
\begin{equation} \label{breakoffhigh}
    \left\lvert\int_{\RR^d} \psi_U a_1 \wedge a_1\right\rvert \lesssim  \int \psi_U \left\lvert P_{\le \bar \ell} a_1 \wedge P_{\le \bar \ell} a_1\right\rvert + 2^{-\bar \ell} L^4.
\end{equation}
We will choose $\bar \ell$ later, in the range $2^{\bar \ell} \ge L^{1/10}$.  This guarantees that the last term is $\lesssim L^{3.9}$, which is much smaller than our goal.

To control the first term, we apply Lemma \ref{kge4} with $b_i = P_{\le \bar \ell} a_i(x)$ at each point $x$.  Lemma \ref{kge4} tells us that at each point
\[\left\lvert P_{\le \bar \ell} a_1 \wedge P_{\le \bar \ell} a_1 \right\rvert \lesssim \sum_{i \neq j} \left\lvert P_{\le \bar \ell} a_i \wedge P_{\le \bar \ell} a_j \right\rvert + \left\lvert P_{\le \bar \ell} a_i \wedge P_{\le \bar \ell} a_i - P_{\le \bar \ell} a_j \wedge P_{\le \bar \ell} a_j\right\rvert.\]
Plugging into the integral, we get
\[\int \psi_U | P_{\le \bar \ell} a_1 \wedge P_{\le \bar \ell} a_1 | \lesssim \sum_{i \neq j} \underbrace{\int \psi_U |P_{\le \bar \ell} a_i \wedge P_{\le \bar \ell} a_j|}_{I} + \underbrace{\int \psi_U |P_{\le \bar \ell} a_i \wedge P_{\le \bar \ell} a_i - P_{\le \bar \ell} a_j \wedge P_{\le \bar \ell} a_j|}_{II}.\]
The two terms are similar to each other.  We focus on the terms of type I first.  The same arguments apply to type II.

The form $P_{\le \bar \ell} a_i \wedge P_{\le \bar \ell} a_j$ looks a little bit like $P_{\le \bar \ell} (a_i \wedge a_j)$, which has strong bounds coming from Lemma \ref{lowfreqaiaj}.  However, these forms are not equal to each other.  We will examine the situation more carefully and find that
\begin{equation} \label{lowfreqcompare}
  P_{\le \bar \ell} a_i \wedge P_{\le \bar \ell} a_j = P_{\le \bar \ell + 3} (a_i \wedge a_j) + \text{ additional terms}.
\end{equation}
The additional terms are crucial to our story---they actually make the largest contribution in our bound for the degree of $f$.

To work out the details of \eqref{lowfreqcompare}, we begin by doing the Littlewood--Paley expansion of $a_i$ and $a_j$:
\[a_i \wedge a_j = \sum_{k_1, k_2 \in \ZZ} P_{k_1} a_i \wedge P_{k_2} a_j.\]
Grouping the terms according to whether $k_1$ or $k_2$ is bigger, we get
\begin{equation} \label{expandaiaj}
  a_i \wedge a_j =  P_{\le \bar \ell} a_i \wedge P_{\le \bar \ell} a_j + \sum_{k_1 = \bar \ell+1}^\infty P_{k_1} a_i \wedge P_{\le k_1} a_j + \sum_{k_2 = \bar \ell+1}^\infty P_{< k_2} a_i \wedge P_{k_2} a_j.
\end{equation}

Note that the Fourier transform of $P_{\le \bar \ell} a_i \wedge P_{\le \bar \ell} a_j$ is supported in $|\omega| \le 4 \cdot 2^{\bar \ell}$ (cf. Lemma \ref{foursupp}).  Therefore
\[P_{\le \bar \ell + 3} (  P_{\le \bar \ell} a_i \wedge P_{\le \bar \ell} a_j ) =  P_{\le \bar \ell} a_i \wedge P_{\le \bar \ell} a_j .\]
We apply $P_{\le \bar \ell + 3}$ to both sides of \eqref{expandaiaj} to get
\begin{equation} \label{expandaiaj2}
  \begin{aligned}
  P_{\le \bar \ell +3} ( a_i \wedge a_j ) &=  P_{\le \bar \ell} a_i \wedge P_{\le \bar \ell} a_j \\
  & \qquad {}+ \sum_{k_1 = \bar \ell+1}^\infty P_{\le \bar \ell +3} (P_{k_1} a_i \wedge P_{\le k_1} a_j) + \sum_{k_2 = \bar \ell+1}^\infty P_{\le \bar \ell + 3} (P_{< k_2} a_i \wedge P_{k_2} a_j).
  \end{aligned}
\end{equation}
This gives us our fleshed out version of \eqref{lowfreqcompare}:
\begin{equation} \label{lowfreqcompare2}
  \begin{aligned}
  P_{\le \bar \ell} a_i \wedge P_{\le \bar \ell} a_j &= \underbrace{P_{\le \bar \ell + 3} (a_i \wedge a_j)}_{\textrm{Term 1}} \\
  & \qquad {} - \underbrace{ \sum_{k_1 = \bar \ell+1}^\infty P_{\le \bar \ell +3} (P_{k_1} a_i \wedge P_{\le k_1} a_j) }_{\textrm{Term 2.1}}- \underbrace{\sum_{k_2 = \bar \ell+1}^\infty P_{\le \bar \ell + 3} (P_{< k_2} a_i \wedge P_{k_2} a_j)}_{\textrm{Term 2.2}}.
  \end{aligned}
\end{equation}

We want to bound $\int \psi_U | P_{\le \bar \ell} a_i \wedge P_{\le \bar \ell} a_j |$.  We plug in \eqref{lowfreqcompare2}, and then we have to bound the contributions of term 1, term 2.1 and term 2.2.  The contribution of Term 1 is bounded using Lemma \ref{lowfreqaiaj}:
\begin{equation} \label{term1bound}
\int \psi_U | P_{\le \bar \ell + 3} (a_i \wedge a_j)| \lesssim 2^{\bar \ell} L^3.
\end{equation}
We will choose $\bar \ell$ in the range $2^{\bar \ell} \le L^{9/10}$, and so the right-hand side is $\lesssim L^{3.9}$, much smaller than our goal.

Terms 2.1 and 2.2 are similar, so we just explain Term 2.1.  The contribution of Term 2.1 is at most
\begin{equation} \label{term21}
  \sum_{k_1 = \bar \ell+1}^\infty \int \psi_U \lvert P_{\le \bar \ell +3} (P_{k_1} a_i \wedge P_{\le k_1} a_j) \rvert \le \sum_{k_1 = \bar \ell + 1}^\infty \lVert  P_{\le \bar \ell +3} (P_{k_1} a_i \wedge P_{\le k_1} a_j) \rVert_{L^1}.
\end{equation}
We start with a direct bound for this $L^1$ norm.  Lemma \ref{Pkbound} gives
\[\|  P_{\le \bar \ell +3} (P_{k_1} a_i \wedge P_{\le k_1} a_j) \|_{L^1} \lesssim \|  P_{k_1} a_i \wedge P_{\le k_1} a_j \|_{L^1} \le \| P_{k_1} a_i \|_{L^1} \| P_{\le k_1} a_j \|_{L^\infty}.\]
Now Lemma \ref{Pkbound} again gives $\| P_{\le k_1} a_j \|_{L^\infty} \lesssim \| a_j \|_{L^\infty} \lesssim L^2$.  All together this gives
\begin{equation} \label{direct}
  \|  P_{\le \bar \ell +3} (P_{k_1} a_i \wedge P_{\le k_1} a_j) \|_{L^1} \lesssim L^2 \| P_{k_1} a_i \|_{L^1}.
\end{equation}

If $k_1 = \bar \ell$, this is the best bound we know.  But if $k_1$ is much larger than $\bar \ell$, then we can get a better estimate by using the primitive of $P_{k_1} a_i$ and integrating by parts.
\[ P_{\le \bar \ell +3} (P_{k_1} a_i \wedge P_{\le k_1} a_j)  = \eta_{\le \bar \ell + 3}^\vee * \left[ d \Prim ( P_{k_1} a_i) P_{\le k_1} a_j \right].\]
Writing out what this means and integrating by parts, we get:
\begin{align*}
  \left\lvert P_{\le \bar \ell +3} (P_{k_1} a_i \wedge P_{\le k_1} a_j) (x)\right\rvert
  &= \left\lvert \int \eta_{\le \bar \ell + 3}^\vee(y)  ( d \Prim ( P_{k_1} a_i))(x-y) \wedge P_{\le k_1} a_j (x-y) dy \right\rvert \\
  &= \left\lvert \int d \eta_{\le \bar \ell + 3}^\vee(y)  ( \Prim ( P_{k_1} a_i))(x-y) \wedge P_{\le k_1} a_j (x-y) dy \right\rvert.
\end{align*}
Therefore, we have a pointwise bound
\[\left\lvert P_{\le \bar \ell +3} (P_{k_1} a_i \wedge P_{\le k_1} a_j) \right\rvert \le \left\lvert d \eta_{\le \bar \ell + 3} * \left[ \Prim (P_{k_1} a_i) \cdot P_{k_1} a_j \right]\right\rvert.\]
Taking $L^1$ norms, we get
\[\lVert P_{\le \bar \ell +3} (P_{k_1} a_i P_{\le k_1} a_j) \rVert_{L^1} \le \lVert d \eta_{\le \bar \ell + 3} \rVert_{L^1} \lVert \Prim P_{k_1} a_i \rVert_{L^1} \lVert P_{k_1} a_j \rVert_{L^\infty}.\]
Now Lemma \ref{etaveebound} gives $ \| d \eta_{\le \bar \ell + 3} \|_{L^1} \lesssim 2^{\bar \ell}$ and Lemma \ref{prim} gives $ \lVert\Prim P_{k_1} a_i \rVert_{L^1} \lesssim 2^{-k_1} \lVert P_{k_1} a_i \rVert_{L^1}$.  We also know by Lemma \ref{Pkbound} that $\lVert P_{k_1} a_j \rVert_{L^\infty} \lesssim \lVert a_j \rVert_{L^\infty} \lesssim L^2$.  Putting these bounds together, we see that
\begin{equation} \label{intpartsbound}
  \lVert P_{\le \bar \ell +3} (P_{k_1} a_i \wedge P_{\le k_1} a_j) \rVert_{L^1} \lesssim 2^{\bar \ell - k_1} L^2 \lVert P_{k_1} a_i \rVert_{L^1}.
\end{equation}
Returning to the contribution of Term 2.1 in \eqref{term21}, we have the bound
\begin{equation} \label{term21bound}
  \sum_{k_1 = \bar \ell+1}^\infty \int \psi_U \lvert P_{\le \bar \ell +3} (P_{k_1} a_i \wedge P_{\le k_1} a_j) \rvert \le \sum_{k_1 = \bar \ell + 1}^\infty 2^{\bar \ell - k_1} L^2 \lVert P_{k_1} a_i \rVert_{L^1}.
\end{equation}

Putting together our bounds for all the different terms, we get the following estimate for any choice of scale $\bar \ell$:
\begin{equation} \label{finalbound}
 \left\lvert\int_{\RR^d} \psi_U a_1 \wedge a_1\right\rvert \lesssim 2^{- \bar \ell} L^4 + 2^{\bar \ell} L^3 + \sum_{k_1 = \bar \ell + 1}^\infty 2^{\bar \ell - k_1} L^2 \| P_{k_1} a_i \|_{L^1}.
\end{equation}
(On the right-hand side, the first term comes from high frequency pieces, the next term comes from Term 1 and is bounded using the low frequency method, and the final term comes from Terms 2.1 and 2.2.  The fact that $k \ge 4$ is used in the bound for Term 1.)

Let us pause to digest this bound.  To begin, note that the first two terms, $2^{- \bar \ell} L^4 + 2^{\bar \ell} L^3$, can be made much smaller than $L^4$.  For instance, we can choose $\bar \ell$ so that $2^{\bar \ell} = L^{1/2}$, and then these first two terms give $L^{3.5}$.  The final term is often the most important.

Now let us try to get some intuition about the last term.  Because of the exponentially decaying factor $2^{\bar \ell - k_1}$, the final term comes mainly from $k_1$ close to $\bar \ell$.  If $\| P_{k_1} a_i \|_{L^1}$ is very small for a range of $k_1$, then it is strategic for us to choose $\bar \ell$ at the start of this range.  This scenario could lead to a bound which is much stronger than $L^4 (\log L)^{-1/2}$ -- see Proposition \ref{allfreq} below.  On the other hand, it may happen that $\| P_{k_1} a_i \|_{L^1}$ are all roughly equal.  This is actually the worst scenario from the point of view of Theorem \ref{cp2k}.  In this case we can improve on the bound $\| P_{k_1} a_i \|_{L^1} \lesssim \| a_i \|_{L^1} = L^2$ by using the orthogonality of the $P_{k_1} a_i$. By Cauchy--Schwarz, $\| P_{k_1} a_i \|_{L^1} \lesssim \|P_{k_1} a_i \|_{L^2}$, and $\sum_{k_1} \| P_{k_1} a_i \|_{L^2}^2 \lesssim \| a_i \|_{L^2}^2 \lesssim L^4$.  If $\|P_{k_1} a_i \|_{L^1}$ are all equal, then we can compute $\| P_{k_1} a_i \|_{L^1} \lesssim L^2 (\log L)^{-1/2}$.  Plugging this into the last term, and summing the geometric series, the last term contributes $L^4 (\log L)^{-1/2}$.

We now finish the formal proof of Theorem \ref{cp2k}.  We will choose $\bar \ell$ in the range $L^{1/10} \le 2^{\bar \ell} \le L^{9/10}$.  The number of different $\bar \ell$ in this range is $\sim \log L$.  For each $\bar \ell$ in this range, (\ref{finalbound}) gives:

\[ \left\lvert\int_{\RR^d} \psi_U a_1 \wedge a_1\right\rvert \lesssim L^{3.9} + \sum_{k_1 = \bar \ell + 1}^\infty 2^{\bar \ell - k_1} L^2 \| P_{k_1} a_i \|_{L^1}.\]
Adding together all the $\bar \ell$ in this range, we get
\begin{equation} \label{addscales}
  \log L \left\lvert\int_{\RR^d} \psi_U a_1 \wedge a_1\right\rvert \lesssim L^{3.91} + \sum_{L^{1/10} \le 2^{\bar \ell} \le L^{9/10}} \sum_{k_1 = \bar \ell + 1}^\infty 2^{\bar \ell - k_1} L^2 \| P_{k_1} a_i \|_{L^1}.
\end{equation}
In this sum, the terms with $2^{k_1} > L$ can be bounded by $L^{3.9}$ and absorbed into the first term.  The remaining terms are
\[ \sum_{L^{1/10} \le 2^{k_1} \le L} \sum_{L^{1/10} \le 2^{\bar \ell} \le 2^{k_1}-1} 2^{\bar \ell - k_1} L^2 \| P_{k_1} a_i \|_{L^1} \lesssim \sum_{L^{1/10} \le 2^{k_1} \le L} L^2 \|P_{k_1} a_i \|_{L^1}.\]
Next we want to use orthogonality from Lemma \ref{orth}: $\| a_i \|_{L^2}^2 \sim \sum_k \| P_k a_i \|_{L^2}^2$.  To get these $L^2$ norms into play we apply Cauchy--Schwarz.  Since the $a_i$ are supported in a fixed ball, and since the $P_{k_1} a_i$ are rapidly decaying away from that ball, we have $\| P_{k_1} a_i \|_{L^1} \lesssim \| P_{k_1} a_i \|_{L^2}$.  Since there are $\sim \log L$ values of $k_1$ in the range $L^{1/10} \le 2^{k_1} \le L$, we have
\begin{align*}
  \sum_{L^{1/10} \le 2^{k_1} \le L} L^2 \|P_{k_1} a_i \|_{L^1} &\lesssim (\log L)^{1/2} L^2 \left( \sum_{L^{1/10} \le 2^{k_1} \le L} \|P_{k_1} a_i \|_{L^2}^2 \right)^{1/2} \\
  &\lesssim (\log L)^{1/2} L^2 \| a_i \|_{L^2} \lesssim (\log L)^{1/2} L^4.
\end{align*}
Plugging this back into \eqref{addscales}, we see that
\[\log L \left\lvert\int_{\RR^d} \psi_U a_1 \wedge a_1\right\rvert \lesssim L^{3.91} +  (\log L)^{1/2} L^4 \]
and so
\[\left\lvert\int_{\RR^d} \psi_U a_1 \wedge a_1\right\rvert \lesssim (\log L)^{-1/2} L^4.\]
But the degree of $f$ is given by \eqref{degform2}:
\[\deg (f) = \sum_U \int \psi_U a_1 \wedge a_1 \lesssim  (\log L)^{-1/2} L^4.\]
This finishes the proof of Theorem \ref{cp2k}. \end{proof}

The bound (\ref{finalbound}) contains somewhat more information than Theorem \ref{cp2k}.  It also tells us that if the degree of $f$ is close to $L^4 (\log L)^{-1/2}$, then the forms $a_i$ must have contributions from essentially all frequency ranges.  We make this precise in the following proposition.

\begin{prop} \label{allfreq}
Suppose that $k \ge 4$.  Suppose $f: X_k \rightarrow X_k$ is $L$-Lipschitz.  Let the forms $a_i$ be as in (\ref{defai}), and fix $0 < \beta_1 < \beta_2 < 1$.

Suppose that for every chart and every $i$, and every $k_1$ in the range $L^{\beta_1} < 2^{k_1} < L^{\beta_2}$,
\begin{equation} \label{freqrangebound}
    \lVert P_{k_1} a_i \rVert_{L^1} \le L^{2 - \gamma}.
\end{equation}
Then the degree of $f$ is bounded by $C(g) L^{4 - \eta}$, where
\[\eta = \min( \beta_1, \beta_2 - \beta_1, \gamma).\]
\end{prop}

\begin{proof}
Recall that $\| P_{k_1} a_i \|_{L^1} \lesssim \| a_i \|_{L^1} \lesssim L^2$.  The hypothesis (\ref{freqrangebound}) says that we have a stronger bound on $\| P_{k_1} a_i \|_{L^1}$ when $2^{k_1}$ lies in the range $[L^{\beta_1}, L^{\beta_2}]$.

To prove the bound, we plug all our hypotheses into the bound (\ref{finalbound}).  That shows that the degree is bounded by
\[L^{4 - \beta_1} + L^{3 + \beta_1} + \sum_{L^{\beta_1} \le 2^{k_1} \le L^{\beta_2}} L^{\beta_1} 2^{-k_1} L^2 L^{2 - \gamma} + \sum_{2^{k_1} \ge L^{\beta_2}} L^{\beta_1} 2^{-k_1} L^4.\]
Carrying out the geometric series and grouping terms finishes the proof.
\end{proof}

\subsection{General estimate}

In this section, we prove theorem \ref{gendegbound}.  We recall the statement.

\begin{thm*}
  Suppose that $M$ is a closed connected oriented $n$-manifold such that $H^*(M; \RR)$ does not embed into $\Lambda^* \RR^n$, and $N$ is any closed oriented $n$-manifold.  Then there exists $\alpha(M) > 0$ so that for any metric $g$ on $M$ and $g'$ on $N$ and any map $f: N \rightarrow M$ with $\Lip(f) = L$,
  \[\deg(f) \le C(M,g,N,g') L^n (\log L)^{- \alpha(M)}.\]
\end{thm*}
\begin{rmk}
  The constant $\alpha(M)$ depends only on the real cohomology algebra of $M$, $H^*(M; \RR)$.
\end{rmk}
\begin{rmk}
  Because the constant $C(M,g)$ depends on $g$, it suffices to prove the estimate for any one metric $g$.
\end{rmk}

The main difference between the general situation in Theorem \ref{gendegbound} and the special case $X_k = (\CC P^2)^{\# k}$ in Theorem \ref{cp2k} is to find the right analogue of Lemma \ref{kge4}.  Lemma \ref{kge4} takes advantage of the hypothesis that $k \ge 4$ for $X_k$.  Similarly, the following lemma takes advantage of the hypothesis that $H^*(M; \RR)$ does not embed into $\Lambda^* \RR^n$.  

\begin{lem} \label{topvsrel}
  Suppose that $M$ is a closed connected oriented $n$-manifold such that $H^*(M; \RR)$ does not embed into $\Lambda^* \RR^n$.  Then there exists an integer $m(M)$ so that the following holds.

  Let $u_j \in H^{d_j} (M; \RR)$ be a set of generators for the cohomology algebra of $M$, including a generator $u_{\topp} \in H^n(M; \RR)$.  Suppose that the relations of the cohomology algebra are given by $R_r(u_1, \ldots, u_J) = 0$.

Fix $\beta_j \in \Lambda^{d_j} \RR^n$ for each $j =1, \ldots, J$ such that $|\beta_j| \le 1$ for each $j$ and $| R_r(\vec \beta)| \le \eps$ for each $r$.  Then $| \beta_{\topp} | \le C_M \eps^{\frac{1}{2m}}$.

\end{lem}

\begin{proof}
  The tuple $(\beta_1, \ldots, \beta_J)$ belongs to the space $\prod_{j=1}^J \Lambda^{d_j} \RR^n$, which is isomorphic to $\RR^N$.  We can think of (each component of) $\beta_j$ as a coordinate on this space, and we can think of $R_r$ as a polynomial on this space.  We let $V(R_1, \ldots, R_k)$ be the set of $\vec \beta$ where all the polynomials $R_r$ vanish.

  Each $(\beta_1, \ldots, \beta_J) \in V(R_1, \ldots, R_k)$ corresponds to a homomorphism $\phi: H^*(M; \RR) \rightarrow \Lambda^* \RR^n$ with $\beta_j = \phi(u_j)$.  By hypothesis, each such homomorphism is non-injective.  By Poincar\'e duality, we have that each such homomorphism sends $u_{\topp}$ to 0.  Therefore, $\beta_{\topp} = 0$ on $V(R_1, \ldots, R_k)$.

  For any set $X \subset \RR^N$, we let $I(X)$ denote the ideal of polynomials $f \in \RR[\beta]$ that vanish on $X$.  So we see that $\beta_{\topp} \in I (V (R_1, \ldots, R_k))$.  The structure of $I( V( R_1, \ldots, R_k))$ is described by the real Nullstellensatz---cf. \cite[\S2.3]{Marshall}:

  \begin{thm} [Real Nullstellensatz]
    A polynomial $f \in \RR[\beta]$ lies in $I (V (R_1, \ldots, R_k))$ if and only if there is an integer $m \ge 1$ and polynomials $g_i, h_r \in \RR[\beta]$ so that
    \[f^{2m} + g_1^2 + \ldots + g_s^2 = \sum_{r=1}^k h_r R_r.\]
  \end{thm}

  By the real Nullstellensatz, we see that there is some integer $m$ such that
  \[\beta_{\topp}^{2m} + g_1(\beta)^2 + \ldots + g_s(\beta)^2 = \sum_r h_r(\beta) R_r(\beta).\]
  If we also know that $|\beta_j| \le 1$ for every $j$ and $|R_r(\beta)| \le \eps$ for every $r$, then we see that
  \[\beta_{\topp}^{2m} \le C_M \eps.\]
  Therefore, $| \beta_{\topp} | \le C_M \eps^{\frac{1}{2m}}$.
\end{proof}

With this lemma, we can start the proof of the theorem.  The ideas are the same.  We just have to carry them out in a more general situation, with a little more notation.  

Recall that $u_j \in H^{d_j} (M; \RR)$ is a set of generators for the cohomology of $M$, with $u_{\topp}$ the generator of $H^n(M; \RR)$.  Suppose that the relations of the cohomology algebra are given by $R_r(u_1, \ldots, u_J) = 0$.

Choose $\alpha_j$ to be a closed form on $M$ in the cohomology class $u_j$.  The cohomology class of $R_r(\vec \alpha)$ is zero, so $R_r(\vec \alpha)$ is exact.  Choose a primitive:
\[d \gamma_r = R_r(\vec\alpha).\]

Next suppose that $f: N \rightarrow M$ is an $L$-Lipschitz map.  Cover $N$ with charts $U'$, and let $1 = \sum_{U'} \psi_{U'}$ be a partition of unity subordinate to the cover.  Let $\phi: U \rightarrow U'$ be a parametrization of $U'$, where $U \subset \RR^n$, which extends to a smooth map $\phi: \RR^n \rightarrow M$ sending the complement of a large ball in $\RR^n$ to a single point in $M$.  Define a smooth compactly supported function
\[\psi_U(x) = \begin{cases}
  \phi^{*} \psi_{U'}(x) & x \in U \\
  0 & x \notin U.
\end{cases}\]

Define forms on $\RR^n$ which correspond to the $\alpha_j$ as follows:
\[a_j := \frac{1}{L^{d_j}} \phi^* f^* \alpha_j.\]
With this normalization, $\| a_j \|_{L^\infty} \lesssim 1$ and the $a_j$ are smooth compactly supported differential forms.  Then
\begin{equation} \label{degintegral}
  \deg(f) = L^n \sum_{U} \int_{\RR^n} \psi_U a_{\topp}.
\end{equation}

Define forms on $\RR^n$ which correspond to the $\gamma_r$ as follows.  If $\gamma_r \in H^{d_r}(M; \RR)$, then
\[g_r := \frac{1}{L^{d(\gamma_r)+1}} \phi^* f^* \gamma_r.\]
The forms $g_r$ are also smooth compactly supported differential forms.  The power of $L$ is chosen so that
\[d g_r = R_r(a_j).\]
The power of $L$ works out to make the forms $g_r$ very small:
\[\| g_r \|_{L^\infty} \lesssim L^{-1}.\]
This allows us to show that the low-frequency parts of the forms $R_r(a)$ are small.

\begin{lem} \label{lowfreqrel}
  $\| P_{\le k} R_r(a) \|_{L^\infty} \lesssim 2^k L^{-1}$.
\end{lem}

\begin{proof}
We start by computing
\[P_{\le k} R_r(a) (x) = \int_{\RR^n} \eta_k^\vee(y) R_r(a) (x-y) dy =  \int_{\RR^n} \eta_k^\vee(y) dg_r (x-y) dy.\]
Now we can integrate by parts to get
\[\int_{\RR^n} \eta_k^\vee(y) dg_r (x-y) dy =   \int_{\RR^n} d\eta_k^\vee(y) g_r (x-y) dy.\]
Taking norms and using $\| g_r \|_{L^\infty} \lesssim L^{-1}$, we see that
\[| P_{\le k} R_r(a) (x) | \lesssim L^{-1} \int | d \eta_k^\vee| \lesssim 2^k L^{-1}. \qedhere\]
\end{proof}

We want to bound $\int \psi_U a_{\topp}$.  We break this up into a low frequency and high frequency part at a frequency cutoff $k$ which we will choose later.  (Eventually we will average over many $k$.)
\begin{equation} \label{lowhigh}
  \int \psi_U a_{\topp} = \underbrace{\int \psi_U P_{\le k} a_{\topp}}_{\textrm{low}} + \underbrace{\sum_{\ell > k} \int \psi_U P_\ell a_{\topp}}_{\textrm{high}}.
\end{equation}

For the high frequency pieces in \eqref{lowhigh}, we will find a small primitive and then integrate by parts.  Lemma \ref{prim} tells us that $P_{\ell} a_{\topp}$ has a primitive with
\[\lVert \Prim ( P_\ell a_{\topp}) \rVert_{L^\infty} \lesssim 2^{- \ell} \lVert P_{\ell} a_{\topp} \rVert_{L^\infty} \lesssim 2^{- \ell} \lVert a_{\topp} \rVert_{L^\infty} \lesssim 2^{- \ell}.\]
Then we can bound $\int \psi_U P_{\ell} a_{\topp}$ by 
\[\int \psi_U P_{\ell} a_{\topp} = \int d \psi_U \Prim (P_\ell a_{\topp}) \lesssim 2^{-\ell}.\]
We will choose $k$ with $2^k \ge L^{1/10}$, and so the contribution of all the high frequency parts is bounded by $L^{-1/10}$, which is much smaller than the bound we are aiming for.

For the low-frequency piece in \eqref{lowhigh}), we apply Lemma \ref{topvsrel} to the forms $P_{\le k} a_j$.  Since all these forms have norm $\lesssim 1$ pointwise, the lemma gives us a pointwise bound
\[|P_{\le k} a_{\topp}(x)| \lesssim \sum_r | R_r( P_{\le k} a)|^{\frac{1}{2m}}.\]
Integrating and using the H\"older inequality, we get the bound
\begin{equation} \label{notepoint}
  \int \psi_U P_{\le k} a_{\topp} \le  \sum_r \int \psi_U | R_r(P_{\le k} a)|^{\frac{1}{2m}} \lesssim \sum_r \left( \int \psi_U |R_r(P_{\le k} a)| \right)^{\frac{1}{2m}}.
\end{equation}
In the H\"older step, in detail we wrote 
\begin{align*}
  \int \psi_U | R_r(P_{\le k} a)|^{\frac{1}{2m}} &= \int \psi_U^{\frac{2m-1}{2m}} \cdot \psi_U^{\frac{1}{2m}} |R_r(P_{\le k} a)|^{\frac{1}{2m}} \\
  &\le \underbrace{\left( \int \psi_U \right)^{\frac{2m-1}{2m}}}_{\lesssim 1} \left( \int \psi_U |R_r(P_{\le k} a)| \right)^{\frac{1}{2m}}.
\end{align*}

Now we have to bound each integral $\int \psi_U |R_r( P_{\le k} a)|$.  Since $\| a \|_{L^\infty} \lesssim 1$, we get a bound $\int \psi_U |R_r( P_{\le k} a)|\lesssim 1$, and to prove our theorem we need to beat this bound by a power of $\log L$, at least for some choice of $k$.  The key input is the bound on the low freq part of $R_r(a)$: Lemma \ref{lowfreqrel} tells us that $\| P_{\le k} R_r(a) \|_{L^\infty} \le 2^{k} L^{-1}$.  Next we have to relate $R_r (P_{\le k} a)$ with $P_{\le k} R_r(a)$.

Remember that each $R_r$ is a polynomial in the $a_j$.  Each $R_r(a_j)$ is a sum of terms of the form $c a_{j_1} \wedge \cdots \wedge a_{j_P}$.  If we do a Littlewood--Paley decomposition of each $a_j$, we see that
\begin{equation} \label{lpexpand}
  a_{j_1} \wedge \cdots \wedge a_{j_P} = \sum_{k_1, \ldots, k_P} P_{k_1} a_{j_1} \wedge \cdots \wedge P_{k_P} a_{j_P}.
\end{equation}
For each choice of $k_1, \ldots, k_P$, we write $k_{\max} = \max_p k_p$.  We let $p_{\max}$ be the value of $p$ that maximizes $k_p$.  If there is a tie, we let $p_{\max}$ be the smallest $p$ so that $k_p = k_{\max}$.  We can now organize the sum on the right-hand side of \eqref{lpexpand} according to the value of $k_{\max}$ and $p_{\max}$:
\begin{multline*}
  \sum_{k_1, \ldots, k_P} P_{k_1} a_{j_1} \wedge \cdots \wedge P_{k_P} a_{j_P}
  = \sum_{k_{\max}} \sum_{p_{\max} = 1}^P P_{< k_{\max}} a_{j_1} \wedge \cdots \wedge \\
  {} \wedge P_{< k_{\max} } a_{j_{p_{\max} - 1}} \wedge P_{k_{\max}} a_{j_{p_{\max}}} \wedge P_{\le k_{\max}} a_{j_{p_{\max}}+1} \wedge \cdots \wedge P_{\le k_{\max}} a_{j_P}.
\end{multline*}
Similarly,
\[P_{\le k} a_{j_1} \wedge \cdots \wedge P_{\le k} a_{j_P}= \sum_{k_{\max} \le k } \sum_{p_{\max} = 1}^P P_{< k_{\max}} a_{j_1} \wedge \cdots \wedge P_{k_{\max}} a_{j_{p_{\max}}} \wedge \cdots \wedge P_{\le k_{\max}} a_{j_P}.\]
Therefore,
\begin{multline*}
  P_{\le k} a_{j_1} \wedge \cdots \wedge P_{\le k} a_{j_P} \\
  =  a_{j_1} \wedge \cdots \wedge a_{j_P} - \sum_{k_{\max} > k } \sum_{p_{\max} = 1}^P P_{< k_{\max}} a_{j_1} \wedge \cdots \wedge P_{k_{\max}} a_{j_{p_{\max}}} \wedge \cdots \wedge P_{\le k_{\max}} a_{j_P}.
\end{multline*}

This discussion applies to each monomial of $R_r$.  Therefore, $R_r(a)$ is equal to $R_r( P_{\le k} a)$ plus a finite linear combination of terms of the form 
\begin{equation} \label{termform}
  \sum_{k_{\max} > k } \sum_{p_{\max} = 1}^P P_{< k_{\max}} a_{j_1} \wedge \cdots \wedge P_{k_{\max}} a_{j_{p_{\max}}} \wedge \cdots \wedge P_{\le k_{\max}} a_{j_P}. 
\end{equation}

Now for a large constant $c$, we have $P_{\le k+c}  R_r( P_{\le k} a_j) =  R_r( P_{\le k} a_j)$.  Therefore, $ R_r( P_{\le k} a)$ is equal to $P_{\le k+c} R_r(a) $ plus a finite linear combination of terms of the form 
\begin{equation} \label{termform2}
  \sum_{k_{\max} > k } \sum_{p_{\max} = 1}^P P_{\le k + c} \left( P_{< k_{\max}} a_{j_1} \wedge \cdots \wedge P_{k_{\max}} a_{j_{p_{\max}}} \wedge \cdots \wedge P_{\le k_{\max}} a_{j_P}\right) . 
\end{equation}
In summary,
\begin{equation} \label{strucsummary}
  R_r(P_{\le k} a) = P_{\le k+ c} R_r(a) + \text{ terms of the form } \eqref{termform2}.
\end{equation}

The first term in \eqref{strucsummary} is controlled by Lemma \ref{lowfreqrel}: $\| P_{\le k+ c} R_r(a) \|_{L^\infty} \lesssim 2^{k+c} L^{-1} \lesssim 2^k L^{-1}$.  
We will choose $k$ so that $2^k \le L^{9/10}$, so this term is bounded by $L^{-1/10}$, which is much smaller than our goal.

For each remaining term of type \eqref{termform2}, we will again take a primitive and integrate by parts.  We apply Lemma \ref{prim} to get a good primitive: $P_{k_{\max}} a_{j_{p_{\max}}} = d \Prim ( P_{k_{\max}} a_{j_{p_{\max}}} )$, where $\lVert \Prim ( P_{k_{\max}} a_{j_{p_{\max}}} ) \rVert_{L^p} \lesssim 2^{-k_{\max}} \lVert P_{k_{\max}} a_{j_{p_{\max}}} \rVert_{L^p}$ for every $1 \le p \le \infty$.  For each fixed choice of $k_{\max}$ and $p_{\max}$, we write
\begin{multline*}
  P_{\le k + c} \left( P_{< k_{\max}} a_{j_1} \wedge \cdots \wedge P_{k_{\max}} a_{j_{p_{\max}}} \wedge \cdots \wedge P_{\le k_{\max}} a_{j_P}\right) \\
  = \eta_{\le k + c}^\vee *  \left( P_{< k_{\max}} a_{j_1} \wedge \cdots \wedge d \Prim( P_{k_{\max}} a_{j_{p_{\max}}}) \wedge \cdots \wedge P_{\le k_{\max}} a_{j_P}\right) \\
  = d \eta_{\le k + c}^\vee *  \left( P_{< k_{\max}} a_{j_1} \wedge \cdots \wedge \Prim( P_{k_{\max}} a_{j_{p_{\max}}}) \wedge \cdots \wedge P_{\le k_{\max}} a_{j_P}\right).
\end{multline*}
We now take the $L^1$ norm of our term.  Since $\| a_j\|_{L^\infty}$ and $\| P_{< k_{\max} a_j} \|_{L^\infty}$ are all $\lesssim 1$, we see that
\begin{multline*}
  \bigl\lVert d \eta_{\le k + c}^\vee *  \left( P_{< k_{\max}} a_{j_1} \wedge \cdots \wedge \Prim( P_{k_{\max}} a_{j_{p_{\max}}}) \wedge \cdots \wedge P_{\le k_{\max}} a_{j_P}\right)\bigr\rVert_{L^1} \\
  \lesssim \bigl\lVert d \eta_{\le k+ c}^\vee \bigr\rVert_{L^1} \bigl\lVert \Prim ( P_{k_{\max}} a_{j_{p_{\max}}}) \bigr\rVert_{L^1} \lesssim 2^{k+c} 2^{-k_{\max}} \bigl\lVert P_{k_{\max}} a \bigr\rVert_{L^1}.
\end{multline*}

To summarize, we have proved the following bound on each summand of \eqref{termform2}:
\begin{equation} \label{termbound}
  \| P_{\le k + c} \left( P_{< k_{\max}} a_{j_1} \wedge \cdots \wedge P_{k_{\max}} a_{j_{p_{\max}}} \wedge \cdots \wedge P_{\le k_{\max}} a_{j_P}\right) \|_{L^1} \lesssim 2^{k+c} 2^{-k_{\max}} \| P_{k_{\max}} a \|_{L^1}.
\end{equation}
Now the $L^1$ norm of each term of form \eqref{termform2} is bounded as follows:
\begin{multline*}
  \biggl\lVert \sum_{k_{\max} > k } \sum_{p_{\max} = 1}^P P_{\le k + c} \left( P_{< k_{\max}} a_{j_1} \wedge \cdots \wedge P_{k_{\max}} a_{j_{p_{\max}}} \wedge \cdots \wedge P_{\le k_{\max}} a_{j_P}\right) \biggr\rVert_{L^1} \\
  \lesssim  \sum_{k_{\max} > k} 2^{k - k_{\max}} \| P_{k_{\max}} a \|_{L^1}.
\end{multline*}

We now have our bounds on all the terms and we just have to put them together.  Recall \eqref{notepoint} tells us that 
\begin{equation}
  \Bigl(\int \psi_U P_{\le k} a_{\topp}\Bigr)^{2m} \lesssim \sum_r  \int \psi_U |R_r(P_{\le k} a)|.
\end{equation}
By \eqref{strucsummary}, we can break up $R_r(P_{\le k} a)$ into pieces:
\begin{equation*}
  R_r(P_{\le k} a) = P_{\le k+ c} R_r(a) + \text{terms of the form }\eqref{termform2}.
\end{equation*}
We have now bounded each term on the right-hand side.  Combining our bounds, we see that
\[\Bigl(\int \psi_U P_{\le k} a_{\topp}\Bigr)^{2m} \lesssim \sum_r \int \psi_U |R_r(P_{\le k} a)| \lesssim 2^k L^{-1} + \sum_{k_{\max} > k} 2^{k-k_{\max}} \| P_{k_{\max}} a \|_{L^1}.\]

Let us pause to digest this bound.  The first term $2^k L^{-1}$ is very small as long as $2^k \le L^{9/10}$.  In the second term, there is exponential decay for $k_{\max} > k$.  Therefore, the main contribution on the right hand side is when $k_{\max} = k$, which gives $\| P_k a \|_{L^1}$.  For comparison, it would be straightforward to get an upper bound of $\| a \|_{L^1} \lesssim 1$.  The upper bound $\| P_k a \|_{L^1}$ is an improvement because it includes only one Littlewood--Paley piece of $a$.  We can then take advantage of this improvement by averaging over $k$ and using orthogonality: $\sum_k \| P_k a \|_{L^2}^2 \sim \| a \|_{L^2}^2$.   Now we turn to the details of this estimate.  

We will sum over $k$ in the range $L^{1/10} \le 2^k \le L^{9/10}$.  There are $\sim \log L$ different $k$ in this range.
\[\sum_{L^{1/10} \le 2^k \le L^{9/10}} \Bigl(\int \psi_U P_{\le k} a_{\topp}\Bigr)^{2m} \lesssim L^{-1/10} + \sum_{L^{1/10} \le 2^k \le L^{9/10}} \sum_{2^k < 2^{k_{\max}} < L} 2^{k-k_{\max}} \| P_{k_{\max}} a \|_{L^1}.\]
(Here, terms with $2^{k_{\max}} > L$ are bounded by the $L^{-1/10}$ term).  Now the last term is bounded by
\[\sum_{L^{1/10} \le 2^k \le L^{9/10}} \sum_{2^k < 2^{k_{\max}} < L} 2^{k-k_{\max}} \| P_{k_{\max}} a \|_{L^1} \lesssim  \sum_{L^{1/10} \le k_{\max} \le L} \| P_{k_{\max}} a \|_{L^1}.\]
The number of terms in this last sum is $\sim \log L$.  Therefore, we can use the Cauchy--Schwarz inequality to get
\[\sum_{L^{1/10} \le k_{\max} \le L} \| P_{k_{\max}} a \|_{L^1} \le (\log L)^{1/2} \biggl( \sum_{L^{1/10} \le 2^{k_{\max}} \le L} \| P_{k_{\max}} a \|_{L^1}^2 \biggr)^{1/2}.\]
Since $a$ is supported on a fixed compact set, and $P_{k_{\max}} a$ is essentially supported on that set, Cauchy--Schwarz gives $\| P_{k_{\max}} a \|_{L^1} \lesssim \| P_{k_{\max}} a \|_{L^2}$.  Plugging this into the last term above gives
\[(\log L)^{1/2} \biggl( \sum_{L^{1/10} \le 2^{k_{\max}} \le L} \| P_{k_{\max}} a \|_{L^2}^2 \biggr)^{1/2} \lesssim (\log L)^{1/2} \| a \|_{L^2}.\]
All together, we now have
\[\sum_{L^{1/10} \le 2^k \le L^{9/10}} \Bigl(\int \psi_U P_{\le k} a_{\topp} \Bigr)^{2m} \lesssim (\log L)^{1/2} \| a \|_{L^2}.\]

Since there are $\sim \log L$ terms on the left-hand side, we can choose $k$ in the range $L^{1/10} \le 2^k \le L^{9/10}$ so that
\[\Bigl(\int \psi_U P_{\le k} a_{\topp}\Bigr)^{2m} \lesssim (\log L)^{-1/2} \| a \|_{L^2} \lesssim (\log L)^{-1/2}.\]
Taking roots, we get $\int \psi_U P_{\le k} a_{\topp} \lesssim (\log L)^{-\frac{1}{4m}}$.

Recall that we broke up $\int \psi_U a_{\topp}$ into low frequency and high frequency pieces in \eqref{lowhigh}:
\[\int \psi_U a_{\topp} = \underbrace{\int \psi_U P_{\le k} a_{\topp}}_{\textrm{low}} + \underbrace{\sum_{\ell > k} \int \psi_U P_\ell a_{\topp}}_{\textrm{high}}.\]
We showed that the high frequency pieces are bounded by $\lesssim 2^{-k}$.  We just found $k$ with $L^{1/10} \le 2^k \le L^{9/10}$ where the low frequency piece has the bound $\lesssim (\log L)^{- \frac{1}{4m}}$.  Therefore, the total is bounded:
\[\int \psi_U a_{\topp} \lesssim (\log L)^{- \frac{1}{4m}}.\]

Recall from \eqref{degintegral} that $\deg f = L^n \sum_U \int \psi_U a_{\topp}$, and so
\[\deg f \lesssim L^n (\log L)^{- \frac{1}{4m}}.\]
This proves the theorem, with $\alpha(m) = \frac{1}{4m}$.  The integer $m$ came from the real Nullstellensatz, and it depended only on the cohomology ring $H^*(M; \RR)$.

\subsubsection{Proof of Theorem \ref{balldegbound}}
Finally, we describe the modifications needed to prove the result on the ball, which we restate here:
\begin{thm*}
  Suppose that $M$ is a closed connected oriented $n$-manifold such that $H^*(M; \RR)$ does not embed into $\Lambda^* \RR^n$, and let $\alpha(M)>0$ be as in the statement of Theorem \ref{gendegbound}.  Let $B^n \subseteq \RR^n$ be the unit ball.  Then for any metric $g$ on $M$ and any $L$-Lipschitz map $f:B^n \to M$,
  \[\int_{B^n} f^*d\vol_M \leq C(M,g)L^n(\log L)^{-\alpha(M)}.\]
\end{thm*}
\begin{proof}
  Our argument above already studies forms defined on a ball.  The only difference is that above we study $\int_{B^n} \psi f^*d\vol_M$, where $\psi:B^n \to M$ is some function which decays to $0$ at the boundary, whereas we now want to understand $\int_{B^n} f^*d\vol_M$.  To bridge the gap, we expand the domain.  Define a function $\tilde f:B_2(0) \to M$ on the ball of radius $2$ by
  \[\tilde f(x)=\begin{cases}
    f(x) & \lVert x \rVert \leq 1 \\
    f(x/\lVert x \rVert) & \lVert x \rVert>1.
  \end{cases}\]
  If $f$ is $L$-Lipschitz, this function is $2L$-Lipschitz.  Moreover, since $\tilde f$ has rank $n-1$ outside the ball of radius $1$, $\tilde f^*d\vol_M=0$ outside that ball.  Therefore, for any $\psi:\RR^n \to \RR$ such that $\psi|_{B^n} \equiv 1$, we have
  \[\int_{B_2(0)} \psi \tilde f^*d\vol_M=\int_{B^n} f^*d\vol_M.\]
  The argument in the proof of Theorem \ref{gendegbound} bounds the left side as desired.
\end{proof}

\section{Explicit construction of efficient self-maps} \label{S:lower}

In this section, we discuss the lower bound of Theorem \ref{main}, which follows from the following result:
\begin{thm} \label{self-maps}
  Let $Y$ be a formal compact Riemannian manifold such that $H_n(Y;\mathbb{Q})$ is nonzero for $d$ different values of $n>0$.  Then there are integers $a>0,p>1$ such that for every $\ell \in \mathbb{N}$ and $q=ap^\ell$, there is an $O(\ell^{d-1}p^\ell)$-Lipschitz map $\arr_q:Y \to Y$ which induces multiplication by $q^n$ on $H_n(Y;\mathbb{Q})$.
\end{thm}

For the purpose of this section, a simply connected finite CW complex $Y$ is \emph{formal} if and only if for some $q>1$, there is a map $\arr_q:Y \to Y$ which induces multiplication by $q^n$ on $H^n(Y;\QQ)$ for every $n$.  Clearly, if such a map exists for some $q$, then it exists for $q^\ell$ for every $\ell$.  
This is not the original definition of formality due to Sullivan, which is based on the rationalization of $Y$ \cite{DGMS,SulLong}; the equivalence of our definition in the case of finite complexes was first stated by \cite{Shiga}.

To see that Theorem \ref{self-maps} indeed implies the lower bound of Theorem \ref{main}, suppose that $Y$ is an $n$-manifold.  Let $K(\ell)$ be the Lipschitz constant of $\arr_{ap^\ell}:Y \to Y$, and notice that for $\ell \geq 2$,
\[K(\ell)/K(\ell-1)=p \cdot \frac{\ell^{d-1}}{(\ell-1)^{d-1}} \leq 2p.\]
Then for $L>>0$, somewhere between $L/2p$ and $L$ is a value of $K(\ell)$ for some $\ell$.  This means that for $q=ap^\ell$,
\[L/2p=O(q(\log q)^{d-1})\]
and therefore there is an $O(L)$-Lipschitz map $f:Y \to Y$ such that
\[\deg f=q^n=\Omega(L^n(\log L)^{-n(d-1)}).\]



\subsection{Warmup example}
We start by proving Theorem \ref{self-maps} in the simple case of connected sums of $\mathbb CP^2$, before moving on to the general case.
\begin{thm} \label{warmup}
  Let $M=\#_k \CC P^2$.  Then there is a constant $C$ such that for each $\ell>0$, there is a self-map $\arr_{2^\ell}:M \to M$ of degree $2^{4\ell}$ and Lipschitz constant bounded by $C\ell \cdot 2^\ell$.
\end{thm}
As discussed in the introduction, the strategy is to build $\arr_{2^\ell}$ inductively by gluing together several copies of $\arr_{2^{\ell-1}}$ without adding too much stuff in between.  Before giving the detailed proof, we start with a lemma about self-maps of spheres which will also be useful for the general case of Theorem \ref{self-maps}.
\begin{lem} \label{SntoSn}
  For every $d$, there is a map $f_d:S^n \to S^n$ of degree $d^n$ whose Lipschitz constant is $C_1(n)d$.  Moreover, for each $p>1$ there is a $C_2(n)pd$-Lipschitz homotopy $H_p:S^n \times [0,1] \to S^n$ between $f_{pd}$ and $f_d \circ f_p$.
\end{lem}
\begin{proof}
  Give $S^n$ the metric of $\partial [0,1]^{n+1}$, which is $C_0$-bilipschitz to the round metric, and divide the face $[0,1]^n \times \{0\}$ into $d^n$ identical sub-cubes, $d$ to a side.  We map all other faces to a base point, and the sub-cubes to the sphere by a rescaling of a degree 1 map
  \[g:([0,1]^n,\partial[0,1]^n) \to (\partial [0,1]^{n+1},\text{pt})\]
  whose restriction to $g^{-1}([0,1]^n \times \{0\})$ is homothetic to the identity map.  The resulting map has degree $d^n$ and its Lipschitz constant in the round metric is bounded by $C_0^2(\Lip g)d$.
  
  Now consider the map $f_d \circ f_p$.  Like $f_{pd}$, it consists of $(pd)^n$ cubical preimages of $S^n$, with the rest of the sphere mapped to the basepoint.  However, instead of one cluster of preimages filling a whole face of $\partial [0,1]^{n+1}$, there are $p^n$ clusters of slightly smaller preimages.  We homotope $f_d \circ f_p$ to $f_{pd}$ by linearly expanding these preimages to fill the whole face.  The Lipschitz constant of this homotopy is bounded by $\Lip f_d \cdot \Lip f_p=C_0^4(\Lip g)^2 pd$.
\end{proof}

\begin{proof}[Proof of Theorem \ref{warmup}]
  We fix a cell structure for $M=\#_k \CC P^2$ consisting of one $0$-cell, $k$ $2$-cells, and a $4$-cell.  Let $\iota:[0,1]^4 \to M$ be the inclusion map of the $4$-cell, and let
  \[\partial=\iota|_{\partial[0,1]^4}:S^3 \to M^{(2)}=\bigvee_{i=1}^k S^2\]
  be its attaching map.  The projection of $\partial$ to each $S^2$ summand has Hopf invariant one.  Notice that a map $\bigvee_{i=1}^k S^2 \to \bigvee_{i=1}^k S^2$ which sends each $S^2$ to itself with degree $d$ extends to a map $M \to M$ of degree $d^2$.
  
  We prove the theorem by induction on $\ell$.  For the base of the induction we take $\arr_1:M \to M$ to be any map whose restriction to each $2$-cell is the map $f_2:S^2 \to S^2$ from Lemma \ref{SntoSn}.
  
  For the inductive step, assume that we have constructed a $C(\ell-1) \cdot 2^{\ell-1}$-Lipschitz map $\arr_{2^{\ell-1}}:M \to M$ whose restriction to each $2$-cell is $f_{2^{\ell-1}}$.  To build $\arr_{2^\ell}$, we take a $2 \times 2 \times 2 \times 2$ grid of sub-cubes inside $[0,1]^4$, each of side length $\frac{1}{2}\cdot\frac{\ell-1}{\ell}$, and send each of them to $M$ via a homothetic rescaling of $\arr_{2^{\ell-1}} \circ \iota$.  Then the Lipschitz constant on each sub-cube is $C\ell \cdot 2^\ell$.
  
  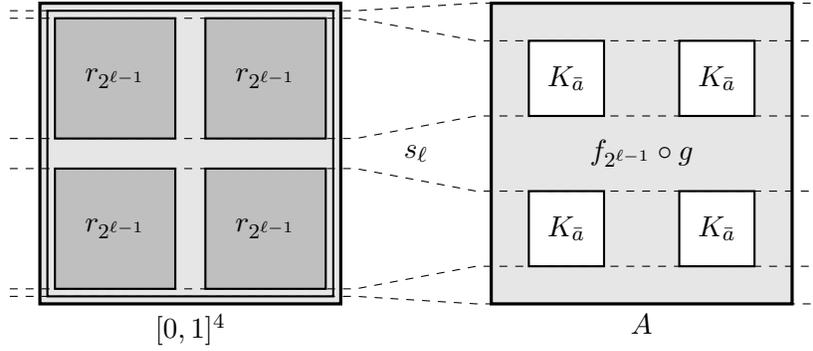
\begin{figure}
      \centering
      \begin{tikzpicture}
        
        \filldraw[very thick,fill=gray!20] (4,-2) rectangle (8,2);
        \draw[thick] (4.1,-1.9) rectangle (7.9,1.9);
        \foreach \x/\y in {4.2/-1.8, 4.2/0.2, 6.2/-1.8, 6.2/0.2} {
          \filldraw[thick,fill=gray!50] (\x,\y) rectangle (\x+1.6,\y+1.6);
          \draw (\x+0.8,\y+0.8) node {$\arr_{2^{\ell-1}}$};
        }
        \draw (6,-2) node[anchor=north] {$[0,1]^4$};

        \filldraw[very thick,fill=gray!20] (10,-2) rectangle (14,2);
        \foreach \x/\y in {10.5/-1.5, 10.5/0.5, 12.5/-1.5, 12.5/0.5} {
          \filldraw[thick,fill=white] (\x,\y) rectangle (\x+1,\y+1);
          \draw (\x+0.5,\y+0.5) node {$K_{\bar a}$};
        }
        \draw (12,-2) node[anchor=north] {$A$};
        \draw (12,0) node {$f_{2^{\ell-1}} \circ g$};

        \draw[dashed] (3.6,1.9) -- (8.2,1.9) -- (9.8,2) -- (14.4,2);
        \draw[dashed] (3.6,1.8) -- (8.2,1.8) -- (9.8,1.5) -- (14.4,1.5);
        \draw[dashed] (3.6,0.2) -- (8.2,0.2) -- (9.8,0.5) -- (14.4,0.5);
        \draw[dashed] (3.6,-1.9) -- (8.2,-1.9) -- (9.8,-2) -- (14.4,-2);
        \draw[dashed] (3.6,-1.8) -- (8.2,-1.8) -- (9.8,-1.5) -- (14.4,-1.5);
        \draw[dashed] (3.6,-0.2) -- (8.2,-0.2) -- (9.8,-0.5) -- (14.4,-0.5);
        \draw (9,0) node {$s_\ell$};
      \end{tikzpicture}
      \caption{Inductively assembling the map $\arr_{2^\ell}$.  The light gray regions map to $M^{(2)}$ and the dark gray regions map to the $4$-cell.  Some regions are labeled with the restriction of $\arr_{2^\ell}$ to that region.}
      \label{interstitial}
  \end{figure}
  
  We must now extend the map to the rest of $[0,1]^4$, filling the space in between with the same Lipschitz constant.  These gaps have width on the order of $1/\ell$.
  
  First, we fix some notation.  Let $A \subseteq [0,1]^4$ be the complement of the $16$ open subcubes
  \[K_{\bar a}=(a_1,a_2,a_3,a_4)+(1/8,3/8)^4,\qquad\text{for each }\bar a=(a_1,a_2,a_3,a_4),\; a_i \in \{0,1/2\},\]
  and fix a Lipschitz map $g:A \to M^{(2)}$ which restricts to a map homothetic to $\partial$ on each $\partial K_{\bar a}$ and to $f_2 \circ \partial$ on $\partial [0,1]^4$.  Here we write $f_d:\bigvee_{i=1}^k S^2 \to \bigvee_{i=1}^k S^2$ for the map which induces the map from Lemma \ref{SntoSn} on each wedge summand.

  Now we construct $\arr_{2^\ell}$ as follows:
  \begin{itemize}
      \item In $[0,1]^4$ but outside of $(\frac{1}{8\ell},1-\frac{1}{8\ell})^4$, the map is a homotopy from $f_{2^\ell} \circ \partial$ to $f_{2^{\ell-1}} \circ f_2 \circ \partial$.  Such a homotopy with domain $S^3 \times [0,1]$ can be made $C_2 \cdot 2^\ell$-Lipschitz by Lemma \ref{SntoSn}, so this map is $C_3\ell \cdot 2^\ell$-Lipschitz for some fixed constant $C_3$.
      \item In $[\frac{1}{8\ell},1-\frac{1}{8\ell}]^4$ but outside of the $16$ sub-cubes of width $2\frac{\ell-1}{\ell}$, the map is $f_{2^{\ell-1}} \circ g \circ s_\ell$, where $s_\ell$ is a $2\ell$-Lipschitz piecewise linear map that sends the domain to $A$, as shown in Figure \ref{interstitial}.
  \end{itemize}
  Then we have
  \[\Lip \arr_{2^\ell}=\max\{C_2\ell \cdot 2^\ell,C_1 \cdot 2^{\ell-1} \cdot \Lip g \cdot 2\ell,\frac{2\ell}{\ell-1}\Lip \arr_{2^{\ell-1}}\}.\]
  By induction, the theorem is proven with $C=\max\{C_2,2C_1\Lip g,\Lip \arr_1\}$.
\end{proof}

\subsection{Building efficient self-maps}
We give a mostly elementary proof of Theorem \ref{self-maps}, building maps $\arr_q$ ``by hand''.  The definition of formality gives us a self-map $\arr_p:Y \to Y$ of degree $p^n$; the proof consists of homotoping the iterates $(\arr_p)^\ell$ to maps $\arr_{p^\ell}$ with a controlled Lipschitz constant.  Although we have no control over the Lipschitz constant of the original $\arr_p$, this only affects the multiplicative constant.

First, we assume that $Y$ is a finite CW complex of a particular form.  We construct $\arr_{p^\ell}$ by induction on skeleta, extending along one cell at a time.  Each $n$-cell maps to itself with degree $p^{\ell n}$, and contains a grid of homeomorphic preimages of its interior, $p^\ell$ to a side.  The tricky part, and the source of the polylog factor, is filling in the area between these preimages.  This is done by induction on $\ell$: we take $p^n$ copies of $\arr_{p^{\ell-1}}$, arranged in a grid, and glue them together using a homotopy built in the course of the $(n-1)$-dimensional construction.  The Lipschitz constant of this homotopy is proportional to the Lipschitz constant obtained for self-maps of $Y^{(n-1)}$; since there are $\log \ell$ nested layers, we gain a factor of $\log\ell$ in moving from $Y^{(n-1)}$ to $Y^{(n)}$.

In passing from self-maps of the CW complex to those of our original manifold, we gain an additional factor of $a$ for the degree.


We now give the details of this argument.  This is the heart of the proof of Theorem \ref{self-maps}, although it only covers a special case.  The remainder of the section after this proof is devoted to showing that this is sufficient to prove the general case.


\begin{lem} \label{detailed}
  Let $Z$ be a simply connected finite CW complex with the following properties:
  \begin{itemize}
  \item $H^i(Z)$ is nontrivial in $d$ different dimensions (not including $i=0$).
  \item The cellular chain complex has zero differential.  (In other words, the cells are in bijection with a basis for $H^*(Z)$.)
  \item The attaching maps of $Z$ are Lipschitz maps $D^n \to Z^{(n-1)}$.
  \end{itemize}
  Let $\arr_p:Z \to Z$ be a map which induces multiplication by $p^i$ on $H^i(Z;\mathbb Q)$ for every $i>0$.  Then there is a metric on $Z$ such that every iterate $(\arr_p)^\ell$ of $\arr_p$ is homotopic to a $C(\arr_p,Z)\ell^{d-1}p^\ell$-Lipschitz map $\arr_{p^\ell}:Z \to Z$.  Moreover, $\arr_{p^\ell}$ is homotopic to $\arr_{p^{\ell-1}} \circ \arr_p$ via a $C'(\arr_p,Z)\ell^{d-1}p^\ell$-Lipschitz homotopy $H_\ell:Z \times [0,1] \to Z$.
\end{lem}
The homotopy $H_\ell$ is needed for the inductive step, in order to prove the
lemma one dimension higher.
\begin{proof}
  First suppose that $d=1$, and let $n=\dim Z$.  Then $Z$ is a wedge of $n$-spheres, so the base of the induction is provided by Lemma \ref{SntoSn}.

  Now suppose that we have proved the lemma for spaces with cells in $d-1$ dimensions, in particular for $Z^{(n-1)}$ where $\dim Z=n \geq 3$.  We start by building a metric on $Z$ as follows.  First, homothetically shrink $Z^{(n-1)}$ until the attaching maps of $n$-cells can be given by $1$-Lipschitz maps from $\partial[0,1]^n$.  Then give $Z$ the \emph{nearly Euclidean metric} (as defined further down in \S\ref{subS:metric}) derived from attaching cells isometric to $[0,1]^n$.  

  By Proposition \ref{htpy-to-lip}, proved further down, we can also assume that $\arr_p:Z \to Z$ is cellular and Lipschitz.  By applying a homotopy which is constant on the $(n-1)$-skeleton, we can also ensure that $\arr_p$ has the following property:
  \begin{quote}
    For every open $n$-cell $e$ of $Z$, $\overline{\arr_p^{-1}(e)}$ is a disjoint union of $p^n$ subcubes of $(0,1)^n$, arranged in a grid inside $e$, whose interiors map homothetically to $e$.
  \end{quote}
  Such a homotopy can be performed in several steps.  First, ensure that $\arr_p$ is smooth on the preimages of the ``middle halves'' of $n$-cells, and that the centers of the cells are regular values.  Then, by composing with a homotopy that expands a small neighborhood of the center to cover the whole cell, ensure that the preimage of each open $n$-cell is a disjoint union of homeomorphic copies.  Then, since $Z$ is simply connected and $n \geq 3$, it is possible to cancel out copies of opposite orientations.  The details of this purely topological argument can be found, for example, in \cite[Lemma 5.3]{GrMo} or \cite{White}.  Finally, we can deform this map to obtain the desired geometry.

  We now construct $\arr_{p^\ell}$ and $H_\ell$ by induction on $\ell$.  Suppose we have constructed a map $\arr_{p^{\ell-1}}$ that is $C(\arr_p,Z)(\ell-1)^{d-1}p^{\ell-1}$-Lipschitz and is an extension of $\arr_{p^{\ell-1}}^{(n-1)}$ to the $n$-cells of $Z$.  We will homotope $\arr_{p^{\ell-1}} \circ \arr_p$ to the desired $C(\arr_p,Z)\ell^{d-1}p^\ell$-Lipschitz map $\arr_{p^\ell}$.

  We first apply the homotopy $H_\ell^{(n-1)}$ to $Z^{(n-1)}$.  We extend this homotopy to a $n$-cell $e$ as follows.  Equip $e$ with polar coordinates $(s,\theta)$, with $\theta \in S^{n-1}$ and $s \in [0,1)$, and let $\partial_e:S^{n-1} \to Z^{(n-1)}$ denote the attaching map of $e$.  We define a homotopy $\tilde H:e \times [0,1] \to Z^{(n-1)}$ by
  \[\tilde H(s,\theta,t)=\begin{cases}
  H_\ell^{(n-1)}(\partial_e(\theta),t+2(s-1)), & s \geq 1-t/2, \\
  \arr_{p^{\ell-1}} \circ \arr_p(\theta,(1-t/2)^{-1}s), & s \leq 1-t/2.
  \end{cases}\]
  From this formula we see that:
  \begin{itemize}
  \item When $s=1$, $\tilde H(s,\theta,t)$ agrees with $H_\ell^{(n-1)}$.
  \item At $s=1-t/2$, $\tilde H$ is continuous since
  \[H_\ell^{(n-1)}(\partial_e(\theta),t+2(s-1))=H_\ell^{(n-1)}(\partial_e(\theta),0)=\arr_{p^{\ell-1}} \circ \arr_p(\theta,1).\]
  \end{itemize}
  
  \begin{figure}
    \centering
    \begin{minipage}{.33\textwidth}
        \centering
        \subfloat[\centering $\tilde H|_{t=0}=\arr_{p^{\ell-1}} \circ \arr_p$]
        {\includegraphics[width=4cm]{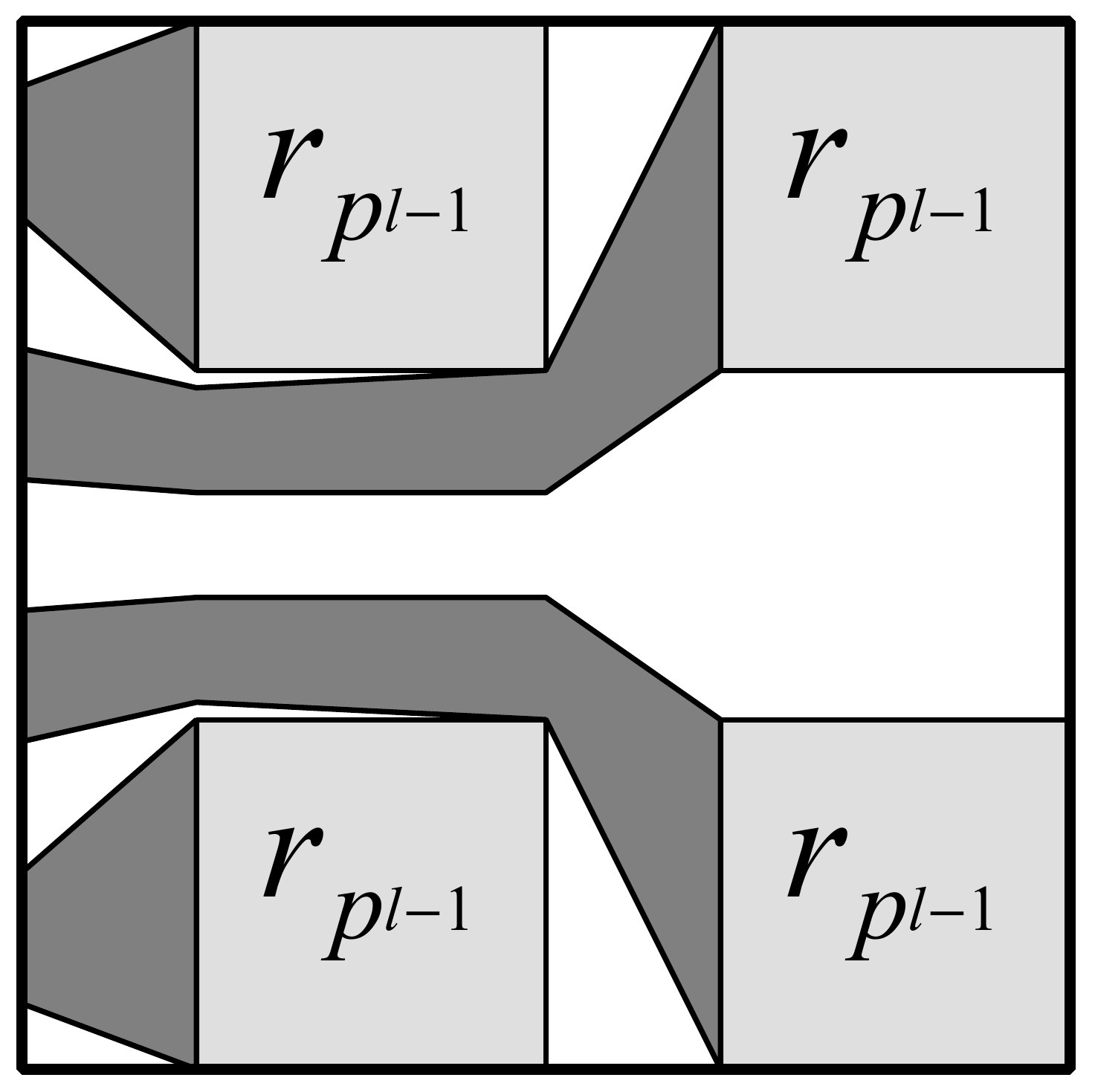}}
        \label{ft1}
    \end{minipage}%
    \begin{minipage}{.33\textwidth}
        \centering
        \subfloat[\centering $\tilde H|_{t=1}$: $L_i$'s differ]
        {\includegraphics[width=4cm]{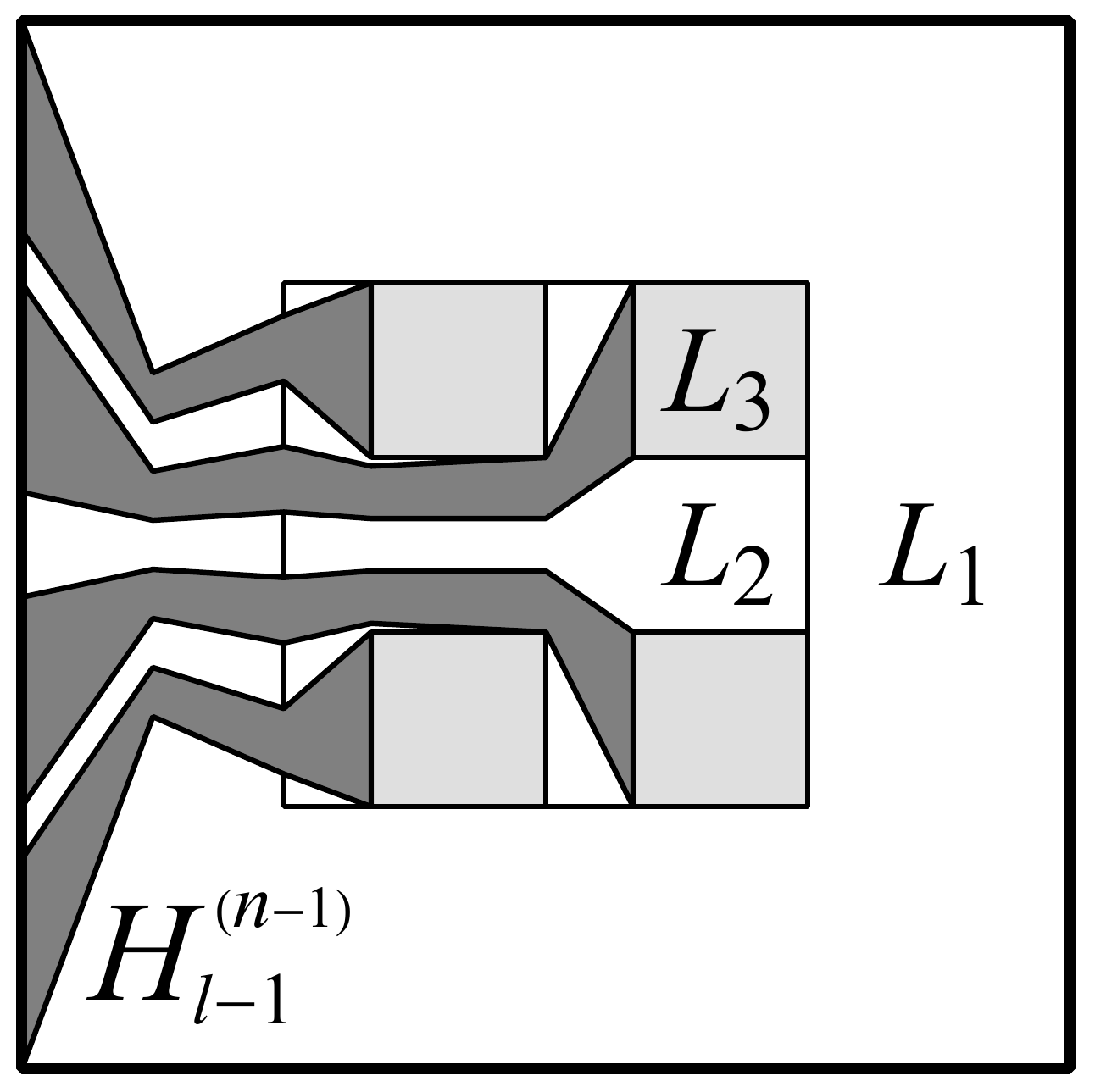}}
        \label{ft2}
    \end{minipage}
    \begin{minipage}{.33\textwidth}
        \centering
        \subfloat[\centering $J|_{t=1}$: $L_i$'s equalized]
        {\includegraphics[width=4cm]{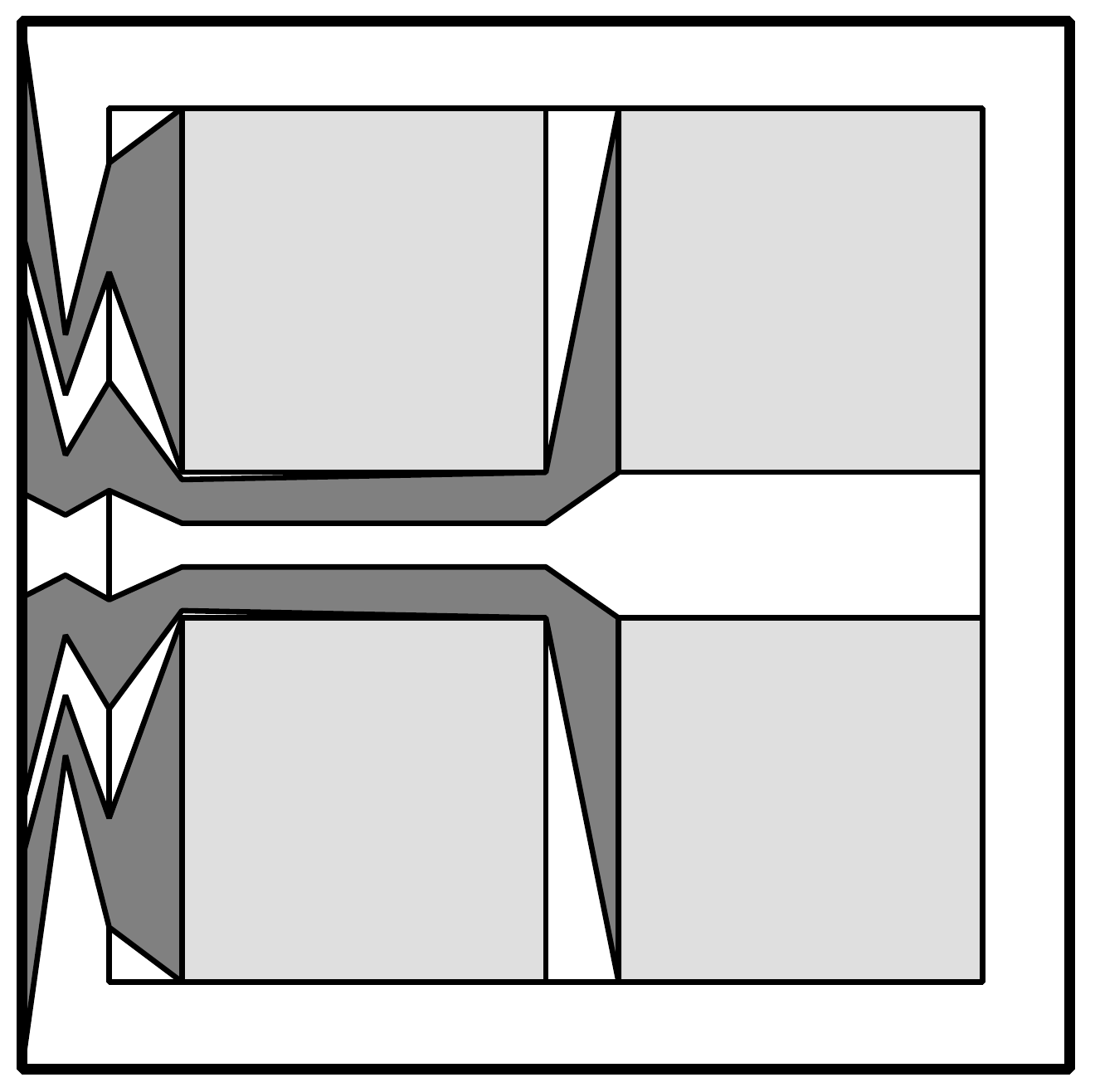}}
        \label{ft3}
    \end{minipage}
    \caption{Stages of the homotopy $H_\ell$, the concatenation of $\tilde H$ and $J$.}
  \end{figure}
  
  At this point, $\tilde H|_{e \times \{1\}}$ has different Lipschitz constants on different regions of $e$, which we bound by induction on $\ell$ and $d$:
  \begin{enumerate}[(i)]
  \item On the outer half of the disk, the Lipschitz constant is
    \[L_1=2\Lip H_\ell^{(n-1)} \leq 2C'(\arr_p,Z^{(n-1)})\ell^{d-2}p^\ell.\]
  \item On the inner half, but outside $\frac{1}{2}\arr_p^{-1}(e)$ (here $\frac{1}{2}$ refers to the homothety $(s,\theta) \mapsto (\frac{s}{2},\theta)$), the Lipschitz constant is
    \[L_2=2\Lip\bigl(\arr_{p^{\ell-1}} \circ \arr_p\bigr) \leq \Lip(\arr_p) \cdot 2C(\arr_p,Z^{(n-1)})(\ell-1)^{d-2}p^{\ell-1}.\]
    This bound holds because on this subdomain, the image of $\arr_p(\theta,s/2)$ lies in $Z^{(n-1)}$.
  \item In $\frac{1}{2}\arr_p^{-1}(e)$, the Lipschitz constant is
    \[L_3=D^{-1}\Lip \arr_{p^{\ell-1}} \leq D^{-1}C(\arr_p,Z)(\ell-1)^{d-1}p^{\ell-1},\]
    where $D$ is the side length of one of the subcubes comprising $\frac{1}{2}\arr_p^{-1}(e)$.
  \end{enumerate}
  In the second stage $J:Z \times [0,1] \to Z$ of the homotopy, which is constant on $Z^{(n-1)}$, we expand and shrink these three regions via a product of piecewise linear homotopies of $[0,1]$ so as to equalize the Lipschitz constants.  At time $1$, $e$ is nearly covered by a $p \times \cdots \times p$ grid of subcubes which each map to $Z$ via $r_{p^{\ell-1}}|_e$ composed with a homothety; the outer half of $\tilde H|_{e \times \{1\}}$ is relegated to a thin shell on the outside of the cube.  We can imagine expanding every part of the domain until the Lipschitz constant is $1$ on each relevant subinterval, and then shrinking the whole domain proportionally.  This shows that the resulting map $J|_{t=1}$ has Lipschitz constant bounded above by
  \begin{multline*}
    pDL_3+\left(\frac{1}{2}-pD\right)L_2+\frac{1}{2}L_1 \\
    \leq pC(\arr_p,Z)(\ell-1)^{d-1}p^{\ell-1} + \Lip(\arr_p)C(\arr_p,Z^{(n-1)})(\ell-1)^{d-2}p^{\ell-1} + C'(\arr_p,Z^{(n-1)})\ell^{d-2}p^\ell \\
  \leq C(\arr_p,Z)\ell^{d-1}p^\ell,
  \end{multline*}
  where the second inequality holds as long as
  \[C(\arr_p,Z) \geq p^{-1}\Lip(\arr_p)C(\arr_p,Z^{(n-1)})+C'(\arr_p,Z^{(n-1)}).\]


  Then we set $\arr_{p^\ell}=J|_{t=1}$ and $H_\ell$ to be the concatenation of $\tilde H$ and $J$.  By computing derivatives of $\tilde H$ and $J$ in the space and time directions, we see that
  \[\Lip(H_\ell)=\max\{L_1,L_2,L_3\},\]
  and therefore we can set $C'(\arr_p,Z) \leq \max\{2,(pD)^{-1}\}C(\arr_p,Z)$.
\end{proof}

\subsection{Lipschitz homotopy equivalence} \label{subS:metric}
To show that Lemma \ref{detailed} implies Theorem \ref{self-maps}, we need to introduce some geometric and topological facts.  We start with the geometry, discussing metrics on CW complexes: we would like to show that the ``special'' metric we imposed on the complex $Z$ in Lemma \ref{detailed} is not too special to be useful.

The relevant ideas date back to Gromov, see e.g.~\cite[\S7.20]{GrMS}, and are developed more systematically in \cite{LLY}.  The basic idea is that if two homotopy equivalent metric spaces are compact and sufficiently locally nice, then they are Lipschitz homotopy equivalent (in the obvious sense).

The importance of this is that asymptotic results about Lipschitz constants are preserved under Lipschitz homotopy equivalence.  That is, for metric spaces $X$ and $Y$, define the Lipschitz norm of a homotopy class $\alpha \in [X,Y]$ to be
\[\lVert\alpha\rVert_{\Lip}=\min \{\Lip(f) : f \in \alpha\}.\]
Suppose now that $f:X' \to X$ and $g:Y \to Y'$ are Lipschitz homotopy equivalences.  Then there are constants $C,K>0$ depending on $f$ and $g$ (but not $\alpha$) such that
\[\frac{1}{C}\lVert\alpha\rVert_{\Lip}-K \leq \lVert g \circ \alpha \circ f \rVert_{\Lip} \leq C\lVert\alpha\rVert_{\Lip}+K.\]
Therefore, asymptotics such as those in Theorem \ref{main} are invariant under Lipschitz homotopy equivalence.
\begin{defn}
  A \emph{nearly Euclidean CW complex} is a CW complex $X$ equipped with a metric constructed inductively as follows.  The 1-skeleton is a metric graph.  Once we have constructed a metric on $X^{(n-1)}$, we also fix a metric $d_i$ on $D^n$ for every $n$-cell $e_i$, such that $d_i$ is bilipschitz to the standard Euclidean metric and the attaching map $f_i:S^{n-1} \to X^{(n-1)}$ is Lipschitz with respect to the induced metric on $S^{n-1}=\partial D^n$.  Then the metric on $X^{(n)}$ is the quotient metric with respect to this gluing.
\end{defn}
In particular, notice that if $L=\max_i \Lip(f_i)$, then for points $x,y \in X^{(n-1)}$,
\[\frac{1}{L}d_{X^{(n-1)}}(x,y) \leq d_{X^{(n)}}(x,y) \leq d_{X^{(n-1)}}(x,y).\]

For example, every compact Riemannian manifold is smoothly triangulable; with any such triangulation it is a nearly Euclidean CW complex.  More generally, every simplicial complex with a simplexwise Riemannian metric is an example.
\begin{prop} \label{he-to-lhe}
  Suppose that $X$ and $Y$ are homotopy equivalent nearly Euclidean finite CW complexes.  Then they are Lipschitz homotopy equivalent.
\end{prop}
In particular, the metric we constructed on $Z$ in the proof of Lemma \ref{detailed} is nearly Euclidean, and so $Z$ is Lipschitz homotopy equivalent to, for example, any homotopy equivalent compact Riemannian manifold.

This follows immediately from the following more general statement:
\begin{prop} \label{htpy-to-lip}
  Let $X$ and $Y$ be nearly Euclidean finite CW complexes, and $A \subset X$ a subcomplex.  Let $f:X \to Y$ be a map such that $f|_A$ is Lipschitz.  Then $f$ is homotopic rel $A$ to a Lipschitz map.  Moreover, if the original map is cellular, then so is the new map.
\end{prop}
There is another useful consequence of this fact:
\begin{cor} \label{CWRiem}
  Given a finite CW complex $X$, we can always find a homotopy equivalent complex with a nearly Euclidean metric.
\end{cor}
\begin{proof}
  We use induction on skeleta.  Suppose we have constructed a complex $Y^{(k)}$ with a nearly Euclidean metric and a homotopy equivalence $f:X^{(k)} \to Y^{(k)}$.  Then for every $(k+1)$-cell of $X$ with attaching map $g:S^k \to X^{(k)}$, $f \circ g$ is homotopic to a Lipschitz map $\tilde g:S^k \to Y^{(k)}$.  Then we can attach a $(k+1)$-cell along $\tilde g$ and extend $f$ to the $(k+1)$-cell by combining $\tilde g$ and the homotopy.
\end{proof}
\begin{proof}[{Proof of Prop.~\ref{htpy-to-lip}.}]
  We start by proving a lemma:
  \begin{lem}
    $Y$ is \emph{locally Lipschitz contractible}, that is, for every $y \in Y$, there is a neighborhood $N_y \ni y$ which admits a Lipschitz deformation retraction to a point.  In particular, for every $n$, every Lipschitz map $S^n \to N_y$ extends to $D^{n+1}$ (as a Lipschitz map).
  \end{lem}
  \begin{proof}
    We build such a neighborhood by induction on skeleta, using the standard construction for a contractible neighborhood inside a CW complex.  Let $y \in Y$, and let $k$ be such that $y$ is contained in an open $k$-cell. Then we can take a ball in that $k$-cell which is Lipschitz contractible in $Y^{(k)}$.  Now suppose we have constructed a contractible neighborhood $N(n)$ of $y$ in $Y^{(n)}$, and consider an $(n+1)$-cell with attaching map $f:S^n \to Y$.  Then, thinking of the cell as the cone on $S^n$, we can add $f^{-1}(N(n)) \times [0,\epsi)$ to our neighborhood.  Doing this for every cell gives us a neighborhood in $Y^{(n+1)}$ with an obvious deformation retraction to $N(n)$, which is Lipschitz since the metric on the cell is bilipschitz to the Euclidean metric.
  \end{proof}
  We now make $f$ Lipschitz, also by induction on skeleta.  Clearly $f|_{X^{(0)}}$ is Lipschitz to begin with.  Now suppose that $f|_{X^{(k)}}$ is Lipschitz (notice that this is true with respect to the metric induced from $X^{(k+1)}$ as well as that on $X^{(k)}$) and consider a $(k+1)$-cell not in $A$ with an inclusion map $e:D^{k+1} \to X$.  Now take a triangulation of $D^{k+1}$ at a small enough scale that $f \circ e$ takes every simplex into a Lipschitz contractible neighborhood.  By induction on the skeleta of this triangulation, we deform $f \circ e$ to a Lipschitz map, while leaving it constant on $\partial D^{k+1}$.

  If $f$ is cellular, then we can construct the $(k+1)$st stage of the homotopy as a map to $Y^{(k+1)}$, rather than to $Y$.  Then the resulting map is still
  cellular.
\end{proof}

\subsection{Properties of formal spaces} \label{formal}
Finally, we need to show that the topological properties of $Z$ are also not too special to be useful.  This requires some discussion of properties of formal spaces.

One property, which follows from \cite[Proposition 3.1]{Papa}, is that a map between formal spaces which induces isomorphisms on rational cohomology is rationally invertible:
\begin{prop} \label{back-and-forth}
  If $Y$ is formal and $f:Z \to Y$ is a map between simply connected complexes inducing an isomorphism on rational cohomology, then $Z$ is formal, and there is a map $g:Y \to Z$ such that $g \circ f$ induces multiplication by $q^n$ on $H^n(Y;\mathbb{Q})$, for some $q$.
\end{prop}

Now, given $Y$, we build a rationally equivalent $Z$ which satisfies the topological hypotheses of Lemma \ref{detailed}:
\begin{prop} \label{cells}
  Let $Y$ be a simply connected space with finite-dimensional rational homology, and fix a basis for $H_n(Y;\mathbb{Q})$ for every $n$.  Then there is a CW complex $Z$ and a map $f:Z \to Y$ which induces isomorphisms on rational cohomology such that:
  \begin{enumerate}[(i)]
  \item The rational cellular chain complex of $Z$ has zero differential; that is, rational cellular chains on $Z$ are in bijection with $H_*(Z;\mathbb Q)$.
  \item The induced isomorphism $f_*:H_n(Z;\mathbb Q) \to H_n(Y;\mathbb Q)$ maps each cell to a multiple of a basis element.
  \end{enumerate}
\end{prop}
\begin{proof}
  We construct $Z$ and $f$ by induction on skeleta.  We set $Z^{(0)}=Z^{(1)}=*$.  Now suppose we have built $Z^{(n)}$ and a map $f_n:Z^{(n)} \to Y$ which induces an isomorphism on $H_k({-};\mathbb{Q})$, $k \leq n$.  By the rational relative Hurewicz theorem, the Hurewicz map induces an isomorphism
  \[\pi_{n+1}(Y,Z^{(n)}) \otimes \mathbb{Q} \to H_{n+1}(Y,Z^{(n)};\mathbb{Q}) \cong H_{n+1}(Y;\mathbb{Q}).\]
  So choose elements $\alpha_1,\ldots,\alpha_r \in \pi_{n+1}(Y,Z^{(n)})$ forming a basis for $\pi_{n+1}(Y,Z^{(n)}) \otimes \mathbb{Q}$.  We build $Z^{(n+1)}$ by attaching an $(n+1)$-cell $e_i$ along each $\partial\alpha_i$, $i=1,\ldots,r$, and extend $f_n$ to $f_{n+1}:Z^{(n+1)} \to Y$ by mapping each $e_i$ to $Y$ via a representative of $\alpha_i$.

  Since $(f_n)_*:H_n(Z^{(n)};\mathbb{Q}) \to H_n(Y;\mathbb{Q})$ is an isomorphism, by the long exact sequence of that pair, the Hurewicz image of each $\partial\alpha_i$ is zero.  Therefore the map
  \[H_{n+1}(Z^{(n+1)};\mathbb{Q}) \to H_{n+1}(Z^{(n+1)},Z^{(n)};\mathbb{Q})\]
  is an isomorphism; in other words, the cells of $Z^{(n+1)}$ form a basis for $H_{n+1}(Z^{(n+1)};\mathbb{Q})$.  Moreover, by the definition of the extension $f_{n+1}$, the map
  \[(f_{n+1})_*:\pi_{n+1}(Z^{(n+1)},Z^{(n)}) \otimes \mathbb{Q} \to \pi_{n+1}(Y,Z^{(n)}) \otimes \mathbb{Q}\]
  is an isomorphism.  But these groups are naturally isomorphic to $H_{n+1}(Z^{(n+1)};\mathbb{Q})$ and $H_{n+1}(Y;\mathbb{Q})$, respectively.  This shows that $f_{n+1}$ induces a bijection on $H_{n+1}({-};\mathbb{Q})$ as well.
  
  Once we have done this in every dimension in which $H_*(Y;\mathbb{Q}) \neq 0$, we have constructed the desired $Z$.  To satisfy condition (ii), notice that we can always pick the $\alpha_i$ to be integer multiples of the elements of our chosen basis.
\end{proof}

Now we conclude the section:
\begin{proof}[Proof of Theorem \ref{self-maps}]
  Let $Y$ be a simply connected formal compact Riemannian manifold.  Using Proposition \ref{cells}, we can find a complex $Z$ such that the cellular chain complex of $Z$ has zero differential and a rational equivalence $g:Z \to Y$.  Moreover, by Proposition \ref{back-and-forth}, there is a rational equivalence $f:Y \to Z$ such that $f \circ g$ induces multiplication by $a^n$ on $H^n(Y;\mathbb Q)$, for some $a>0$.

  By Corollary \ref{CWRiem}, we can put a nearly Euclidean metric on $Z$, and by Proposition \ref{he-to-lhe} we can assume $f$ and $g$ are Lipschitz.

  Now let $\arr_p:Z \to Z$ be a map that induces multiplication by $p^n$ on $H^n(Z;\mathbb Q)$.  By Lemma \ref{detailed} and Proposition \ref{he-to-lhe}, for any nearly Euclidean metric on $Z$, and for every $\ell$, there are $O(p^\ell\ell^{d-1})$-Lipschitz maps $\arr_{p^\ell}$ homotopic to $\arr_p^\ell$. Then the maps $f \circ \arr_{p^\ell} \circ g$ are again $O(p^\ell\ell^{d-1})$-Lipschitz and induce multiplication by $(ap^\ell)^n$ on $H^n(Y;\mathbb R)$.
\end{proof}


\section{Rational homotopy theory} \label{S:RHT}

The remainder of the paper will require machinery from rational homotopy theory.  We will give a very brief review of Sullivan's theory of minimal models, referring the reader to \cite{GrMo,FHT} for more details on the general background and \cite{PCDF,scal} for treatments geared towards quantitative topology.

Rational homotopy theory provides a way of translating the topology of simply connected spaces into algebraic language.  There are several equivalent such languages, but the main one we will use is that of differential graded algebras, as developed by Sullivan.

A \emph{(commutative) differential graded algebra}, or \emph{DGA}, is a cochain complex over a field, typically $\QQ$ or $\RR$, with a graded-commutative multiplication satisfying the graded Leibniz rule.  The prototypical examples are:
\begin{itemize}
\item The smooth forms $\Omega^*(X)$ on a smooth manifold $X$, or the simplexwise smooth forms on a simplicial complex.
\item Sullivan's \emph{minimal DGA} $\mathcal{M}_Y^*(\mathbb{F})$ for a simply connected space $Y$, which is a free graded commutative algebra generated in degree $n$ by the \emph{indecomposable} elements $V_n=\Hom(\pi_n(Y);\mathbb{F})$ and with a differential which takes elements of $V_n$ to elements of $\Lambda_{k=2}^{n-1} V_k$ and is dual to the $k$-invariants in the Postnikov tower of $Y$, $k_n \in H^{n+1}(Y_{n-1};\pi_n(Y))$.  We will write
\[\mathcal M_Y^*=\mathcal M_Y^*(\RR) \cong \Lambda_{n=2}^\infty V_n,\]
noting that this isomorphism is non-canonical.  We also write
\[\mathcal M_Y^*(n)=\Lambda_{k=2}^n V_k;\]
this is the minimal DGA of the $n$th Postnikov stage of $Y$.
\end{itemize}
There is an algebraic notion of homotopy between morphisms of DGAs which will not figure explicitly in this paper.  A \emph{quasi-isomorphism} between DGAs is a map inducing an isomorphism on cohomology.  The existence of such a map between $\mathcal A$ and $\mathcal B$ is not an equivalence relation; therefore we say that two DGAs are \emph{quasi-isomorphic} if they are connected by a zig-zag of one or more quasi-isomorphisms
\[\mathcal A \leftarrow \mathcal C_1 \rightarrow \cdots \leftarrow \mathcal C_k \rightarrow \mathcal B.\]

If $Y$ is a smooth manifold or simplicial complex, then it has a (non-unique) \emph{minimal model}, that is, a quasi-isomorphism $m_Y:\mathcal M_Y^* \to \Omega^*(Y)$ realizing the generators of the minimal DGA as differential forms.  The codomain of the minimal model may also be the algebra $\Omega^*_\flat(Y)$ of flat forms in the sense of Whitney, which are a completion of the smooth forms with respect to the $L^\infty$ norm; see \cite[\S2 and 6.1]{scal} and \cite[Ch.~IX]{GIT}.  When we want to be noncommittal about whether we are using smooth or flat forms, we write $\Omega^*_{(\flat)}(Y)$.

We will frequently leave the map $m_Y$ implicit when we speak of the \emph{rationalization} of a map $f:Y \to Z$, which is a map $\rho$ which completes the commutative square
\[\xymatrix{
  \mathcal M_Z \ar[r]^\rho \ar[d]_{m_Z} & \mathcal M_Y \ar[d]_{m_Y} \\
  \Omega^*_{(\flat)}Z \ar[r]^{f^*} & \Omega^*_{(\flat)}Y
}\]
up to homotopy.  Such a map $\rho$ always exists and is unique up to homotopy of DGA homomorphisms.

In the rest of this section, we introduce some prior results in quantitative homotopy theory as well as some information about formal spaces.

\subsection{The shadowing principle}

The main technical result of \cite{PCDF} shows a kind of coarse density of genuine maps in the space of ``formal'' rational-homotopic maps between spaces $X$ and $Y$.  That is, given a homomorphism $\mathcal{M}_Y^* \to \Omega_{(\flat)}^*(X)$, one can produce a nearby genuine map $X \to Y$ whose Lipschitz constant depends on geometric properties of the homomorphism.

To state this precisely, we first introduce some definitions.  Let $X$ and $Y$ be finite simplicial complexes or compact Riemannian manifolds such that $Y$ is simply connected and has a minimal model $m_Y:\mathcal{M}_Y^*\to\Omega_\flat^*Y$.  Fix norms on the finite-dimensional vector spaces $V_k$ of degree $k$ indecomposables of $\mathcal{M}_Y^*$; then for homomorphisms $\ph:\mathcal{M}_Y^* \to \Omega_\flat^*(X)$ we define the formal dilatation
\[\Dil(\ph)=\max_{2 \leq k \leq \dim X} \lVert\ph|_{V_k}\rVert_{\mathrm{op}}^{1/k},\]
where we use the $L^\infty$ norm on $\Omega_\flat^*(X)$.  Notice that if $f:X \to Y$ is an $L$-Lipschitz map, then $\Dil(f^*m_Y) \leq CL$, where the exact constant depends on the dimension of $X$, the minimal model on $Y$, and the norms.  Thus the dilatation is an algebraic analogue of the Lipschitz constant.

Given a formal homotopy
\[\Phi:\mathcal{M}_Y^* \to \Omega_\flat^*(X \times [0,T]),\]
we can define the dilatation $\Dil_T(\Phi)$ in a similar way.  The subscript indicates that we can always rescale $\Phi$ to spread over a smaller or larger interval, changing the dilatation; this is a formal analogue of defining separate Lipschitz constants in the time and space direction, although in the DGA world they are not so easily separable.

Now we can state some results from \cite{PCDF}.  They are stated in that paper in terms of smooth forms; for the argument that they can be adapted to flat forms, see \cite[\S6]{scal}.
\begin{thm}[{A special case of the shadowing principle, \cite[Thm.~4--1]{PCDF}}] \label{shadow}
  Let $X$ be a Riemannian manifold or simplicial complex of locally bounded geometry, and let $Y$ be a simply connected compact Riemmanian manifold or simplicial complex.  Let $\ph:\mathcal{M}_Y^* \to \Omega_{(\flat)}^*(X)$ be a homomorphism with $\Dil(\ph) \leq L$ which is formally homotopic to $f^*m_Y$ for some $f:X \to Y$.  Then $f$ is homotopic to a $g:X \to Y$ which is $C(X,Y)(L+1)$-Lipschitz and such that $g^*m_Y$ is homotopic to $\ph$ via a homotopy $\Phi$ with $\Dil_{1/L}(\Phi) \leq C(X,Y)(L+1)$.
\end{thm}
In other words, one can produce a genuine map by a small formal deformation of $\ph$.  Note that in the above result, $X$ does not have to be compact.  In fact, the constants depend only on the bounds on the local geometry of $X$.

We also present one relative version of this result:
\begin{thm}[{Cf.~\cite[Thm.~5--7]{PCDF}}] \label{relshadow}
  Let $X$ and $Y$ be finite simplicial complexes or compact Riemannian manifolds, with $Y$ simply connected.  Let $f,g:X \to Y$ be two nullhomotopic $L$-Lipschitz maps and suppose that $f^*m_Y$ and $g^*m_Y$ are formally homotopic via a homotopy $\Phi:\mathcal{M}_Y^* \to \Omega_{(\flat)}^*(X \times [0,T])$ with $\Dil_T(\Phi) \leq L$.  Then there is a $C(X,Y)(L+1)$-Lipschitz homotopy $F:X \times [0,T] \to Y$ between $f$ and $g$.
\end{thm}
It is important for this result that the maps be nullhomotopic, rather than just in the same homotopy class.  This is because we did not require our formal homotopy to be in the relative homotopy class of a genuine homotopy.  In the zero homotopy class, one can always remedy this by a small modification, but in general the minimal size of the modification may depend in an opaque way on the homotopy class.

\subsection{Formal spaces, again} \label{S:formal}

In \S\ref{formal}, we introduced formal spaces as spaces which admit self-maps of a certain type.  However, the original definition comes from rational homotopy theory, and there are a number of other equivalent definitions.  As we will use several of these definitions, we collect a number here, most found in the work of Sullivan \cite[\S12]{SulLong} and Halperin and Stasheff \cite[\S3]{HaSt}.

For any simply connected space $Y$, fix an isomorphism
\[\mathcal M_Y^* \cong \Lambda_{n=2}^\infty V_n.\]
Now we can state some equivalent definitions of formality of $Y$:
\begin{prop} \label{prop:formal}
  The following are equivalent for a simply connected space $Y$.
  \begin{enumerate}[(i)]
    \item \label{forms-qi} The algebra of forms $\Omega^*Y$ is quasi-isomorphic to $H^*(Y;\RR)$.
    \item \label{min-qi} There is a quasi-isomorphism $\mathcal M_Y^* \to H^*(Y;\RR)$.
    \item \label{H* in U_0} The cohomology of $\mathcal M_Y^*$ is a quotient of the subalgebra $W_0 \subseteq \mathcal M_Y^*$ generated by indecomposables with zero differential.  (In other words, a minimal DGA is \emph{non}-formal if and only if it has a cohomology class which has no representative in $W_0$.)
    \item \label{bigrading} There is a (non-canonical) second grading $\mathcal M_Y^*=\bigoplus_i W_i$ such that $H^*(Y;\RR)$ lives in $W_0$ and the differential with respect to the second grading has degree $-1$.  That is:
    \begin{itemize}
        \item $H^*(Y;\RR) \cong W_0/dW_1$.
        \item If $a \in W_i$ and $b \in W_j$, then $ab \in W_{i+j}$.
        \item If $a \in W_i$, then $da \in W_{i-1}$.
    \end{itemize}
    \item \label{aut} The \emph{grading automorphism} $\rho_t:H^*(Y;\RR) \to H^*(Y;\RR)$ sending every $\alpha \in H^n(Y;\RR)$ to $t^n\alpha$ is induced by an automorphism $\hat\rho_t:\mathcal M_Y^* \to \mathcal M_Y^*$.
  \end{enumerate}
\end{prop}
The arguments proving the equivalence of \ref{min-qi}--\ref{aut} do not depend on the ground field used for the DGAs.  Moreover, Sullivan \cite[Thm.~12.1]{SulLong} shows that the definitions of formality with respect to any ground field $\FF \supseteq \QQ$ are equivalent.  More generally, without reference to spaces, we can say a DGA is \emph{formal} if it is quasi-isomorphic to its cohomology ring.
\begin{proof}[Proof sketch and remarks]
  \ref{forms-qi}$\iff$\ref{min-qi}. Since the minimal model $\mathcal M_Y^* \to \Omega^*(Y)$ is a quasi-isomorphism, one is quasi-isomorphic to $H^*(Y;\RR)$ if and only if the other is.  Moreover, it is a property of minimal DGAs that if $\mathcal M_Y^*$ is quasi-isomorphic to another DGA $\mathcal A$, then there is in fact a quasi-isomorphism $\mathcal M_Y^* \to \mathcal A$.

  It follows from \ref{min-qi} that, while many rational homotopy types may have the same cohomology ring, exactly one of these is formal, and its minimal DGA can be constructed ``formally'' from the cohomology ring: at stage $k$, one adds indecomposables in degree $k$ that kill the relative $(k+1)$st cohomology of the map $\mu_{k-1}:\mathcal{M}_Y^*(k-1) \to H^*(Y;\RR)$ and extends $\mu_{k-1}$ to a map $\mu_k:\mathcal{M}_Y^*(k) \to H^*(Y;\RR)$.  This is the genesis of the term ``formal''.

  Using this construction, one inductively proves that \ref{min-qi}$\Rightarrow$\ref{H* in U_0}, by showing that for each $\mathcal{M}_Y^*(k)$ $W_0$ contains cycles representing the cohomology through dimension $k$.  For $\mathcal{M}_Y^*(2)$ this is clearly true since $\mathcal{M}_Y^*(2) \subseteq W_0$.  Now suppose we have a map $\mu_{k-1}:\mathcal{M}_Y^*(k-1) \to H^*(Y;\RR)$.  By induction, the map
  \[(\mu_{k-1})_*:H^{k+1}(\mathcal M_Y^*(k-1)) \to H^{k+1}(Y;\RR)\]
  has image in the subring generated by $H^{\leq k-1}(Y;\RR)$, and therefore we can pick preimages in $W_0$.  The rest of $H^{k+1}(\mathcal{M}_Y^*(k-1))$ is killed by differentials of elements of $V_k$.  On the other hand, the cokernel of
  \[(\mu_{k-1})_*:H^k(\mathcal M_Y^*(k-1)) \to H^k(Y;\RR) \cong H^k(\mathcal M_Y^*(k))\]
  is spanned by elements of $V_k$ with zero differential, which are also in $W_0$.  Together, these span $H^k(\mathcal M_Y^*)$.

  \ref{H* in U_0}$\Rightarrow$\ref{bigrading} is also proved by induction on dimension of indecomposables.  Suppose that we have defined the bigrading on $\mathcal M_Y^*(k-1)$.  By induction, the space of $(k+1)$-cycles in $\mathcal M_Y^*(k-1)$ splits as a direct sum of subspaces $Z_i \subseteq W_i$, since differentials of terms in $W_i$ can only cancel out with those of others in $W_i$.  Moreover, all of $\bigoplus_{i \geq 1} Z_i$ must be in the image of the differential on $V_k$.  This allows us to split $V_k$ as a direct sum of elements of various $W_i$, $i \geq 0$, so as to ensure $dW_i \subseteq W_{i-1}$.  We assign $\ker d \subseteq V_k$ to $W_0$.

  \ref{bigrading}$\Rightarrow$\ref{min-qi}.  If a bigrading as in \ref{bigrading} exists, then $\mathcal M_Y^* \to W_0/dW_1$ is a quotient map of DGAs.

  \ref{bigrading}$\Rightarrow$\ref{aut}.  We can define $\hat\rho_t(a)=t^{n+i}a$ for every $a \in W_i \cap V_n$.

  \ref{aut}$\Rightarrow$\ref{bigrading} (see also \cite[Thm.~12.7]{SulLong}).  This argument requires some information about the automorphism group of $\mathcal M_Y^*$, which is a linear algebraic subgroup of the group $\bigoplus_n \Aut(V_n)$.  Taking the ``diagonal part'' in the Iwasawa decomposition of $\hat\rho_t$, we get another automorphism, also inducing the map $\rho_t$ on cohomology, which has a basis of eigenvectors in each of the $V_n$.  An inductive argument then shows that these eigenvalues are of the form $t^{n+i}$, and setting $W_i \cap V_n$ to be the eigenspace for the eigenvalue $t^{n+i}$ gives a bigrading as in \ref{bigrading}.
\end{proof}

Now we connect these definitions to that in the previous section.  If $Y$ is a finite complex and \ref{aut} is satisfied with $\mathbb Q$ coefficients, then any family of lifts $\hat\rho_t$ can be realized by genuine maps:
\begin{thm}[{\cite[Theorem A]{PWSM}}] \label{thm:PWSM}
  Let $Y$ be a formal, simply connected finite CW complex and let $\hat\rho_t:\mathcal M_Y \to \mathcal M_Y$ be the map
  \[\hat\rho_t(w)=t^{n+i}w,\qquad w \in W_i \cap V_n,\]
  for some bigrading $\mathcal M_Y=\bigoplus_i W_i$.  Then there is an integer $t_0 \geq 1$ such that for every $z \in \mathbb{Z}$, there is a genuine map $\arr_z:Y \to Y$ whose rationalization is $\hat\rho_{zt_0}$.
\end{thm}
The same paper also gives a stronger version of Proposition \ref{back-and-forth}:
\begin{thm}[{\cite[Theorem B]{PWSM}}] \label{thm:PWSM-B}
  Let $Y$ be a formal, simply connected finite CW complex and let $\hat\rho_t:\mathcal M_Y \to \mathcal M_Y$ be the map
  \[\hat\rho_t(w)=t^{n+i}w,\qquad w \in W_i \cap V_n,\]
  for some bigrading $\mathcal M_Y=\bigoplus_i W_i$.  Suppose that $Z$ is another simply connected complex and $f:Z \to Y$ is a map inducing an isomorphism on rational cohomology.  Then for some $p$, there is a map $g:Y \to Z$ such that the rationalization of $f \circ g$ is $\hat\rho_p$.
\end{thm}
We then get the following upgraded statement of Theorem \ref{self-maps}:
\begin{thm} \label{self-maps+}
  Let $Y$ be a formal, simply connected finite CW complex whose rational homology is nontrivial in $d$ positive degrees, and let $\hat\rho_t:\mathcal M_Y \to \mathcal M_Y$ be the map
  \[\hat\rho_t(w)=t^{n+i}w,\qquad w \in W_i \cap V_n,\]
  for some bigrading $\mathcal M_Y=\bigoplus_i W_i$.  Then there is a constant $C>0$, depending on the choice of $\hat\rho_t$ as well as $Y$, such that for every homotopy class in $[Y,Y]$ whose rationalization is $\hat\rho_t$ there is a $(Ct(\log t)^{d-1}+C)$-Lipschitz representative $f:Y \to Y$.
\end{thm}
\begin{proof}
  Using Theorems \ref{thm:PWSM} and \ref{thm:PWSM-B}, we obtain topological control over the maps $f$, $g$, and $\arr_p$ used in the proof of Theorem \ref{self-maps}.  Then we see that there are $a$ and $p$ such that for every $q=ap^\ell$ there is a $C_0(q(\log q)^{d-1}+1)$-Lipschitz map $f_q:Y \to Y$ whose rationalization is $\hat\rho_q$, where $C_0$ depends on the family $\hat\rho_t$.
  
  Now suppose that $r_t:Y \to Y$ is some map whose rationalization is $\hat\rho_t$, and let $m_Y:\mathcal M_Y \to \Omega^*Y$ be a minimal model of $Y$.  Let $q=ap^\ell$ satisfy $ap^{\ell-1} \leq t<ap^\ell$.  Then the map
  \[f_q^*m_Y\hat\rho_{t/q}:\mathcal M_Y \to \Omega^*Y\]
  is algebraically homotopic to $r_t^*m_Y$.  Notice also that, with an appropriate norm on indecomposables, the operator norm of $\hat\rho_{t/q}$ is $t/q$.  Therefore, by the shadowing principle, there is an $C_1(Y)((t/q)\Lip f_q+1)$-Lipschitz map in the homotopy class of $r_t$.  Then we are done because
  \[\Lip f_q \leq C_0(pt\log(pt)^{d-1}+1). \qedhere\]
\end{proof}



\section{A finite criterion for scalability} \label{S:optconj}

In this section we prove Theorems \ref{optconj} and \ref{elliptic}.  In \cite{scal}, it was shown that the following conditions are equivalent for a finite simplicial complex or compact manifold $Y$ which is formal and simply connected:
\begin{enumerate}[(i)]
\item There is a homomorphism $H^*(Y) \to \Omega_\flat^*Y$ of differential graded algebras which sends each cohomology class to a representative of that class.  Here $\Omega_\flat^*Y$ denotes the flat forms, an algebra of not-necessarily-smooth differential forms studied by Whitney.
\item There is a constant $C(Y)$ and infinitely many $p \in \mathbb{N}$ such that there is a $C(Y)(p+1)$-Lipschitz self-map which induces multiplication by $p^n$ on $H^n(Y;\mathbb{R})$.
\item For all finite simplicial complexes $X$, nullhomotopic $L$-Lipschitz maps $X \to Y$ have $C(X,Y)(L+1)$-Lipschitz nullhomotopies.
\item For all $n<\dim Y$, homotopic $L$-Lipschitz maps $S^n \to Y$ have $C(Y)(L+1)$-Lipschitz homotopies.
\end{enumerate}
Spaces satisfying these conditions are called \emph{scalable}.  Now we will show:
\begin{thm} \label{thm:finite}
  The following condition is also equivalent to those above:
  \begin{enumerate}[(i),resume]
  \item For some $n_1,\ldots,n_N$, there is an injective $\RR$-algebra homomorphism
    \[h:H^*(Y;\mathbb R) \to \bigoplus_{i=1}^N \Lambda^*\RR^{n_i}.\]
  \end{enumerate}
  If $Y$ is a closed $n$-manifold (or, more generally, satisfies Poincar\'e duality over the reals), the following conditions are also equivalent to those above:
  \begin{enumerate}[(i),resume]
  \item[(v$' $)] There is an injective $\mathbb R$-algebra homomorphism $h:H^*(Y;\mathbb R) \to \Lambda^*\mathbb R^n$.
  \item There is a $1$-Lipschitz map $f:\mathbb R^n \to M$ of positive asymptotic degree.
  \end{enumerate}
\end{thm}

\begin{rmk}
  Condition (v$' $) can also be thought of as saying that there is an injective homomorphism $H^*(Y;\mathbb R) \to H^*(T^n;\mathbb R)$.  When is this homomorphism induced by a genuine map $T^n \to Y$ of positive degree?  A necessary condition is that the homomorphism can also be realized over the rationals.  In fact, this condition is also sufficient.  A homomorphism $H^*(Y; \mathbb Q) \to H^*(T^n;\mathbb Q)$ lifts (non-uniquely) to a homomorphism of minimal models.  By \cite[Proposition 3.1]{Papa}, after composing with a self-map $\mathcal M_Y \to \mathcal M_Y$ that induces multiplication by $p^n$ on $H^n(Y;\mathbb Q)$ for some $p$, this homomorphism becomes the rationalization of a genuine map $T^n \to Y$.
  
  This does not always happen.  For example, take the real Poincar\'e duality space
  \[Y=(S^2 \times S^2)/(x,*) \sim (*,x) \mathbin{\#} 2\mathbb CP^2 \mathbin{\#} 3 \overline{\mathbb CP^2},\]
  where $*$ is a basepoint.  The cup product $H^2(Y) \times H^2(Y) \to H^4(Y)$ is the quadratic form $\langle 2,1,1,-1,-1,-1 \rangle$, which has discriminant $-2$, and therefore is not rationally equivalent to the quadratic form induced by the cup product $H^2(T^4) \times H^2(T^4) \to H^4(T^4)$.  However, $Y$ is scalable, since over the reals, the two quadratic forms are equivalent.
  
  To get a manifold counterexample, embed $Y$ in $\mathbb R^{10}$ and let $M$ be the boundary of a thickening $W$ of this embedding.  Using Alexander duality and the Mayer--Vietoris sequence, we see that the injection $M \to W$ induces an isomorphism
  \[H^*(Y) \cong H^*(W) \xrightarrow{\simeq} H^{\leq 4}(M),\]
  and the classes in $H_{\geq 5}(M)$ are Poincar\'e duals of those coming from $W$.  This determines the rational and hence the real cohomology ring: $H^*(M;\RR) \cong H^*(Y;\RR) \times \RR\langle h^5 \rangle$, where $h^5$ is Poincar\'e dual to the fundamental class of $Y$.  Clearly this embeds in $\Lambda^*\RR^9$.  A somewhat more subtle argument shows:
  \begin{prop}
      $M$ is formal.
  \end{prop}
  \begin{proof}
    We will show that the minimal model of $M$ satisfies condition \ref{H* in U_0} of Proposition \ref{prop:formal}.  As an algebra, $H^*(M;\RR)$ is generated by six $2$-dimensional classes and the $5$-dimensional class $h^5$.  It suffices to show that these classes lie in $W_0$.  To do this, we show that the Hurewicz map $\pi_k(M) \otimes \QQ \to H_k(M;\QQ)$ is surjective for $k=2$ and $5$.  After dualizing, this means that all elements of $H^k(M;\QQ)$ are represented by indecomposables in $V_k$.

    For $k=2$, this is true by the Hurewicz theorem.  For $k=5$, we apply the relative Hurewicz theorem for the pair $(W,M)$.  By the long exact sequence of a pair, $H_i(W,M) \cong 0$ for $i \leq 5$, and so $H_6(W,M) \cong \pi_6(W,M)$.  Then from the commutative diagram of exact sequences
    \[\xymatrix{
      \pi_6(W,M) \ar[r] \ar[d]^{\cong} & \pi_5(M) \ar[r] \ar[d] & \pi_5(W) \ar[d] \\
      H_6(W,M) \ar@{->>}[r] & H_5(M) \ar[r] & 0,
    }\]
    it is evident that the Hurewicz map $\pi_5(M) \to H_5(M)$ is surjective.
  \end{proof}
  Therefore $M$ is scalable.  On the other hand, if there were a injection $H^*(M;\mathbb Q) \to H^*(T^9;\mathbb Q)$, this would induce an injection $H^*(Y) \to H^*(T^4;\mathbb Q)$, which we already showed cannot exist.
  
  Thus one can distinguish a class of ``rationally scalable'' manifolds within the larger class of scalable spaces.  It would be interesting to know what other properties distinguish these two classes.
\end{rmk}

\begin{proof}[Proof of Theorem \ref{thm:finite}]
We will prove that (i) implies (v) for all simply connected finite complexes (which is straightforward) and that (v) implies (ii) for all simply connected finite complexes (which is an application of the shadowing principle).  We will also show that for scalable closed $n$-manifolds, (v) implies (v$' $); the converse is obvious.  Then we will show that scalable closed $n$-manifolds satisfy (vi) and, conversely, (vi) implies (v$' $) for any closed $n$-manifold.

To see that (i) implies (v), choose a basis $u_1,\ldots,u_N$ for $H^*(Y;\QQ)$ and let $\omega_1,\ldots,\omega_N$ be the corresponding flat differential forms.  Then for each $i$, there is a set of positive measure on which $\omega_i \neq 0$.  Since the homomorphism $H^*(Y) \to \Omega_\flat^*(Y)$ is multiplicative almost everywhere, we can choose a point $x_i \in Y$ such that $u_j \mapsto \omega_j|_{x_i}$ is a homomorphism $h_i:H^*(Y;\RR) \to \Lambda^*\RR^{n_i}$ such that $h_i(u_i) \neq 0$.  Then we can take 
\[h=(h_1,\ldots,h_N):H^*(Y;\mathbb R) \to \bigoplus_{i=1}^N \Lambda^*\mathbb R^{n_i}.\]

If Poincar\'e duality is satisfied, then (v) implies (v$' $) since we can project $h$ to some $\Lambda^*\mathbb R^{n_i}$ under which the image of the fundamental class is nonzero.  This projection is still injective.

Now we prove that if $Y$ is a closed $n$-manifold, then (v$'$) implies (vi), in part as a warmup for the more elaborate proof that (v) implies (ii).  Since $Y$ is formal, there is a quasi-isomorphism $\ph:\mathcal M_Y^* \to H^*(Y;\RR)$.  Composing this with the homomorphism $h:H^*(Y;\RR) \to \Lambda^*\RR^n$, we get a homomorphism
\[\eta:\mathcal M_Y^* \to \Omega^*\RR^n, \qquad \eta|_x=h \circ \ph\text{ for all }x \in \RR^n,\]
whose image consists of constant forms, and such that the image of the fundamental class $\omega_{[Y]}$ is (perhaps after rescaling) the volume form.  Since $\RR^n$ is contractible and has locally bounded geometry, we can apply the shadowing principle to $\eta$ to produce a Lipschitz map $f:\RR^n \to Y$ which is related to $\eta$ by a formal homotopy
\[\Phi:\mathcal M_Y^* \to \Omega^*(\RR^n \times [0,1])\]
such that $\Dil(\Phi)<\infty$.

The pullback map on forms induced by $f$ should be thought of as looking \emph{on average} like $\eta$.  Geometrically, $f$ can be built so that $\RR^n$ is tiled (periodically or aperiodically) by homeomorphic preimages of an open dense subset of $Y$.  From a Fourier point of view, $f$ has a large constant term and the rest of the nonzero terms are at very high frequency.  Intuitively, such a map must have positive asymptotic degree.  To show this formally, we apply Stokes' theorem to the form $\Phi(\omega_{[Y]})$ on $B_R(0) \times [0,1]$, getting
\[\int_{B_R(0)} f^*d\vol=\int_{B_R(0)} \eta(\omega_{[Y]})-\int_{\partial B_R(0) \times [0,1]} \Phi(\omega_{[Y]})=\vol(B_R(0))+O(R^{n-1}).\]
We can turn $f$ into a $1$-Lipschitz map of positive asymptotic degree by rescaling.

Now we will prove that (v) implies (ii).  We prove this by constructing maps skeleton-by-skeleton.  When we extend to $n$-cells, we do it by piecing together ``almost constant'' maps from $\RR^n$, like the map $f$ in the previous two paragraphs.

Suppose $Y$ satisfies (v) (and therefore so does any complex in its rational homotopy class).  By Proposition \ref{cells} we may replace $Y$ with a rationally equivalent complex $Z$ whose rational cellular chain complex has zero differential; in other words, the cells of $Z$ form a basis for $H_*(Z;\mathbb R)$.  We equip $Z$ with a nearly Euclidean metric.  Theorem \ref{thm:PWSM-B} implies that to show that $Y$ satisfies (ii), it suffices to show that $Z$ does.

Fix a second grading $\mathcal M_Z^* \cong \bigoplus_i W_i$ as in Proposition \ref{prop:formal}\ref{bigrading}.  We get a quasi-isomorphism $\ph:\mathcal M_Z^* \to H^*(Z;\mathbb R)$ by projecting to $W_0/dW_1$, and an automorphism $\rho_t:\mathcal M_Z^* \to \mathcal M_Z^*$ which takes $w \in W_i \cap V_n$ to $t^{n+i}w$; then $\ph \circ \rho_t=t^{\deg}\ph$.  Moreover, by Theorem \ref{thm:PWSM}, for some $p>1$ there is a genuine self-map $\arr_p:Z \to Z$ whose rationalization is $\rho_p$, and in particular
induces multiplication by $p^n$ on $H^n(Z;\mathbb R)$.

We will show that $Z$ satisfies (ii) by induction on skeleta.  From (i), it follows that skeleta of scalable spaces are scalable.  Conversely, we will show that if $Z$ is an $n$-complex satisfying (v) and $Z^{(n-1)}$ is scalable, then so is $Z$.  We first show that if $Z^{(n-1)}$ is scalable, then for every $\ell>0$, the iterate $(\arr_p)^\ell|_{Z^{(n-1)}}$ is homotopic to an $O(p^\ell)$-Lipschitz map.  Moreover, for each $n$-cell, condition (v) lets us build an $O(p^\ell)$-Lipschitz map from $[0,1]^n$ to $Z$ whose degree over that cell is $p^{\ell n}$.  We construct self-maps of $Z$ satisfying (ii) by patching these together; this shows that $Z$ is also scalable.

Now we give the details.  Let $\mathbf Z$ and its submanifold $\mathbf Z^{(n-1)}$ be compact Riemannian manifolds with boundary homotopy equivalent to $Z$ and $Z^{(n-1)}$.  Let $\ph:\mathcal M_Z^* \to H^*(Z; \RR)$ be a quasi-isomorphism, which exists since $Z$ is formal, and let $i_{n-1}:Z^{(n-1)} \to Z$ be the inclusion map.  

Suppose, by induction, that $Z^{(n-1)}$ is scalable.  By condition (i), there is an injective homomorphism $H^*(Z^{(n-1)}; \RR) \to \Omega_\flat^*(\mathbf Z^{(n-1)})$ which sends each class to a representative; composing with $i_{n-1}^*\ph$ gives a map $\mathcal M_Z^* \to \Omega_\flat^*(\mathbf Z^{(n-1)})$, and by a Poincar\'e lemma argument this extends to a minimal model $m_{Z}:\mathcal M_Z^* \to \Omega_\flat^*(\mathbf Z)$ whose projection to $\Omega_\flat^*(\mathbf Z^{(n-1)})$ factors through $\ph$.  Then $(\arr_p^\ell)^*m_{Z}$ is formally homotopic to $m_{Z}\rho_p^\ell$.  By the shadowing principle \ref{shadow} and the Lipschitz homotopy equivalence between $\mathbf Z$ and $Z$, this lets us homotope $\arr_p^\ell$ to a map $\arr_{p^\ell,n-1}:Z \to Z$ which is $O(p^\ell)$-Lipschitz on $Z^{(n-1)}$.

Now we explain how to extend this map to the $n$-cells.  Let $\iota_1,\ldots,\iota_r:[0,1]^{n} \to Z$ be the inclusion maps of the $n$-cells of $Z$, and let $a_1,\ldots,a_r \in H_n(Z)$ be the corresponding homology classes.  Recall that we are assuming that there is an injective homomorphism $h:H^*(Z;\RR) \to \bigoplus_{i=1}^N \Lambda^*\RR^{n_i}$.  Since for every $i$, $\Lambda^n\RR^{n_i}$ is spanned by simple tensors, we can choose $n$-dimensional subspaces
\[V_1 \subseteq \RR^{n_{i_1}},\ldots,V_r \subset \RR^{n_{i_r}}\]
such that the projections
\[h_j=h|_{V_j}:H^*(Z; \RR) \to \Lambda^*V_j\]
collectively distinguish all elements of $H^n(Z; \RR)$.  Each $h_j|_{H^n(Z;\RR)} \in \Hom(H^n(Z;\RR),\RR)$ can be identified with a $b_j \in H_n(Z;\RR)$, and we can find coefficients $x_{ij}$ such that
\[a_i=\sum_{j=1}^r x_{ij}b_j.\]

For each $j=1,\ldots,r$ and $c \in \RR$, consider the map $\eta_{c,j}:\mathcal M_Z^* \to \Omega^*([0,1]^n)$ where, for every $x \in [0,1]^n$,
\[\eta_{c,j}(a)|_{T_xI^n}=c^k h_j \circ \ph(a),\qquad a \in \mathcal M_Z^k.\]
Applying the shadowing principle, we get an $O(c)$-Lipschitz map $f_{c,j}:[0,1]^n \to Z$ such that $f^*_{c,j}\ph$ is related to $\eta_{c,j}$ by a formal homotopy
\[\Phi:\mathcal M_Z^* \to \Omega^*([0,1]^{n} \times [0,c^{-1}])\]
satisfying $\Dil(\Phi)=O(c)$.  We can moreover assume without loss of generality that $f_{c,j}$ sends $\partial [0,1]^n$ to $Z^{(n-1)}$.  If $f_{c,j}$ does not have this property, it has a short homotopy to a map that does, by the following lemma:
\begin{lem}
    Let $L \geq 1$ and let $f:[0,L]^n \to X$ be an $1$-Lipschitz map to an $n$-dimensional CW complex $X$ with a nearly Euclidean metric.  Then there is a $C(X)$-Lipschitz homotopy
    \[H:[0,L]^n \times [0,1] \to X\]
    between $f$ and a map which sends $\partial [0,L]^n$ to $X^{(n-1)}$.
\end{lem}
Applying the lemma to $f_{c,j}$, we get a map with the desired property.   We modify $\Phi$ by appending the pullback map induced by the homotopy given by the lemma.
\begin{proof}
    Note first that it suffices to construct the homotopy on $\partial [0,L]^n \times [0,1]$.  Then it can be extended to $[0,L]^n \times [0,1]$ by pulling back along a projection map
    \[[0,L]^n \times [0,1] \to \partial [0,L]^n \times [0,1] \cup [0,L]^n \times \{0\}.\]

    Recall that we can write $X=X^{(n-1)} \cup_{\partial_i} \bigcup_i D^n$ where $\partial_i:S^n_i \to X^{(n-1)}$ are Lipschitz attaching maps, and the metric on $X$ is the quotient metric under this identification.

    The homotopy $\partial [0,L]^n \times [0,1] \to X$ will have two steps.  In the first step we homotope $f|_{\partial [0,L]^n}$ into a collar neighborhood of $Z^{(n-1)}$, namely
    \[X^{(n-1)} \cup_{\partial_i} \bigcup_i D^n \setminus B_{1/2}(0),\]
    while keeping the map $C(X)$-Lipschitz.  In the second step we retract from this collar down into $Z^{(n-1)}$ via a straight-line homotopy, which is $C(X)$-Lipschitz by definition.

    To perform the first step, we first fix a $C(n)$-Lipschitz embedding of $\Delta^n$ in $B_1(0)$ such that the interior of the image contains $B_{1/2}(0)$.  This induces an embedding $\iota_i:\Delta^n \to Z$ for every $n$-cell.  Now we triangulate $\partial [0,L]^n$ using simplices uniformly bilipschitz to the standard simplex such that the diameter of each simplex is at most $1/8$.  Then we choose the homotopy in the first step as follows:
    \begin{itemize}
        \item Vertices whose image lies in $B_{5/8}(0)$ inside the $i$th cell are homotoped linearly to the nearest vertex of $\iota_i(\Delta^n)$.  This homotopy extends on the subcomplex spanned by these vertices (which was originally mapped to $B_{3/4}(0)$ inside the $i$th $n$-cell) to a linear homotopy to a simplicial map to $\iota_i(\Delta^n)$.
        \item The homotopy is constant on the subcomplex spanned by vertices whose image lies outside $\bigcup_i B_{5/8}(0)$.  Note that the image of this subcomplex lies outside $\bigcup_i B_{1/2}(0)$.
        \item On simplices that include vertices from both subcomplexes, we extend the homotopy by interpolating linearly on the join.
    \end{itemize}
    This homotopy is again $C(X)$-Lipschitz.
\end{proof}

Since $f_{c,j}$ maps the boundary of the cube to $Z^{(n-1)}$, it makes sense to discuss the homology class of $f_{c,j}$ in $Z$, which we write $a(f_{c,j}) \in H_{n}(Z;\RR)$.  By Stokes' theorem, for any cohomology class $u \in H^{n}(Z;\RR)$,
\[u(a(f_{c,j}))=\int_{[0,1]^{n} \times \{0\}} \Phi(u)+\int_{\partial [0,1]^{n} \times [0,c^{-1}]} \Phi(u)=c^n h_j(u)+O(c^{n-1}).\]
In other words, $a(f_{c,j})=c^n b_j+O(c^{n-1})$, and therefore
\[p^{\ell n} a_i=\sum_{j=1}^r a(f_{p^\ell x_{ij}^{1/n},j}).\]
We will construct an extension of $r_{p^\ell,n-1}$ to the $i$th cell by patching together $f_{p^\ell x_{ij}^{1/n},j}$ for each $j$ together with an ``error-correcting'' map which gets rid of the $O(c^{n-1})$ error term in the homology class and a homotopy which connects the map on the boundary of the cube to $r_{p^\ell,n-1} \circ \iota_i$.

We first build the error-correcting map.  For each $i=1,\ldots,r$, fix a map
\[g_i:([0,1]^n,\partial [0,1]^n) \to (Z,Z^{(n-1)})\]
which maps to the $i$th $n$-cell with degree 1 and sends all but one of the faces of $[0,1]^n$ to a basepoint $p_0$.  Splitting $[0,1]^{n-1} \times [0,p^{-\ell}]$ into an $(n-1)$-dimensional grid of subdomains, $O(p^\ell)$ to a side, we build an $O(p^\ell)$-Lipschitz map
\[f_{\text{error}}:[0,1]^{n-1} \times [0,p^{-\ell}] \to Z\]
by mapping each subdomain to $Z$ via the appropriate $g_i$ (and mapping any leftover subdomains via a constant map to $p_0$) so that the induced homology class is the sum of the error terms of each $f_{p^\ell x_{ij}^{1/n},j}$.

Now let $g:[0,1]^{n} \to \bigvee_{r+1} [0,1]^{n}$ be an $\text{const}(r)$-Lipschitz map whose relative degree over each cube is $1$.  Then the map
\[\tilde f=(f_{p^\ell x_{i1}^{1/n},1} \vee \cdots \vee f_{p^\ell x_{ir}^{1/n},r} \vee f_{\text{error}}) \circ g:[0,1]^n \to Z.\]
is in the homotopy class of $p^{\ell n}[\iota_i] \in \pi_n(Z,Z^{(n-1)})$.  Since $\tilde f$ and $r_{p^\ell,n-1} \circ \iota_i$ are in the same class in $\pi_n(Z,Z^{(n-1)})$, their restrictions to the boundary are in the same class in $\pi_{n-1}(Z^{(n-1)})$.

Therefore, since $Z^{(n-1)}$ is scalable, using condition (iii) we can construct an $O(p^\ell)$-Lipschitz homotopy in $Z^{(n-1)}$ between $\tilde f|_{\partial [0,1]^n}$ and $\arr_{p^\ell,n-1} \circ \iota_i|_{\partial [0,1]^n}$.  We then extend $\arr_{p^\ell,n-1}|_{Z^{(n-1)}}$ to our $n$-cell in an $O(p^\ell)$-Lipschitz way using this homotopy on the outer part of the cell and $\tilde f$ on the inner part.

After we do this for every $n$-cell, we get an $O(p^\ell)$-Lipschitz map $Z \to Z$ that induces the right action on homology.  Although this map may not be homotopic to $\arr_p^\ell|_{Z}$, this is sufficient to prove condition (ii) and therefore the inductive step.

Now we argue that (vi) implies (v$' $).  One way to see this is by a direct application of Theorem \ref{balldegbound}, which shows that (vi) implies (v$' $) for any closed $n$-manifold, as well as giving a quantitative result describing how fast the degree goes to $0$ asymptotically if (v$'$) is not satisfied.

We can also use a softer, less technical argument related to Lemma \ref{topvsrel}.  Suppose there is a $1$-Lipschitz map $f:\RR^n \to Y$ of positive asymptotic degree.  Let $u_j \in H^{d_j}(Y;\RR)$ be a set of generators for the cohomology algebra of $Y$.  Suppose that the relations of the cohomology algebra are given by $R_r(u_1,\ldots,u_J)=0$, where $R_r$ is a homogeneous polynomial of graded degree $D_r$ in the free exterior algebra $\Lambda(u_1,\ldots,u_J)$.  Define forms $\omega_j \in \Omega^{d_j}(Y)$ representing the $u_j$ and $\alpha_r \in \Omega^{D_r-1}(Y)$ such that $d\alpha_r=R_r(\omega_1,\ldots,\omega_J)$.

For every $t>0$ define $f_t(x)=f(tx)$; this is a $t$-Lipschitz map.  Now we consider forms
\[\omega_{j,t}=\frac{f_t^*\omega_j}{t^{d_j}}, \qquad \alpha_{r,t}=\frac{f_t^*\alpha_r}{t^{D_r}}.\]
Since pulling back along a $t$-Lipschitz map multiplies the infinity-norm of a $k$-form by at most $t^k$, we have
\[\lVert \omega_{j,t} \rVert_\infty \leq 1, \qquad \lVert \alpha_{r,t} \rVert_\infty \leq 1/t.\]
By definition of positive asymptotic degree, there is an $\epsi>0$ and a sequence of $t \to \infty$ such that $\int_{B_1(\RR^n)} f_t^*d\vol_M \geq \epsi$.  By the Arzel\`a--Ascoli theorem, this sequence has a subsequence $t_1,t_2,\ldots \to \infty$ for which the $\omega_{j,t_k}$ converge in the flat norm; we have
\[\lim_{k \to \infty} \omega_{j,t_k}=\omega_{j,\infty} \in \Omega^{d_j}_\flat(\RR^n), \qquad \lim_{k \to \infty} \alpha_{r,t_k}=0.\]
This implies that the ring homomorphism $\Lambda(u_1,\ldots,u_J) \to \Omega^*_\flat(\RR^n)$ defined by $wu_j \mapsto \omega_{j,\infty}$ passes to a well-defined map on the quotient ring by the relations $R_r$, giving a ring homomorphism
\[\ph_\infty:H^*(M;\RR) \to \Omega^*_\flat(\RR^n).\]
Moreover, flat convergence implies that
\[\int_{B_1(\RR^n)} \ph_\infty(d\vol_M) \geq \epsi.\]
In particular, $\ph_\infty(d\vol_M)$ is nonzero on some set of positive measure.  While flat forms are not well-defined pointwise, they are well-defined up to a measure zero set, so we can choose representatives and then choose a point in this set of positive measure where these representatives actually restrict to a ring homomorphism
\[H^*(M;\RR) \to \Lambda^*\RR^n.\]
This homomorphism sends the fundamental class to a nonzero element, so by Poincar\'e duality it is injective.
%
\end{proof}

\section{Efficient nullhomotopies} \label{S:homotopies}

Now we prove Theorem \ref{main:homotopies}, which we restate here:

\begin{thm} \label{homotopies}
  Let $Y$ be a finite formal CW complex with a piecewise Riemannian metric and Lipschitz attaching maps such that $H_n(Y;\mathbb{Q})$ is nonzero for $d$ different values of $n>0$.  Then for any finite simplicial complex $X$, any nullhomotopic $L$-Lipschitz map $f:X \to Y$ is $O(L(\log L)^{d-1})$-Lipschitz nullhomotopic.
\end{thm}

We will use Theorem \ref{self-maps} to prove Theorem \ref{homotopies}.  The argument is similar to the proof of (ii)$\Rightarrow$(iii) of the main theorem of \cite{scal}.

\begin{proof}
  Let $X$ be a finite simplicial complex and $f:X \to Y$ a nullhomotopic $L$-Lipschitz map.  Fix a minimal model $m_Y:\mathcal M_Y \to \Omega^*Y$ and a family of automorphisms $\rho_t:\mathcal M_Y \to \mathcal M_Y$ which induce the grading automorphisms on cohomology sending a class $z \in H^n(Y;\RR)$ to $t^n z$.
  By Theorem \ref{thm:PWSM}, there is a $p>1$ and a self-map $\arr_p:Y \to Y$ whose rationalization is $\rho_p$.  Moreover, by Theorem \ref{self-maps+}, there is a sequence of $O(\ell^{d-1}p^\ell)$-Lipschitz maps $\arr_{p^\ell}$ homotopic to the $\ell$th iterate $\arr_p^\ell$. 
  
  We will define a nullhomotopy of $f$ by homotoping through a series of maps which are more and more ``locally organized''.  Specifically, for $1 \leq \ell \leq s=\lceil\log_pL\rceil$, we look at the map $\rho_{p^{-\ell}}$ which multiplies each degree $d$ generator by $p^{-\ell k}$ where $k \geq d$.  Thus applying the shadowing principle \ref{shadow} to the map
  \[f^*m_Y\rho_{p^{-\ell}}:\mathcal{M}_Y^* \to \Omega^*X\]
  gives a $C(X,Y)(L/p^\ell+1)$-Lipschitz map $f_\ell:X \to Y$.  Similarly, we get a $C(Y)(s^{d-1}p^\ell+1)$-Lipschitz self-map $g_\ell:Y \to Y$ homotopic to $\arr_{p^\ell}$ by applying the shadowing principle to the map
  \[\arr_{p^s}^*\rho_{p^{\ell-s}}:\mathcal M_Y^* \to \Omega^*Y.\]

  We will build a nullhomotopy of $f$ through the sequence of maps
  \[\xymatrix{
    f \ar@{-}[rd] & g_1 \circ f_1 \ar@{-}[rd] & g_2 \circ f_2 \ar@{-}[rd] & \ldots \ar@{-}[rd] & \arr_{p^s} \circ f_s \ar@{-}[r] & \text{const}. \\
    & \arr_p \circ f_1 \ar@{-}[u] & g_1 \circ \arr_p \circ f_2 \ar@{-}[u] & \ldots & g_{s-1} \circ \arr_p \circ f_s \ar@{-}[u]
  }\]
  As we go right, the \emph{length} (Lipschitz constant in the time direction) of the $\ell$th intermediate homotopy increases---it is $O(s^{d-1} p^\ell)$---while the \emph{thickness} (Lipschitz constant in the space direction) stays a constant $O(s^{d-1}L)$.  Thus all together, these homotopies can be glued into an $O(s^{d-1}p^s)$-Lipschitz nullhomotopy of $f$.

  Informally, the intermediate maps $g_\ell \circ f_\ell$ look at scale $p^\ell/L$ like thickness-$p^\ell$ ``bundles'' or ``cables'' of identical standard maps at scale $1/L$.  This structure makes them essentially as easy to nullhomotope as $L/p^\ell$-Lipschitz maps.

  We now build the aforementioned homotopies:
  \begin{lem} \label{lem:rp}
    There is a homotopy $G_\ell:Y \times [0,1] \to Y$ between $g_\ell$ and $g_{\ell-1} \circ \arr_p$ which has constant length and thickness $O(s^{d-1} p^\ell)$.
  \end{lem}
  Note that the conclusion of Lemma \ref{lem:rp} is similar to that of Lemma \ref{detailed}, but applies to a larger class of spaces.
  \begin{lem} \label{lem:fk}
    There is a homotopy $F_\ell:X \times [0,1] \to Y$ between $f_\ell$ and $\arr_p \circ f_{\ell+1}$ which has constant length and thickness $O(p^\ell)$.
  \end{lem}
  This induces homotopies of thickness $O(s^{d-1}p^s)$ and length
  $O(s^{d-1} p^\ell)$:
  \begin{itemize}
  \item $G_\ell \circ (f_\ell \times \id)$ from $g_{\ell-1} \circ r_p \circ f_\ell$ to $g_\ell \circ f_\ell$ of thickness $O(s^{d-1}p^s)$ and length $O(p^\ell)$;
  \item $g_\ell \circ F_\ell$ from $g_\ell \circ f_\ell$ to
    $g_\ell \circ r_p \circ f_{\ell+1}$ of thickness $O(s^{d-1}p^s)$ and length $O(s^{d-1}p^\ell)$.
  \end{itemize}
  It remains to build a homotopy from $r_p$ to the $C(Y)(s^{d-1}p+1)$-Lipschitz map $g_1$.  By \cite[Theorem 5--6]{PCDF}, such a homotopy $\tilde G: Y \times [0,1] \to Y$ can be chosen to have thickness $O(s^{d-1})$ and length $O(s^{d(d-1)})$.  Thus the homotopy $\tilde G \circ (f_1 \times \id)$ has thickness $O(s^{d-1}p^s)$ and length $O(s^{d(d-1)})$.
  
  Finally, the map $f_s$ is $C(X,Y)$-Lipschitz and therefore has a short homotopy to one of a finite set of nullhomotopic simplicial maps $X \to Y$.  For each map in this finite set, we can pick a fixed nullhomotopy, giving a constant bound for the Lipschitz constant of a nullhomotopy of $f_s$ and therefore a linear one for $r_{p^s} \circ f_s$.

  The lengths of these homotopies are bounded above by a geometric series which sums to $O(L(\log L)^d)$, completing the proof of the theorem modulo the two lemmas above.
\end{proof}
\begin{proof}[Proof of Lemma \ref{lem:rp}.]
  We use the fact that the maps $g_\ell$ were built using the shadowing principle.  Thus, there are formal homotopies $\Psi_i$ of length $C(X,Y)$ between $\arr_{p^s}^*m_Y\rho_{p^{s-i}}$ and $g_i^*m_Y$.  There is also a formal homotopy $\Upsilon$ between $\arr_p^*m_Y$ and $m_Y\rho_p$.  This allows us to construct the following formal homotopies:
  \begin{itemize}
  \item $\Psi_\ell$, time-reversed, between $g_\ell^*m_Y$ and $\arr_{p^s}^*m_Y\rho_{p^{s-\ell}}$, of length $C(Y)$;
  \item $\Psi_{\ell-1}\rho_p$ between $\arr_{p^s}^*m_Y\rho_{p^{s-\ell}}$ and $g_{\ell-1}^*m_Y\rho_p$, of length $C(Y)p$;
  \item and $(g_{\ell-1}^* \otimes \id)\Upsilon$ between $g_{\ell-1}^*m_Y\rho_p$ and $g_{\ell-1}^*\arr_p^*m_Y$, of length $C(Y)$.
  \end{itemize}
  Concatenating these three homotopies and applying the relative shadowing principle \ref{relshadow} to the resulting map $\mathcal{M}^*_Y \to \Omega^*(Y \times [0,1])$ rel ends, we get a linear thickness homotopy of length $O(p)$ between the two maps.
\end{proof}
\begin{proof}[Proof of Lemma \ref{lem:fk}.]
  We use the fact that the maps $f_\ell$ and $f_{\ell+1}$ were built using the shadowing principle.  Thus there are formal homotopies $\Phi_i$ of length $C(X,Y)$ between $f^*m_Y\rho_{p^{-i}}$ and $f_i$.  This allows us to construct the following formal homotopies:
  \begin{itemize}
  \item $\Phi_\ell$, time-reversed, between $f_\ell$ and $f^*m_Y\rho_{p^{-\ell}}$, of length $C(X,Y)$;
  \item $\Phi_{\ell+1}\rho_p$ between $f^*m_Y\rho_{p^{-\ell}}$ and $f_{\ell+1}^*m_Y\rho_p$, of length $C(X,Y)p$;
  \item and $(f_{\ell+1}^* \otimes \id)\Upsilon$ between $f_{\ell+1}^*m_Y\rho_p$ and $f_{\ell+1}^*\arr_p^*m_Y$, of length $C(X,Y)$.
  \end{itemize}
  Concatenating these three homotopies and applying the relative shadowing principle \ref{relshadow} to the resulting map $\mathcal{M}^*_Y \to \Omega^*(X \times [0,1])$ rel ends, we get a linear thickness homotopy of length $O(p)$ between the two maps.
\end{proof}

\section{Non-formal spaces} \label{S:NF}

In this section, we discuss the relationship between the degree and Lipschitz constants of self-maps of non-formal manifolds.

First, we note that such manifolds may have no self-maps of degree $>1$ at all.  Such manifolds are called \emph{inflexible}; examples of this phenomenon are given in \cite{ArLu,CLoh,CV,Am}.  Manifolds which have self-maps of high degree are called \emph{flexible}.

Among flexible manifolds, a distinguished class are those with positive weights.  A space $Y$ has \emph{positive weights} if its minimal model $\mathcal M_Y$ has a one-parameter family of ``rescaling'' automorphisms, i.e.\ there is a basis $\{v_i\}$ for the indecomposables and integers $n_i$ such that the map $\lambda_t:\mathcal M_Y \to \mathcal M_Y$ sending $v_i \mapsto t^{n_i}v_i$ is a DGA automorphism for any $t \in (0,\infty)$.  This can be thought of as a generalization of formality: formal spaces are distinguished by the fact that one can define rescaling automorphisms that send every cohomology class $z \mapsto t^{\dim z} z$, see \S\ref{S:formal}.

\begin{ex} \label{ex:NF}
  One non-formal manifold with positive weights is the example given in the introduction, the total space $M$ of the bundle $S^3 \to M \to S^2 \times S^2$ obtained by pulling back the Hopf fibration along a degree 1 map $S^2 \times S^2 \to S^4$.  According to \cite[p.~95]{FOT}, its minimal model is given by
  \[\mathcal M_M=\bigl(\Lambda(a_1^{(2)},a_2^{(2)},b_{11}^{(3)},b_{12}^{(3)},b_{22}^{(3)}) \mid da_i=0,db_{ij}=a_ia_j\bigr)\]
  and therefore, for any $t$, it has an automorphism which takes $a_i \mapsto ta_i$ and $b_{ij} \mapsto t^2b_{ij}$.  Now,
  \begin{align*}
      H^5(M;\mathbb Q) &\cong \langle b_{11}a_2-a_1b_{12}, b_{12}a_2-a_1b_{22} \rangle \\
      H^7(M;\mathbb Q) &\cong \langle b_{11}a_2^2-a_1a_2b_{12} \sim a_1^2b_{22}-a_1a_2b_{12} \rangle,
  \end{align*}
  and therefore this automorphism multiplies elements of $H^5(M;\mathbb Q)$ by $t^3$ and elements of $H^7(M;\mathbb Q)$ by $t^4$.
\end{ex}

A priori, automorphisms of the minimal model need not be realized by genuine maps of finite complexes.  But manifolds with positive weights have self-maps of arbitrarily high degree \cite[Theorem 3.2]{CMV}.  In fact, for any family of scaling automorphisms $\lambda_t$, there is some $t_0>0$ such that for every $z \in \mathbb Z$, $\lambda_{zt_0}$ is the rationalization of a genuine map $Y \to Y$ \cite[Theorem A]{PWSM}.

Of course, not every flexible manifold has positive weights.  For example, if $M$ is inflexible and $N$ has positive weights, then $M \times N$ is flexible, but does not have positive weights.

\subsection{Upper bounds on degree}
Having introduced the main actors, we prove Theorem \ref{main:NF}, which we restate here for convenience:
\begin{thm*}
  Let $M$ be a closed simply connected $n$-manifold which is not formal.  Then either $M$ is inflexible (has no self-maps of degree $>1$) or the maximal degree of an $L$-Lipschitz map $M \to M$ is bounded by $L^\alpha$ for some rational $\alpha<n$.
\end{thm*}
\begin{ex}
  As stated in the introduction, for the $7$-manifold $M$ described in Example \ref{ex:NF}, we get $\alpha=20/3<7$.  To see this, consider an automorphism $\rho:\mathcal M_M \to \mathcal M_M$ of the minimal model of $M$.  Such an automorphism is determined by the images
  \begin{align*}
    \rho(a_1) &= t_{11}a_1+t_{12}a_2 \\
    \rho(a_2) &= t_{21}a_1+t_{22}a_2.
  \end{align*}
  Then a computation determines that
  \[\deg \rho=\rho([M])=(\det T)^2 [M]\]
  where $T=\begin{pmatrix} t_{11}&t_{12} \\ t_{21}&t_{22} \end{pmatrix}$, and the action of $\rho$ on $H^5(M;\mathbb R)$ with respect to the given basis has matrix $(\det T)T$.  Let $\lambda_1,\lambda_2$ be the eigenvalues of $T$ with $\lvert\lambda_1\rvert \leq \lvert\lambda_2\rvert$.  Then by Lemma \ref{iterates} below, for any self-map $f:M \to M$ whose rationalization is $\rho$,
  \[\Lip f \geq \lvert \lambda_1\lambda_2^2 \rvert^{1/5} \geq \lvert\det T\rvert^{3/10}=\lvert\deg f\rvert^{3/20}.\]
\end{ex}
\begin{proof}[Proof of Theorem \ref{main:NF}]
  We prove the contrapositive.  Suppose that there is a sequence of maps $f_i:M \to M$ with strictly increasing degrees such that for every $\alpha<n$, $\deg f_i$ eventually grows faster than $(\Lip f_i)^\alpha$.  We will show that $M$ must be formal.
  
  This requires a lemma:
  \begin{lem} \label{iterates}
    Let $f:M \to M$, and suppose the induced map $f_*:H^k(M;\CC) \to H^k(M;\CC)$ has an eigenvalue $\lambda$.  Then 
    \[\Lip f \geq \lvert\lambda\rvert^{1/k}.\]
  \end{lem}
  \begin{proof}
    The eigenvalue $\lambda$ is either real or one of a conjugate pair of complex eigenvalues.  If it is real, choose a $\lVert{\cdot}\rVert_\infty$-minimizing form $\omega \in \Omega^k_\flat(M)$ among those which represent an eigenvector $a \in H^k(M;\RR)$.  Then
    \[\lvert\lambda\rvert\cdot\lVert\omega\rVert_\infty \leq \lVert f^*\omega \rVert_\infty \leq (\Lip f)^k \rVert\omega\rVert_\infty.\]
    If $\lambda$ is not real, choose an invariant two-dimensional subspace of $H^k(M;\RR)$ whose complexification contains eigenvectors for $\lambda$ and $\overline\lambda$, and within this, an $f^*/\lvert\lambda\rvert$-invariant ellipse $E$.  Let $\omega \in \Omega^k_\flat(M)$ be a $\lVert{\cdot}\rVert_\infty$-minimizing form among those representing elements of $E$.  Then once again
    \[\lvert\lambda\rvert\cdot\lVert\omega\rVert_\infty \leq \lVert f^*\omega \rVert_\infty \leq (\Lip f)^k \rVert\omega\rVert_\infty. \qedhere\]
  \end{proof}
  Now suppose $f:M \to M$ is of degree $d$ and $f_*:H^k(M;\CC) \to H^k(M;\CC)$ has some eigenvalue $\lambda$ such that $\lvert\lambda\rvert \neq d^{k/n}$.  Then either $\lvert\lambda\rvert>d^{k/n}$, or by Poincar\'e duality the induced map on $H^{n-k}(M;\CC)$ has an eigenvalue $\mu$ with $\lvert\mu\rvert>d^{\frac{n-k}{n}}$.  Therefore, by our hypotheses and Lemma \ref{iterates}, as $i \to \infty$, the absolute values of eigenvalues of $(f_i)_*:H_k(M;\CC) \to H_k(M;\CC)$ uniformly approach $(\deg f_i)^{k/n}$.  That is, for any such eigenvalue $\lambda$,
  \[k/n-C_i \leq \log_{\deg f_i} \lvert\lambda\rvert \leq k/n+C_i, \qquad \text{where }\lim_{i \to \infty} C_i=0.\]
  
  Now consider the automorphisms $\ph_i:\mathcal L_M(\CC) \to \mathcal L_M(\CC)$ induced by the $f_i$.  Here $\mathcal L_M(\CC)$ is the complexified \emph{Lie minimal model} of $M$, a free differential graded Lie algebra whose indecomposables in degree $k$ are $L_k \cong H_k(M;\CC)$, and $\ph_i|_{L_k}=(f_i)_*$.  The Lie minimal model is in many ways dual to the Sullivan minimal model; see \cite[Part IV]{FHT} for the detailed theory.  The endomorphisms of $\mathcal L_M$ form an affine variety in the vector space of graded linear maps $H_*(M;\CC) \to H_*(M;\CC)$, and the automorphisms $\Aut(\mathcal L_M(\CC))$ form a linear algebraic group which is Zariski open inside that variety.  Moreover, the Zariski closure of $\Aut(\mathcal L_M(\CC))$, which is the same as its metric closure, is contained in the endomorphism variety.
  
  We now apply the theory of linear algebraic groups, see e.g.\ \cite[\S III.10 and IV.11]{Borel}.  (A similar argument is applied to rational homotopy theory in \cite[\S2]{BMSS}.)  Choose a Borel subgroup $G \subseteq \Aut(\mathcal L_M(\CC))$; by the Lie--Kolchin theorem \cite[Ch.~III, Theorem 10.5]{Borel}, this is the subgroup of matrices which are upper triangular with respect to some basis $\mathcal B$ of $H_*(M;\CC)$.  Moreover, since elements of $\Aut(\mathcal L_M(\CC))$ preserve the grading of $H_*(M;\CC)$, we can assume that $\mathcal B$ is a graded basis.  By \cite[Ch.~IV, Theorem 11.10]{Borel}, every $\ph_i$ is conjugate to some $\ph_i' \in G$.  Moreover, by \cite[Ch.~III, Theorem 10.6]{Borel}, for every $\ph_i'$, $G$ also contains the diagonal matrix $\ph_i''$ obtained by zeroing out the off-diagonal entries of $\ph_i'$.
  
  As a vector space, $\mathcal L_M(\CC)$ is spanned by iterated Lie brackets of elements of $\mathcal B$.  Therefore, each $\ph_i''$ is diagonal on all of $\mathcal L_M(\CC)$ with respect to a basis of iterated brackets of elements of $\mathcal B$.  Moreover, if $a \in L_k$ is an eigenvector of $\ph_i''$, then $\partial a$ is also an eigenvector with the same eigenvalue.  Therefore, there are well-defined automorphisms
  \[\psi_i=(\ph_i'')^{\log_{\deg f_i} 2^n}:\mathcal L_M(\CC) \to \mathcal L_M(\CC).\]
  The sequence $\{\psi_i\}$ lies in a compact set of automorphisms and therefore has a subsequence which converges to some $\psi_\infty:\mathcal L_M(\CC) \to \mathcal L_M(\CC)$.  This $\psi_\infty$ is also diagonal with respect to $\mathcal B$ and its eigenvalues on $L_k$ have absolute value $2^k$.
  
  As with the $\ph_i''$, $\psi_\infty$ is also diagonalizable as a linear automorphism of $\mathcal L_M(\CC)$, and if $a \in L_k$ is an eigenvector of $\psi_\infty$, then so is $\partial a \in \mathcal L_M(\CC)_{k-1}$.  Therefore, if we replace each eigenvalue of $\psi_\infty$ with its absolute value, then the resulting linear map, which sends every element $a \in L_k$ to $2^k a$, is still an automorphism of $\mathcal L_M(\CC)$.  This automorphism descends to $\mathcal L_M(\QQ)$.  Since the automorphisms of a rational minimal model are the same as those of the rationalized space $M_{(0)}$, this shows that $M$ is formal.
\end{proof}

\subsection{Lower bounds on degree}
Using the techniques of \S\ref{S:lower}, we can give lower bounds on the maximal degree of an $L$-Lipschitz self-map of a manifold with positive weights that complement the upper bound of Theorem \ref{main:NF}:
\begin{thm} \label{lower:NF}
  Let $Y$ be a compact manifold with positive weights and $\rho_t:\mathcal M_Y \to \mathcal M_Y$ a scaling automorphism of its minimal model.  Let $\{z_i\}$ be a graded basis for the rational homology of $Y$ such that $\rho_t$ induces the map $z_i \mapsto t^{n_i}z_i$, and let
  \begin{align*}
      \gamma_n &= \max \{n_i/n \mid \dim z_i=n\} \\
      \alpha_n &= \max_{k \leq n} \gamma_n \\
      \alpha &= \alpha_{dim Y} \\
      d &= \#\{n \mid \gamma_n=\alpha\}.
  \end{align*}
  Then there are integers $a>0$ and $p>1$ such that for every $q=ap^\ell$ there is an $O(q^\alpha(\log q)^{d-1})$-Lipschitz map whose rationalization is $\rho_q$. 
\end{thm}
\begin{ex}
  In particular, this shows that the $7$-manifold $M$ described in Example \ref{ex:NF} has $L$-Lipschitz self-maps of degree $\sim L^{20/3}$: the bound of Theorem \ref{main:NF} is asymptotically sharp in this case.
  
  This is because for the automorphism $\rho_t:\mathcal M_M \to \mathcal M_M$ defined by
  \[a_i \mapsto ta_i, \qquad b_{ij} \mapsto t^2b_{ij},\]
  we get $n_i/\dim z_i=1/2$ when $z_i$ is any $2$-cycle, $3/5$ when $z_i$ is any $5$-cycle, and $4/7$ when $z_i$ is any 7-cycle.  Thus the maximum is only attained in dimension 5, and therefore the number $d$ defined in the statement of Theorem \ref{lower:NF} is $1$ in this case.  For a map $f:M \to M$ whose rationalization is $\rho_t$, we have $\deg f=t^4$; by Theorem \ref{lower:NF}, there are such maps which are $O(t^{3/5})$-Lipschitz.
\end{ex}
\begin{proof}[Proof of Theorem \ref{lower:NF}]
  The proof is almost identical to that of Theorem \ref{self-maps}, so we give an outline and indicate the main differences.
  
  As with Theorem \ref{self-maps}, we first reduce to the case of a nearly Euclidean cell complex $Z$ whose cells are in bijection with the basis for $H_*(Z;\mathbb Q) \cong H_*(Y;\mathbb Q)$ specified in the positive weight decomposition.  Such a complex exists by Proposition \ref{cells}.  The reduction is exactly the same as before, but requires a generalization of Proposition \ref{back-and-forth}:
  \begin{prop}[{\cite[Thm.~B]{PWSM}; see also the slightly weaker \cite[Thm.~3.4]{BCC}}] \label{back-and-forth+}
    Let $Y$ be a space with positive weights, and let $\rho_t:\mathcal M_Y \to \mathcal M_Y$ be a one-parameter family of automorphisms.  If $f:Z \to Y$ is a map between simply connected complexes inducing an isomorphism on rational cohomology, then it is a rational equivalence, and there is a map $g:Y \to Z$ and a $t \in \ZZ$ such that the rationalization of $f \circ g$ is $\rho_t$.
  \end{prop}
  Now, by \cite[Theorem A]{PWSM}, there is a $p>1$ and a map $\arr_p:Z \to Z$ whose rationalization is $\rho_p$.  As in Lemma \ref{detailed}, we construct maps $\arr_{p^\ell}$ homotopic to the iterates $\arr_p^\ell$, bounding the Lipschitz constant by induction on both $\ell$ and the dimension.  We also construct controlled homotopies $H_\ell$ from $\arr_{p^{\ell-1}} \circ \arr_p$ to $\arr_{p^\ell}$.  There are two main points on which the proof differs from that of Lemma \ref{detailed}.
  
  First, as in Lemma \ref{detailed}, we assume that $\arr_p$ has a nice geometric form.  Specifically, we assume that for every $n$-cell $e_i$, $\overline{\arr_p^{-1}(e_i)}$ is a grid inside $e$ of homothetic preimages of $e$.  Rather than $p$ to a side, this grid has $p^{n_i/n}$ subcubes to a side, where $n_i$ is the ``weight'' of the homology class $[e_i]$.  For this to make sense, $p^{n_i/\dim z_i}$ must be an integer; we can make sure this is true for every $i$ by iterating $\arr_p$ at most $(\dim Z)!$ times.
  
  The other main difference is in the Lipschitz constant estimate.  As before, we set
  \begin{align*}
    L_1 &= 2\Lip(H_\ell|_{Z^{(n-1)}}) \\
    L_2 &= 2\Lip(\arr_p)\Lip(\arr_{p^{\ell-1}}|_{Z^{(n-1)}}) \\
    L_3 &= D^{-1}\Lip(\arr_{p^{\ell-1}}),
  \end{align*}
  where $D$ is the side length of a subcube.  Then the Lipschitz constant of $\arr_{p^\ell}$ on a cell $e_i$ is bounded by
  \[p^{n_i/n}DL_3+\left(\frac{1}{2}-p^{\alpha_n}D\right)L_2+\frac{1}{2}L_1.\]
  Now the proof splits into cases.  Suppose, by induction, that \begin{align*}
    \Lip(\arr_{p^\ell}|_{Z^{(n-1)}}) &\leq C(n-1)\ell^{d_{n-1}}p^{\alpha_{n-1}\ell} \\
    \Lip(H_\ell|_{Z^{(n-1)}}) &\leq C'(n-1)\ell^{d_{n-1}}p^{\alpha_{n-1}\ell}.
  \end{align*}
  If $\alpha_{n-1}=n_i/n$, then the proof is exactly as before and \begin{align*}
    \Lip(\arr_{p^\ell}|_e) &\leq C(n)\ell^{d_{n-1}+1}p^{\alpha_{n-1}\ell} \\
    \Lip(H_\ell|_e) &\leq C'(n)\ell^{d_{n-1}+1}p^{\alpha_{n-1}\ell}
  \end{align*}
  for sufficiently large $C(n)$ and $C'(n)$ depending on $Z$ and $\arr_p$.
  
  If $\alpha_{n-1}<n_i/n$, then the estimate for the Lipschitz constant is dominated by the $L_3$ term.  After substituting the expression for the bound on $\Lip(\arr_{p^{\ell-1}})$ and summing a geometric series, we see that
  \[\Lip(\arr_{p^\ell}|_e) \leq C(n)p^{(n_i/n)\ell}\]
  for sufficiently large $C(n)$.
  
  Finally, if $\alpha_{n-1}>n_i/n$, then the estimate for the Lipschitz constant is dominated by the $L_1$ and $L_2$ terms, and therefore, for sufficiently large $C(n)$,
  \[\Lip(\arr_{p^\ell}|_e) \leq C(n)\ell^{d_{n-1}}p^{\alpha_{n-1}\ell}.\]
  Similar estimates hold for the Lipschitz constant of $H_\ell$.
  
  This gives the estimate in the theorem: the polynomial power in the Lipschitz constant is governed by the largest possible value of $n_i/n$, and the power of the polylogarithm is governed by the number of $n$ for which that value is attained.
\end{proof}

\begin{rmk}
  The methods of this theorem do not extend to manifolds without positive weights because Proposition \ref{back-and-forth+} fails.  For example, suppose $M$ is rationally equivalent to $N=P \times Q$, where $P$ has positive weights and $Q$ does not.  Then if $f:P \to P$ is a map of degree $>1$, so is $f \times \id_Q:N \to N$, and Theorem \ref{lower:NF} lets us find efficient maps homotopic to $f^\ell$ for $\ell \geq 1$.  However, this does not automatically tell us whether $M$ has self-maps of positive degree or, if it does, anything about the Lipschitz constants of these maps.  It would be interesting to either show that these properties are rationally invariant or to find examples in which they are not.
\end{rmk}

\bibliographystyle{amsalpha}
\bibliography{liphom}

\end{document}